\documentclass[12pt]{amsart}

\usepackage[english]{babel}
\usepackage{graphicx}

\usepackage[colorlinks= True, citecolor=blue]{hyperref}
\usepackage{amsthm}
\usepackage{amsfonts}
\usepackage{amssymb}
\usepackage{amsmath}
\usepackage{mathrsfs}
\usepackage{mathtools}

\usepackage{xcolor}
\usepackage{xcolor}
\usepackage{ wasysym }

\usepackage{enumerate}

\usepackage{subcaption}

\newcommand{\bzer}{\boldsymbol{0}}

\newcommand{\mr}{{\mathfrak{r}}}
\newcommand{\ml}{{\mathfrak{l}}}
\newcommand{\bd}{\boldsymbol{d}}

\newtheorem{theorem}{Theorem}[section]
\newtheorem{corollary}[theorem]{Corollary}
\newtheorem{lemma}[theorem]{Lemma}
\newtheorem{proposition}[theorem]{Proposition}

\newtheorem{claim}[theorem]{Claim}

\theoremstyle{definition}
\newtheorem{definition}[theorem]{Definition}
\newtheorem{algorithm}{Exploration}
\newtheorem{assumption}[theorem]{Assumption}

\newtheorem{remark}[theorem]{Remark}

\DeclareMathOperator{\Var}{Var}

\DeclareMathAlphabet{\mathscrbf}{OMS}{mdugm}{b}{n}

\numberwithin{equation}{section}

\newcommand{\bV}{\mathbf{V}}

\newcommand{\bB}{\mathbf{B}}
\newcommand{\bzeta}{\boldsymbol{\zeta}}
\newcommand{\GCM}{\G^{\operatorname{CM}}}
\newcommand{\GBER}{\G^{\operatorname{BER}}}
\newcommand{\GBCM}{\G^{\operatorname{BCM}}}
\newcommand{\GRIG}{\G^{\operatorname{RIG}}}

\newcommand{\Bin}{\operatorname{Bin}}
\newcommand{\Poi}{\operatorname{Poi}}

\newcommand{\bone}{{\mathbf{1}}}
\newcommand{\dotms}{\dotsm}

\newcommand{\by}{{\bf y}}
\newcommand{\bx}{{\bf x}}
\newcommand{\ORD}{\operatorname{ORD}}
\newcommand{\bw}{{\bf w}}
\newcommand{\bu}{{\bf u}}
\newcommand{\bz}{{\bf z}}
\newcommand{\bW}{{\bf W}}

\newcommand{\cT}{{\mathcal{T}}}

\newcommand{\sV}{\mathscr{V}}

\newcommand{\eps}{\varepsilon}
\newcommand{\R}{\mathbb{R}}

\newcommand{\D}{{\mathbb{D}}}

\newcommand{\Q}{\mathbb{Q}}
\newcommand{\Z}{\mathbb{Z}}
\newcommand{\N}{\mathbb{N}}

\newcommand{\F}{\mathcal{F}}

\newcommand{\E}{\mathbb{E}}

\newcommand{\cA}{\mathcal{A}}
\newcommand{\cB}{\mathcal{B}}
\newcommand{\cD}{\mathcal{D}}

\newcommand{\cU}{\mathcal{U}}
\newcommand{\cV}{\mathcal{V}}

\newcommand{\cY}{\mathcal{Y}}

\newcommand{\weakarrow}{{\overset{(d)}{\Longrightarrow}}}

\newcommand{\PR}{\mathbb{P}}

\newcommand{\G}{{\mathcal{G}}}

\newcommand{\bbeta}{\boldsymbol{\beta}}
\newcommand{\tbeta}{\widetilde{\beta}}

\newcommand{\cC}{{\mathcal{C}}}

\newcommand{\EE}{\mathcal{E}}
\newcommand{\LL}{\mathcal{L}}

\newcommand{\cS}{\mathcal{S}}

\newcommand{\fS}{\mathfrak{S}}

\newcommand{\Exp}{\textup{Exp}}
\usepackage[margin = 1in]{geometry}

	\author{David Clancy, Jr.}

 \date{\today}
\title[Near-critical BCM and RIGs]{Near-critical bipartite configuration models and their associated intersection graphs}

\usepackage{verbatim}
\usepackage{todonotes}

\begin{document}

\maketitle

\begin{abstract}
Recently, van der Hofstad, Komj\'{a}thy, and Vadon \cite{vdHKV.22} identified the critical point for the emergence of a giant connected component for the bipartite configuration model (BCM) and used this to analyze its associated random intersection graph (RIG) \cite{vdHKV.21}. We extend some of this analysis to understand the graph at, and near, criticality. In particular, we show that under certain moment conditions on the empirical degree distributions, the number of vertices in each connected component listed in decreasing order of their size converges, after appropriate re-normalization, to the excursion lengths of a certain thinned L\'{e}vy process. Our approach allows us to obtain the asymptotic triangle counts in the RIG built from the BCM. Our limits agree with the limits recently identified by Wang \cite{Wang.23} for the RIG built from the bipartite Erd\H{o}s-R\'{e}nyi random graph.
\end{abstract}
\section{Introduction}

It is well-known that many models of sparse random graph are locally tree-like. That is, around a typical vertex $v\in G$, there are no cycles with high probability. See \cite{vanderHofstad.24}, for example. This feature is not shared by many real-life social networks where one is more likely to find small cliques (i.e. isomorphic copies of some complete graph $K_p$ for some small $p$) even when the model is sparse. This has resulted in several attempts to create tractable models of sparse random graphs that exhibit such small scale structure. One example is the random intersection graph. An intersection graph is obtained by taking a bipartite graph $G^{\operatorname{BP}}$ with vertex sets $V\sqcup W$, and forming a new graph $G^{\operatorname{IG}}$ (the intersection graph) with vertices $V$ where an edge $v\sim  v'$ is present if and only if there is a vertex $w\in W$ such that $v\sim w$ and $v'\sim w$ in $G^{\operatorname{BP}}$. One can think of the vertices $v\in V$ as individuals in a population, and each $w\in W$ as some group that individual $v$ is a member of. Then edges appear in the intersection graph between two individuals $v$ and $v'$ if they belong to a common group. It is easy to see that if $\deg(w) = p$ in $G^{\operatorname{BP}}$ then the $p$ many vertices in $G^{\operatorname{IG}}$ will form an isomorphic copy of $K_p$, the complete graph on $p$ (provided that multiple edges are not allowed).

Random intersection graphs have gained some popularity recently \cite{vdHKV.21,vdHKV.22,BST.14,Wang.23,Federico.19,Behrisch.07,Bloznelis.17,Bloznelis.10,Bloznelis.13} and has led to some interest in understanding in the underlying bipartite random graphs. In particular \cite{Federico.19, Wang.23} Federico and Wang were able to establish a precise understanding of the sizes of the connected components of a near-critical bipartite Erd\H{o}s-R\'{e}nyi random graph under particular assumptions on the growth rates for the sizes of the two vertex sets. It is shown therein that the rescaled component sizes are asymptotically scalar multiples of those of the near-critical Erd\H{o}s-R\'enyi random graph described by Aldous in \cite{Aldous.97}. Our work is devoted to extending these results to the bipartite configuration model (BCM) recently analyzed in \cite{vdHKV.22} and the corresponding random intersection graph.

The BCM consists of $n$ many left-vertices ($\ml$-vertices) denoted by $\cV^{\ml}$ and $m$ many right-vertices ($\mr$-vertices) denoted by $\cV^\mr$ and respective degree sequences by $\bd^{\ml}$ and $\bd^{\mr}$. Each $\ml$-vertex $v_i^\ml$ has $d_i^\ml$ many incident half-edges and each $\mr$-vertex $w_j^\mr$ has $d_j^\mr$ many incident half-edges. The BCM is the random multigraph\footnote{In the sequel we will drop the prefix ``multi'' and simply refer to them as graphs.} obtained by uniformly at random matching the half-edges of $\ml$-vertices with a half-edge of an $\mr$-vertex. In order for this graph to be well-defined we impose the condition
\begin{equation*}
    \|\bd^\ml\|_1 =\sum_{j=1}^n d_{j}^\ml= \sum_{j=1}^m d_{j}^\mr = \|\bd^\mr\|_1,
\end{equation*}and, to avoid isolated vertices, we suppose that $d_i^\ml\ge 1$ for all $i\in[n]:=\{1,\dotsm,n\}$ and $d_j^\mr \ge 1$ for all $j\in [m]$. There is also no loss in assuming that $i\mapsto d_i^\ml$ and $j\mapsto d_j^\mr$ are non-increasing and this will be implicitly assumed throughout. We denote the resulting BCM by $\G_n(\bd^\ml,\bd^\mr)$. We will be interested in sequences of BCMs $\G_n(\bd^{\ml,(n)},\bd^{\mr,(n)})$ as the number of $\ml$- and $\mr$-vertices diverge to infinity at comparable rates. That is as $n,m = m(n)\to\infty$. In the sequel, we often omit the super-script $(n)$ when referencing the degree sequences unless it is needed for clarity. 

Recently in \cite{vdHKV.22} van der Hofstad, Komj\'athy, and Vadon established a phase transition for the BCM as both $n,m\to\infty$ and $m \sim \theta n$ for some $\theta\in ( 0,\infty)$. Let us now recall the result. To simplify statements, we say that a sequence of random variables $X_n\to X_\infty$ in the Wasserstein $p$-metric $\mathscr{W}_p$ ($p\ge 1$) if $X_n\weakarrow X_\infty$ in distribution and $\E[|X_n|^p] \to \E[|X_\infty|^p]<\infty$ (see Theorem 6.8 in \cite{Villani.09}). Let $D_n^\ml = d_{V_n^\ml}^{\ml,(n)}$ be the degree of a uniformly chosen $\ml$-vertex $V^\ml_n$ and similarly define $D_n^\mr$.
\begin{theorem}[van der Hofstad, Komj\'{a}thy, Vadon {\cite[Theorem 2.11]{vdHKV.22}}]
    Suppose both $D_n^\ml\to D_\infty^\ml$ and $D_n^\ml\to D_\infty^\mr$ in $\mathscr{W}_1$ and that $\PR(D_\infty^\ml=2)+\PR(D_\infty^\mr=2)<2$. Let $\cC_{1}^{(n)}$ denote the largest connected component of $\cB_n(\bd^{\ml},\bd^{\mr})$. Then there is some non-random constant $\xi\ge 0$ such that
    \begin{equation*}
        \frac{\#\cC_1^{(n)}}{n+m} \weakarrow \xi.
    \end{equation*}
    Moreover, $\xi>0$ if and only if
    \begin{equation*}
         \frac{ \displaystyle \E[D_\infty^\ml(D_\infty^\ml-1)]}{\displaystyle \E[D_\infty^\ml]} \frac{\displaystyle \E[D_\infty^\mr(D_\infty^\mr-1)] }{\displaystyle \E[D_\infty^\mr] }>1
    \end{equation*}
\end{theorem}

This work is concerned with the sizes of the connected components of random bipartite random graphs with a given degree sequence at criticality:
\begin{equation}\label{eqn:nuDef}
 \nu_n:=\frac{ \displaystyle \E[D_n^\ml(D_n^\ml-1)]}{\displaystyle \E[D_n^\ml]} \frac{\displaystyle \E[D_n^\mr(D_n^\mr-1)] }{\displaystyle \E[D_n^\mr] } = 1 + (\lambda+o(1)) \eps_n
\end{equation}  for a sequence $\eps_n$ depending on the particular set of assumptions on the degree distribution. Before turning to that, we recall the results of Dhara et al. in a series of works \cite{DvdHvLS.17,DvdHvLS.20} and some of the prior results of Wang \cite{Wang.23} for the bipartite Erd\H{o}s-R\'{e}nyi graph. We mention that these are extensions of works of Riordan \cite{Riordan.12} and Federico \cite{Federico.19}, respectively.

\subsection{Results for the configuration model}

We begin with some preliminary notation. For each $p\ge 1$ let 
\begin{equation*}
    \ell^p = \left\{\bx = (x_j;j\ge 1):\|\bx\|_p^p:= \sum_{j} x_j^p <\infty\right\}
\end{equation*}
and let $\ell^p_\downarrow$ denote the subset consisting of non-decreasing sequences. There is a natural map $\ORD:\ell^p\to \ell^p_\downarrow$ obtained by ordering all the elements in decreasing order. Given two elements $\bx,\by\in \ell^p$, we write 
\begin{equation*}
    \bx\bowtie \by = \ORD((x_1,y_1,x_2,y_2,\dotsm))
\end{equation*}
which is useful when describing the coloring construction of \cite{AL.98} and some subsequent scaling limits. We also let
\begin{equation*}
    \ell^{p,1} = \left\{\bz = \left((x_j,y_j);j\ge 1\right): x_j,y_j\ge 0 \textup{ and }\|\bz\|_{p,2}^p = \sum_{j=1}^\infty \left( x_j^2 + y_j^2\right)^{p/2} <\infty\right\}.
\end{equation*}
We note that it is an elementary functional analysis exercise to verify that $\ell^{p,1}$ is a cone in the separable Banach space $L^p(\N;\R^2)$, the collection of sequences of vectors whose $1$-norms are $p$th-power summable. All of this is to observe the fact that $\ell^{p,1}$ is Polish.

For a c\`adl\`ag function $f$ without negative jumps, we define an excursion interval of $f$ to be a non-empty interval $(l,r)\subset\R_+$ such that
\begin{align*}
f(l-)  &= f(r-) = \inf_{t\le r} f(t) &\textup{and}&& f(t) &> \inf_{s\le l} f(s)\qquad\forall t\in(l,r) .
\end{align*} let $\EE(f)$ denote the collection of excursion intervals of a function $f$. If possible, we define $\LL^\downarrow(f)$ as the ordered sequence of lengths of excursion intervals (i.e. $r-l$) arranged in decreasing order. 

For each $\bbeta\in \ell^3_\downarrow$, $\kappa,\rho\ge 0$ and $\lambda\in\R$ define
\begin{subequations}
\begin{align}
 \label{eqn:bJump}   &\bV^{\bbeta}(t) = \sum_{j=1}^\infty \beta_j (\bone_{[\eta_j\le t]} - \beta_j t)\qquad\textup{ where } \eta_j\sim \Exp(\beta_j)\\
 \label{eqn:bBrown} & \bB^{\kappa, \rho,\lambda}(t) = \sqrt{\kappa}B(t) + \lambda t - \frac{\rho}{2}t^2\\
 \label{eqn:bWeiner} & \bW^{\kappa,\rho,\lambda,\bbeta}(t) = \bB^{\kappa,\rho,\lambda}(t) + \bV^{\bbeta}(t)
\end{align}
\end{subequations}
where $B$ is a standard Brownian motion independent of the infinite family of exponential random variables $\eta_j$. 
If $\kappa>0$ then we will implicitly assume that $\rho>0$ as well. It is easy to see using \cite{AL.98} and some elementary real analysis that whenever $\kappa>0$ or $\bbeta\in \ell^3_\downarrow\setminus\ell^2_\downarrow$ then
    \begin{equation*}
        \bzeta^{\kappa,\rho,\lambda,\bbeta} := \LL^\downarrow (\bW^{\kappa,\rho,\lambda,\bbeta})\in \ell^2_\downarrow
    \end{equation*}
    is well-defined.

We now describe the results of \cite{DvdHvLS.17,DvdHvLS.20}. We let $\bd = \bd^{(n)} = (d_1,\dotsm, d_n)$ be a degree sequence such that $ m = \frac{1}{2} \sum_{j=1}^n d_j$ is an integer. We denote by $\GCM(\bd)$ the the configuration model with degree sequence $\bd$, that is $\GCM(\bd)$ is the graph on $[n]$ where vertex $i\in [n]$ has $d_i$ many half-edges and these half-edges are paired uniformly at random. Let $D_n = d_{U_n}$ where $U_n \sim \operatorname{Unif}([n])$ independent of the configuration. To emphasize the dependence on $n$ we will often write $\GCM_n(\bd)$ instead of $\GCM(\bd^{(n)})$ even though the dependence is on the degree sequence. Let $(\cC_n(j);j\ge 1)$ be the connected components of $\GCM_n(\bd)$ listed in decreasing order of cardinality with ties broken in some measurable but otherwise arbitrary way.

\begin{theorem}[Dhara et al. \cite{DvdHvLS.17,DvdHvLS.20}]\label{thm:DVDLSTHM}
\begin{enumerate}
    \item[\textbf{(i)}] Suppose that $\bd$ is a degree sequence such that $D_n\to D_\infty$ in $\mathscr{W}_3$ and $\PR(D_\infty = 1) \in (0,1)$. In addition, suppose that for some $\lambda\in\R$ that as $n\to\infty$
    \begin{equation*}
        \frac{\E[D_n(D_n-1)]}{\E[D_n]} = 1 + \lambda n^{-1/3} + o(n^{-1/3}).
    \end{equation*}Then
    \begin{equation*}
        \left(n^{-2/3}\#\cC_n(j);j\ge 1\right)\weakarrow \bzeta^{\kappa,\rho,\lambda,\bzer} \qquad\textup{  in  }\ell^2_\downarrow
    \end{equation*}
    where
    \begin{equation*}
        \kappa = \frac{\E[D_\infty^3]\E[D_\infty] - \E[D_\infty^2]^2}{\E[D_\infty]^2} \qquad\textup{and}\qquad \rho = \frac{\kappa}{\E[D_\infty]}.
    \end{equation*}

    \item[\textbf{(ii)}] Fix $\tau\in (3,4)$ and let $L$ be a function slowly varying at $\infty$. Let $b_n = n^{\frac{\tau-2}{\tau-1}}/L(n)$. Suppose that $\bd$ is a degree sequence such that $D_n\to D_\infty$ in $\mathscr{W}_2$ where $\PR(D_\infty = 1)>0$ and there exists a $\bbeta = (\beta_j;j\ge 1)\in \ell^3_\downarrow\setminus \ell^2_\downarrow$ such that 
    \begin{equation*}
        \frac{b_n}{n} d_j \to \beta_j \qquad\textup{ and } \qquad\lim_{K\to\infty} \limsup_{n\to\infty} \frac{b_n^3}{n^3} \sum_{j=K+1}^n \left(d_j^{(n)}\right)^3 = 0.
    \end{equation*} Further, suppose that
    \begin{equation*}
        \frac{\E[D_n(D_n-1)]}{\E[D_n]} = 1 + \lambda \frac{n}{b_n^2}  + o(n/b_n^2).
    \end{equation*}
    Then
    \begin{equation*}
        \left(b_n^{-1} \#\cC_n(j);j\ge 1\right) \weakarrow \bzeta^{0,0,\lambda,\bbeta}.
    \end{equation*}
\end{enumerate}
\end{theorem}

Theorem \ref{thm:DVDLSTHM}\textbf{(i)} above implies that the configuration model with a sufficiently regular degree distribution lies in the Erd\H{o}s-R\'{e}nyi universality class, whereas Theorem \ref{thm:DVDLSTHM}\textbf{(ii)} says that in the infinite third moment case the large degree vertices determine the behavior of the scaling limit. We will see that when it comes to sizes of the connected components of the bipartite configuration model, these two properties continue to hold. See Theorem \ref{thm:MAIN1} and Theorem \ref{thm:MAIN2} below. 

\subsection{Bipartite Erd\H{o}s-R\'{e}nyi Model}

Let $\GBER(n,m,p)$ be the bipartite Erd\H{o}s-R\'{e}nyi random graph on $n$ many $\ml$-vertices and $m$ many $\mr$-vertices. Whenever $n,m\to\infty$ and $m \sim \theta n$ for some $\theta>0$ then one can use the general results of \cite{BJR.07} to see that the largest, $\cC_n(1)$, is of order $\Theta_\PR(n)$ whenever $p \ge \frac{c}{\sqrt{nm}}$ for some $c>1$ and is of size $O_\PR(\log(n))$ whenever $p\le \frac{c}{\sqrt{nm}}$ for a $c<1$. The precise understanding of this phase transition was recently identified by Wang \cite{Wang.23} which we now recall. See also Federico's work when $m\gg n$ or $m\ll n$.

Let $\cC_n(j)$ denote the $j^\text{th}$ largest connected component of $\GBER(n,m,p)$, with ties broken arbitrarily, and let $\cC^\ml_n(j)$ (resp. $\cC_n^\mr(j)$) be the collection of $\ml$-vertices (resp. $\mr$-vertices) in $\cC_n(j)$. Observe that $\cC^\ml_n(j)$ is are the vertices of a connected component of the associated random intersection graph $\GRIG(n,m,p)$; although it may not be the $j^\text{th}$ largest component therein. 

The following as a consequence of a more general result on random intersection graphs.
\begin{theorem}[Wang \cite{Wang.23}]\label{thm:Wang1}
    Let $p = \frac{1}{\sqrt{nm}} (1+ \lambda n^{-1/3} + o(n^{-1/3}))$ and suppose that $\lim_{n\to\infty} n^{1/3}(\frac{m}{n} -\theta) = 0$ for some $\theta>0$ fixed. Then, jointly,
    \begin{align*}
       & \left(n^{-2/3} \#\cC_n^\ml(j);j\ge 1\right) \weakarrow \bzeta^{\kappa,\kappa,2\lambda, \bzer} \qquad\textup{ in }\ell^2_\downarrow,\\
      &  \left(n^{-2/3} \#\cC_n^\mr(j);j\ge 1\right)\weakarrow \sqrt{\theta} \bzeta^{\kappa,\kappa,2\lambda,\bzer}\qquad\textup{ pointwise}
    \end{align*}
    where $\kappa = 1 + \theta^{-1/2}$ and the same copy $\bzeta^{\kappa,\kappa,2\lambda,\bzer}$.
\end{theorem}

We believe that the assumption that $m/n = \theta + o(n^{-1/3})$ is a remnant of the proof strategy of Wang and can be weakened, perhaps with an additional change in the form of the critical window; however, the ultimate purpose of \cite{Wang.23} is to prove convergence of the connected components as metric measure spaces. Such a condition does not appear in our assumptions below. Another thing Wang analyzes is the number of triangles appearing in the random intersection graph. In the graph $\GRIG(n,m,p)$, let $\cT_n(j)$ denote the number of triangles. That is 
\begin{equation*}
    \cT_n(j) = \#\left\{\{u,v,w\}\subset \cC_n^\ml(j): \substack{\displaystyle u,v,w\textup{ are distinct and}\\\displaystyle u\sim v, u\sim w ,v\sim w\textup{ in }\GRIG(n,m,p)}\right\}.
\end{equation*}
\begin{proposition}[Wang \cite{Wang.23}]\label{prop:Wang2}
    Suppose the Assumptions of Theorem \ref{thm:Wang1}. Then, jointly with the convergence therein,
    \begin{equation*}
        \left(n^{-2/3}\cT_n(j);j\ge 1\right)\weakarrow c_\theta \bzeta^{\kappa, \kappa,2\lambda, \bzer}\qquad\textup{ pointwise}
    \end{equation*}
    where $c_\theta = \frac{1+3\sqrt{\theta}}{6\theta}$.
\end{proposition}

\section{Main Results for the Random Intersection Graph}

We will consider two distinct asymptotic regimes. The first is essentially the finite third moment condition analogous to \cite{Riordan.12, DvdHvLS.17} and the second is the heavy-tailed assumption found in \cite{DvdHvLS.20} recalled in Theorem \ref{thm:DVDLSTHM} above.

\subsection{Finite Third Moment Case}
\begin{assumption}\label{ass:asymptProp}
    There are $n\to\infty$ many $\ml$ vertices and $m = m(n)\to\infty$ many $\mr$ vertices with respective degree distributions $\bd^\ml = \bd^{\ml,n}$ and $\bd^\mr = \bd^{\mr,n}$ such that $\|\bd^\ml\|_1 = \|\bd^\mr\|_1$. Let $D_n^\ml$ and $D_n^\mr$ be the degree of a uniformly chosen $\ml$- and $\mr$-vertex, respectively.
  \begin{enumerate}
      \item[\textbf{(i)}] As $n\to\infty$ there is some constant $\theta\in(0,\infty)$ such that $ m = \theta n + o(n).$
       \item[\textbf{(ii)}] The maximum degrees $d_{\max}^\ml = \max_j d_j^\ml$ and $d_{\max}^\mr = \max_{j} d_j^\mr$ are both $o(n^{1/3})$ as $n\to\infty$.
      \item[\textbf{(iii)}] There are constants $\mu_p^\ml,\mu_p^\mr\in(0,\infty)$, for $p= 1,2,3$ such that
      \begin{equation*}
          \E[\left(D_n^\ml\right)^p] \longrightarrow \mu_p^\ml\qquad\textup{and}\qquad \E[\left(D_n^\mr\right)^p] \longrightarrow \mu_p^\mr
      \end{equation*}
      \item[\textbf{(iv)}] As $n\to\infty$ there is some constant $\lambda\in\R$ such that
      \begin{equation*}
          \nu_n:= \frac{\E[D_n^\ml(D_n^\ml-1)]}{\E[D_n^\ml]} \frac{\E[D_n^\mr(D_n^\mr-1)]}{\E[D_n^\mr]} = 1 + \lambda n^{-1/3} + o(n^{-1/3}).
      \end{equation*}
      \item[\textbf{(v)}] It holds that
  \begin{equation*}
    \mu_3^\ml\mu_1^\ml + \mu_3^\mr\mu_1^\mr >  (\mu_2^\ml)^2  + (\mu_2^\mr)^2.
  \end{equation*}
  \end{enumerate}
\end{assumption}

\begin{remark}
    While not entirely necessary for our subsequent analysis, we remark that Assumptions \ref{ass:asymptProp}\textbf{(ii)}--\textbf{(iii)} are simultaneously implied by assuming that $D_n^\ml\to D_\infty^\ml$ and $D_n^\mr\to D_\infty^\mr$ in the Wasserstein space $\mathscr{W}_3$. If we suppose that $D_n^\ml\to D_\infty^\ml$, $D_n^\mr\to D_\infty^\mr$ in $\mathscr{W}_3$ where $\PR(D_\infty^\ml = 2)+ \PR(D_\infty^\mr=2)<2$ and Assumption \ref{ass:asymptProp}\textbf{(iv)}, then Assumption\textbf{(v)} holds automatically.
\end{remark}

The main theorem in this setting is the following.
\begin{theorem}\label{thm:MAIN1}
    Suppose that $\bd^\ml,\bd^\mr$ satisfy Assumption \ref{ass:asymptProp}. Let $(\cC_n(j);j\ge 1)$ be the connected components of $\GBCM_n(\bd^\ml,\bd^\mr)$ listed in decreasing order of size. Then
    \begin{equation*}
        \left(n^{-2/3}(\#\cC_n^\ml(j), \#\cC_n^\mr(j));j\ge 1\right)\weakarrow \left(\nu_\infty^\mr, 1\right) \bzeta^{\kappa,\rho,\lambda \nu_\infty^\mr, \bzer}\qquad  \textup{ in }\ell^{2,1}.
    \end{equation*}
    where \begin{subequations}
\begin{align}
     \label{eqn:sigmamlmrDef}\kappa^\ml&= \frac{\mu_{3}^\ml\mu_1^\ml - (\mu_2^\ml)^2}{(\mu_1^\ml)^2} & \kappa^\mr&= \frac{\mu_{3}^\mr\mu_1^\mr - (\mu_2^\mr)^2}{(\mu_1^\mr)^2}\\
    \label{eqn:rhopmlmrDef}\rho^\ml &= \frac{\mu_{3}^\ml\mu_1^\ml - (\mu_2^\ml)^2}{(\mu_1^\ml)^3} & \rho^\mr&= \frac{\mu_{3}^\mr\mu_1^\mr - (\mu_2^\mr)^2}{(\mu_1^\mr)^3}\\
      \label{eqn:nuinftyDef}  \nu_\infty^\ml &= \frac{\mu_2^\ml-\mu_1^\ml}{\mu_1^\ml} & \nu_\infty^\mr &= \frac{\mu_2^\mr-\mu_1^\mr}{\mu_1^\mr} \\
      \label{eqn:kapparhodefforlimit}\kappa&= (\nu_\infty^\mr)^3\kappa^\ml + \kappa^\mr &  \rho&= (\nu_\infty^\mr)^3\rho^\ml + \rho^\mr/\theta.
\end{align}
\end{subequations} 
\end{theorem}

Let us compare this result to Wang's result recalled in Theorem \ref{thm:Wang1} above. If $D_n^\ml,D_n^\mr$ are the degree of a uniformly chosen vertex in $\GBER(n,m,p)$ where $p = (1+\lambda n^{-1/3} + o(n^{-1/3}))/\sqrt{nm}$. If $m/n\to \theta$ then as $n\to\infty$
\begin{equation*}
    D_n^\ml \sim \Bin(m-1,p) \to \Poi(\theta^{1/2})\qquad\textup{ and }D_n^\mr\sim \Bin(n-1,p) \to \Poi(\theta^{-1/2})
\end{equation*}
where the convergence holds in the $\mathscr{W}_p$ for any $p\in[1,\infty)$. Here $\Poi(\mu)$ is a Poisson random variable with mean $\mu$. 
Note that $\E[X(X-1)]/\E[X] = \E[X^\ast-1]$ where $X^\ast$ is a size-biased sample of a random variable $X$, and that a size-biased sample of $X\sim\Bin(N,p)$ satisfies $X^\ast-1 \sim \Bin(N-1,p)$ where $\Bin(0,p)\equiv 0$. Therefore,
\begin{align*}
    &\frac{\E[D_n^\ml(D_n^\ml-1)]}{\E[D_n^\ml]}  \frac{\E[D_n^\mr(D_n^\mr-1)]}{\E[D_n^\mr]} = \frac{(m-2)(n-2)}{nm}\left(1+\lambda n^{-1/3}+o(n^{-1/3})\right)^2\\
    &= 1 + 2\lambda n^{-1/3} + o(n^{-1/3}).
\end{align*}
It is easy to compute
\begin{align*}
    \kappa^\ml &= \theta^{1/2}, & \kappa^{\mr} &= \theta^{-1/2}, & \rho^\ml &=\rho^{\mr} = 1,&
    \nu^\ml_\infty &= \theta^{1/2}, & &\textup{and}& \nu_\infty^\mr&= \theta^{-1/2}.
\end{align*}
Therefore $\kappa = \theta^{-1} + \theta^{-1/2}$ and $\rho = \theta^{-3/2} + \theta^{-1}$. 
Let us observe that 
\begin{equation*}
   (\bB^{\kappa,\rho,2\lambda \theta^{-1/2}}(t);t\ge 0 )\overset{d}{=}  \left(\bB^{1+\theta^{-1/2}, 1+\theta^{-1/2}, 2\lambda}(\theta^{-1/2}t)\right)_{t\ge 0}
\end{equation*}
and so 
\begin{equation}\label{eqn:scalingforwangandme}
    \bzeta^{\kappa,\rho,2\lambda \theta^{-1/2},\bzer} \overset{d}{=} \sqrt{\theta}\bzeta^{1+\theta^{-1/2}, 1+\theta^{-1/2}, 2\lambda,\bzer}.
\end{equation}
Recalling the limit in Theorem \ref{thm:Wang1}, we see that the corresponding (non-rigorous) application of Theorem \ref{thm:MAIN1} to the graph $\GBER(n,m,p)$ gives Wang's result.

\begin{remark}
    There is a subtle technicality that we should mention that goes beyond the above informal connection. Our proof technique requires that $d_{j}^\ml,d_j^\mr\ge 1$ while in $\GBER_n(n,m,p)$ a strictly positive proportion of all vertices have degree $0$.
\end{remark}

Next we turn to triangle counts. As we did leading to Proposition \ref{prop:Wang2}, we let $\cT_n(j)$ be the number of triangles in the $j^\text{th}$ largest connected component of the random intersection graph $\GRIG_n(\bd^\ml,\bd^\mr)$ constructed from the bipartite configuration model $\GBCM_n(\bd^\ml,\bd^\mr)$.
\begin{proposition}\label{prop:finThirdMomentTriang}
    Suppose that $\bd^\ml,\bd^\mr$ satisfy Assumption \ref{ass:asymptProp} and, in addition, suppose that $D_n^\mr\to D_\infty^\mr$ in $\mathscr{W}_4$ for some limiting random variable $D_\infty^\mr$. Then, jointly with the convergence in Theorem \ref{thm:MAIN1},
    \begin{equation*}
        \left(n^{-2/3} \#\cT_n(j);j\ge \right)\weakarrow \frac{1}{\E[D_\infty^\mr]} \E\left[D_\infty^\mr\binom{D_\infty^\mr}{3}\right] \bzeta^{\kappa,\rho,\lambda\nu_\infty^\mr,\bzer}\qquad\textup{ pointwise}
    \end{equation*}
    for the same limiting vector $\bzeta^{\kappa,\rho,\lambda\nu_\infty^\mr,\bzer}$.
\end{proposition}
Drawing the connection with Proposition \ref{prop:Wang2} above, we see that if $D_n^\ml$ and $D_n^\mr$ are the degrees of a uniformly chosen $\ml$-vertex and $\mr$-vertex in $\GBER(n,m,p)$ then 
\begin{equation*}
    \frac{1}{\E[D_\infty^\mr]} \E\left[D_\infty^\mr\binom{D_\infty^\mr}{3}\right] = \E\left[\binom{X^*}{3}\right] =\frac{1}{6} \E[(X+1)X(X-1)],
\end{equation*}
where $X\sim \Poi(\theta^{-1/2})$ and $X^*$ is size-biased sample of $X$, i.e. $X^*-1\sim \Poi(\theta^{-1/2})$. One can easily check that
\begin{equation*}
    \frac{1}{6}\E[(X+1)X(X-1)] = c_\theta/\sqrt{\theta}
\end{equation*}where $c_\theta$ is as in Proposition \ref{prop:Wang2}. Therefore, there is a agreement with the two triangle counts using \eqref{eqn:scalingforwangandme}.

\subsubsection{Connection to random hypergraphs}

Consider the following model for a random $k$-uniform hypergraph $H_{k,n}(\bd)$ on $n$ many vertices. The hypergraph consists of $n$ vertices and hyperedges $e$ which consist of cardinality $k$ (multi)sets of $[n]$. Each vertex $i\in[n]$ is a member of exactly $d_i$ many hyperedges for some degree sequence $\bd$ such that $\sum_{j}d_j\in k\Z$. To construct $H_{k,n}(\bd)$ we label the $\sum_{j} d_j$ many stubs by $s_1,s_2,\dotsm, s_{kp}$ and uniformly at random partition $[kp]$ into subsets of size $k$. Each subset of size $k$ then becomes a hyperedge. This forms the configuration model for hypergraphs \cite{Chodrow.20}. 

One can also easily see that this is, in some sense, equivalent to the random intersection graph constructed from the bipartite configuration model consisting of $n$ many $\ml$-vertices with degree sequence $\bd$ and $\frac{1}{k}\sum_{j} d_j$ many $\mr$ vertices of degree exactly $k$. Therefore, using Theorem \ref{thm:MAIN1} we can easily obtain the asymptotics for the sizes of the connected components of the hypergraph configuration model at criticality. 

\begin{proposition}
    Suppose that $\bd$ is a degree sequence for a $k$-uniform hypergraph $H_{k,n}(\bd)$ on $n$ vertices and let $D_n$ be the degree of a uniformly selected vertex. Suppose that there is some random variable $D_\infty$ such that 
    \begin{equation*}
        D_n\to D_\infty\qquad\textup{ in }\mathscr{W}_3
    \end{equation*}
    where $\E[D_\infty^3]\E[D_\infty] > \E[D_\infty^2]^2.$
    Further suppose that there is some $\lambda\in \R$ such that 
    \begin{equation*}
        \frac{\E[D_n(D_n-1)]}{\E[D_n]}  = \frac{1+ \lambda n^{-1/3} + o(n^{-1/3})}{k-1}.
    \end{equation*}
    Then
    \begin{align*}
        \left(n^{-2/3}\#\cC_n(j);j\ge 1\right)\weakarrow \bzeta^{\kappa,\rho,(k-1)\lambda,\bzer}
    \end{align*}
    where
    \begin{equation*}
    \kappa = (k-1)^3\frac{\E[D^3_\infty]\E[D_\infty]-\E[D_\infty^2]}{\E[D_\infty]^2} \qquad \textup{and}\qquad \rho = \frac{\kappa}{\E[D_\infty]}.
    \end{equation*}
\end{proposition}

This should be compared with the results of Bollob\'{a}s and Riordan \cite{BR.12a} who analyzed the connected components of the $k$-uniform random hypergraph $H_k(n,p)$ on $n$ vertices where each hyperedge of size $k$ is included independently with probability $p$.

\subsection{Infinite Third Moment Case}

The assumptions in the infinite third moment case are quite analogous to those in \cite{DvdHvLS.20}.

\begin{assumption}\label{ass:infintieThird} 
There are $n\to\infty$ many $\ml$ vertices and $m = m(n)\to\infty$ many $\mr$ vertices with respective degree distributions $\bd^\ml = \bd^{\ml,n}$ and $\bd^\mr = \bd^{\mr,n}$ such that $\|\bd^\ml\|_1 = \|\bd^\mr\|_1$. Let $D_n^\ml$ and $D_n^\mr$ be the degree of a uniformly chosen $\ml$- and $\mr$-vertex, respectively. Order $d^{\ml}_1\ge d^\ml_2\ge\dotsm$ and $d^\mr_1\ge d^\mr_2\ge\dotsm$. Define for each $k\ge 1$
\begin{align*}
    a_k &= k^{1/(\tau-1)}L(k) & b_k&= \frac{k^{(\tau-2)/(\tau-1)} }{L(k)} & c_k&= \frac{k^{(\tau-3)/(\tau-1)}}{L(k)^2}
\end{align*}
for some fixed $\tau\in(3,4)$ and a function $L$ slowly varying at $+\infty$.
  \begin{enumerate}
      \item[\textbf{(i)}] As $n\to\infty$ there is some constant $\theta\in(0,\infty)$ such that 
      $
          m = \theta n + o(n).
     $
      \item[\textbf{(ii)}] As $n\to\infty$ and for each $i\ge 1$ fixed,
      \begin{equation*}
          a_n^{-1} d_i^\ml\to \beta_i^\ml \qquad\textup{and}\qquad a_m^{-1} d_i^\mr \to \beta_i^\mr
      \end{equation*} where $\bbeta^\ml = (\beta_1^\ml,\beta_2^\ml,\dotsm), \bbeta^\mr = (\beta_1,\beta_2,\dotsm) \in \ell^3_\downarrow$.
      \item[\textbf{(iii)}] There are constants $\mu_p^\ml,\mu_p^\mr\in(0,\infty)$, for $p= 1,2$ such that
      \begin{equation*}
          \E[\left(D_n^\ml\right)^p] \longrightarrow \mu_p^\ml\qquad\textup{and}\qquad \E[\left(D_n^\mr\right)^p] \longrightarrow \mu_p^\mr
      \end{equation*}
      and
      \begin{equation*}
          \lim_{k\to\infty} \limsup_{n\to\infty} a_n^{-3}\sum_{j=k+1}^n (d_j^\ml)^3 + a_n^{-3}\sum_{j=k+1}^m (d_j^\mr)^3  = 0.
      \end{equation*}
      \item[\textbf{(iv)}] As $n\to\infty$ there is some constant $\lambda\in\R$ such that
      \begin{equation*}
          \nu_n:= \frac{\E[D_n^\ml(D_n^\ml-1)]}{\E[D_n^\ml]} \frac{\E[D_n^\mr(D_n^\mr-1)]}{\E[D_n^\mr]} = 1 + \lambda c_n^{-1} + o(c^{-1}_n).
      \end{equation*}
      \item[\textbf{(v)}] $\bbeta^\ml + \bbeta^\mr\notin \ell^2_\downarrow.$
  \end{enumerate}
\end{assumption}

Under Assumption \ref{ass:infintieThird}, we have our second main theorem.
\begin{theorem}\label{thm:MAIN2} Suppose that $\bd^\ml$ and $\bd^\mr$ satisfy Assumption \ref{ass:infintieThird} and set
    \begin{align*}
        \bbeta = \left(\nu^\mr_\infty \bbeta^\ml/\mu_1^\ml\right) \bowtie \left(\theta^{1/(\tau-1)}\bbeta^\mr/\mu_1^\ml\right)\in \ell^3_\downarrow\setminus \ell^2_\downarrow.
    \end{align*}
    Then
    \begin{equation*}
         \left(b_n^{-1}(\#\cC_n^\ml(j), \#\cC_n^\mr(j));j\ge 1\right)\weakarrow \left(\nu_\infty^\mr, 1\right) \bzeta^{{0,0, \lambda\nu_\infty^\mr/\mu_1^\ml, \bbeta}}\qquad  \textup{ in }\ell^{2,1}
    \end{equation*}
    where $\nu_\infty^\ml$ and $\nu_\infty^\mr$ are defined using \eqref{eqn:nuinftyDef}.
\end{theorem}

Theorem \ref{thm:MAIN2} is a direct generalization of Theorem \ref{thm:MAIN1} to the case where the vertices have infinite third moment. Looking back at the triangle counts result in Proposition \ref{prop:finThirdMomentTriang}, we needed to impose stronger conditions on the degree distributions in order to obtain some results on the triangle counts. It turns out that this is \textit{not} the case in the infinite third moment case. 

To state our result, we introduce some new notation similar to \cite{Clancy.24}. Given a c\`adl\`ag function $f$ such that $\LL^\downarrow(f) \in \ell^\infty_\downarrow$ and we can order the excursions intervals
\begin{align*}
    \EE(f) = \{(l_1,r_1),(l_2,r_2),\dotsm\}
\end{align*}
where for $a<b$ we have $r_a-l_a> r_b - l_b$ or $r_a -l_a = r_b-l_b$ but $l_a<l_b$. That is, we first order the excursions by their lengths and break ties in order of appearance. Given
a non-decreasing c\`adl\`ag function $g$, set
\begin{equation*}
    \Gamma_\infty(f,g) = \left(g(r_j)-g(l_j-);j\ge 0\right)\in \R^\infty_+.
\end{equation*}
\begin{proposition}\label{prop:infTriangl}
    Suppose that $\bd^\ml,\bd^\mr$ satisfy Assumption \ref{ass:infintieThird}. Let $\cT_n(j)$ denote the triangle count in the $j^\text{th}$ largest connected component of $\GRIG_n(\bd^\ml,\bd^\mr).$ Then
    \begin{align*}
        \left(a_n^{-3} \cT_n(j);j\ge 1\right) \weakarrow \Gamma_\infty\left(X,T\right)
    \end{align*}
    where
    \begin{align*}
        X(t) &= \bV^{\beta^\ml}\left(\frac{\nu_\infty^\mr}{\mu_1^\ml}t\right) + {\nu_\infty^\ml \theta^{\frac{1}{\tau-1}}} \sum_{j=1}^\infty \beta_j^\mr \left(\bone_{[\eta_j^\mr\le \theta^{-\frac{\tau-2}{\tau-1}}t/\mu_1^\mr]} 
 - \theta^{-\frac{\tau-2}{\tau-1}}\frac{\beta_j^\mr}{\mu_1^\mr}t\right)+\lambda t\\
        T(t)&= \frac{\theta^{-\frac{3}{\tau-1}}}{6}\sum_{j=1}^\infty  \left(\beta^\mr_j\right)^3 \bone_{[\eta_j^\mr\le \theta^{-\frac{\tau-2}{\tau-1}}t/\mu_1^\mr]}
    \end{align*}
    for an independent collection of exponential random variables $\eta_j^\mr\sim \Exp(\beta^\mr_j)$ and an independent process $\bV^{\bbeta^\ml}.$
\end{proposition}
In words, the asymptotic triangle counts in the infinite third moment case are determined \textit{solely} by the large degree vertices. This should not be too surprising, each $\mr$ vertex $w_j^\mr$ will have $d_j^\mr$ many $\ml$-vertices as neighbors (counted with multiplicity). If we suppose for the moment that the surplus count in BCM $\GBCM_n(\bd^\ml,\bd^\mr)$ is small (see Lemma \ref{lem:surpluscountlim}) then each $\mr$-vertex $w_j^\mr$ accounts for $\binom{d_j^\ml}{3}\sim \frac{a_m^3}{6} (\beta_j^\mr)^3\sim a_n^3 \theta^{-3/(\tau-1)} (\beta_j^\mr)^3/6$ many triangles in $\GRIG_n(\bd^\ml,\bd^\mr).$ It turns out that this heuristic is essentially why Proposition \ref{prop:infTriangl} is true.

\subsection{Overview}

Our main approach to analyzing the sizes of the connected components at criticality is the standard approach pioneered by Aldous \cite{Aldous.97}. We encode the depth-first exploration of the random graph $\GBCM_n(\bd^\ml,\bd^\mr)$ using a handful of stochastic processes analogous to the recent exploration of Wang \cite{Wang.23}; however, our analysis of the walk is quite different. Once we have the general form of the exploration, we perturb it slightly and prove weak convergence of this perturbation. This perturbation is a negative push each time the process is at its running minimum and, therefore, it does not drastically affect the encoding. This perturbation does simplify some of our subsequent analysis. 

Once we have the encoding of the exploration of the BCM, we need to establish convergence in $\ell^{2,1}$ by extending some results on the so-called susceptibility function of the subcritical CM due to Janson \cite{Janson.09} to the subcritical BCM. This makes establishing convergence of the component sizes an easy application of results in \cite{Clancy.24}. Triangle counts follow similarly; however, since we prove convergence in the product topology we do not need to analyze a weighted susceptibility function for the BCM.

\section{Exploration and Weak Convergence}

Let us explain in words how we explore the graph. At each time $k$, we have a partition of all the $\ml$-half-edges using three sets, $\cA^\ml(k)\sqcup \cD^\ml(k)\sqcup \cU^\ml(k)$, along with a collection $\cS^\mr(k)$ which consists of the yet-to-be explored $\mr$-vertices. We call $\cA^\ml$ the \textit{active} half-edges, $\cD^\ml$ the \textit{dead} half-edges and $\cU^\ml$ the \textit{unexplored} half-edges. Initially, all $\ml$-half-edges are unexplored. At each step of the algorithm, we will select precisely one $\mr$-vertex and pair all its half-edges to $\ml$-half-edges and upon pairing those $\ml$-half-edges we \textit{kill} them and add them to the collection of dead half-edges. Whenever the stack $\cA^\ml$ is non-empty, one of these $\ml$-half-edges that are paired will be from the stack. Moreover, if $v^\ml$ is an $\ml$-vertex with a newly paired $\ml$-half-edge at a particular step $k$, then all its other half-edges will be activated (if they were not already active). This process continues until all $\mr$-vertices have been explored.
\begin{algorithm}[DFS of Bipartite Configuration Model]\,\label{exploration:DFS}\\
\textbf{Initialization:} Set $k = 0$. Set $\cA^\ml(0) = \cD^\ml(0)=\emptyset$ and $\cU^\ml(0)$ as all $\ml$-half-edges. Let $\cS^\mr(0)$ be all $\mr$-vertices. \\
\textbf{while} $\cS^\mr(k-1)\neq \emptyset$ \textbf{do:} \\
\textbf{if }$\cA^\ml(k-1) \neq \emptyset$ \textbf{do:} 
\begin{enumerate}
    \item[a1)] Set $e\in \cA^\ml(k)$ as the least $\ml$-half-edge included in $\cA^\ml(k-1)$.
    \item[a2)] Find the $\mr$-half-edge $f$ that is paired with $e$ and call $w = w_{(k)}^\mr$ the $\mr$-vertex incident to $b$. Necessarily $w\in \cS^\mr(k-1)$. Label the $r:=d_{w}^\mr-1$ remaining half-edges of $w$ as $f_{w1}, f_{w2},\dotsm, f_{wr}.$
    \item[a3)] For each $\mr$-half-edge $f_{wj}$ label the $\ml$-half-edge $e_{wj}\in \cA^\ml(k-1)\cup \cU^\ml(k-1)\setminus \{e\}$ paired to $f_{wj}$. Label the $\ml$-vertex incident to $e_{wj}$ as $v_{wj}$. Note that $v_{w1},\dotsm, v_{wr}$ need not be distinct and that for each $j$ a repetition of $v_{wj}$ with some $v_{wi}$ for $i<j$ \textit{or} $e_{wj}\in \cA^\ml(k-1)$ corresponds to exactly a single cycle constructed at step $k$.
    \item[a4)] Let $\cB(k) = \{j = 1,2,\dotsm,r : e_{wj}\in \cA^\ml(k-1)\textup{ or }\exists i< j\textup{ s.t. }v_{wj} = v_{wi}\}$. 
    \item[a5)] Set $\cS^\mr(k)=\cS^\mr(k-1)\setminus \{w\}$, $\cD^\ml(k) = \cD^\ml(k-1)\cup \{a,e_{w1},\dotsm, e_{wr}\}$, and $\cA^\ml(k)=\cA^{\ml}(k-1)\cup \{e\in \cU^\ml(k-1): e\textup{ incident to some }v_{wj}\}\setminus \{a,e_{w1},\dotsm,e_{wr}\}$. Finally, set $\cU^\ml(k)= \cU^\ml(k-1)\setminus (\cA^\ml(k)\cup \cD^\ml(k))$. Declare the newly included $\ml$-half-edges in $\cA^\ml(k)$ to be smaller than all other half-edge in $\cA^\ml(k-1)$ and give a comparison between these newly added half-edges in some arbitrary way. Increment $k$ to $k+1$.
\end{enumerate}
\textbf{else do:}
\begin{enumerate}
    \item[b1)] Uniformly at random pick an $\mr$-half-edge $f$ and call $w = w_{(k)}^\mr$ the incident $\mr$-vertex. Necessarily $w\in \cS^\mr(k-1)$. Label the $r:=d_{w}^\mr$ many $\mr$-half-edges incident to $w$ as $f = :f_{w1}, f_{w2},\dotms, f_{wr}$.
    \item[b2)] Do step a3) and a4) as above.
    \item[b3)] Set $\cS^\mr(k)=\cS^\mr(k-1)\setminus \{w\}$, $\cD^\ml(k) = \cD^\ml(k-1)\cup \{e_{w1},\dotsm, e_{wr}\}$, and $\cA^\ml(k)=\cA^{\ml}(k-1)\cup \{e\in \cU^\ml(k-1): e\textup{ incident to some }v_{wj}\}\setminus \{e_{w1},\dotsm,e_{wr}\}$. Finally, set $\cU^\ml(k)= \cU^\ml(k-1)\setminus (\cA^\ml(k)\cup \cD^\ml(k))$. Declare an ordering on $\cA^\ml(k)$ in some arbitrary way. Increment $k$ to $k+1$.
\end{enumerate}
\end{algorithm}

Let us define some stochastic processes that capture this encoding. First let $d_{(k)}^\mr$ be the degree of vertex $w_{(k)}^\mr$ from Step a2) or Step b1). Let $L(k) = \#\{i< k: \cA^\ml(i) = \emptyset\}$ be the number of times the stack $\cA^\ml$ is empty \textit{prior} to time $k$. Set
\begin{equation*}
    Y^{\mr}(k) = \sum_{j=1}^k (d_{(j)}^\mr -1)
\end{equation*}
and note that $r$ defined in Step a2) or Step b1) at time $k$ is always $r = Y^{\mr}(k)+L(k)-Y^{\mr}(k-1) - L(k-1)$, i.e. the increment of $Y^{\mr}+L$. Also note that the sequence $(d_{(j)}^\mr;j\in[m])$ is a size-biased permutation of $(d_j^\mr; j\in[m])$. Set $c_{(k)} = \#\cB(k)$ which is the number of cycles created.

Let $v_{(j)}^\ml$ be the $j^\text{th}$ $\ml$-vertex discovered. We use the convention that $v_{wi}$ is discovered before $v_{wj}$ in Step a3) or b2) if $i<j$. Let $d_{(j)}^\ml$ be the degree of $v_{(j)}^\ml$. Observe that at time $k$ in the algorithm, precisely
\begin{equation*}
    r - c_{(k)} = Y^{\mr}(k)-Y^{\mr}(k-1) + L(k) - L(k-1) - c_{(k)}
\end{equation*} many $\ml$-vertices are discovered. Therefore, set
\begin{equation*}
    V(k) = Y^{\mr}(k) + L(k) - \sum_{j=1}^k c_{(j)}
\end{equation*}
and note that these $\ml$-vertices discovered are $v_{(j)}^\ml$ for $V(k-1)< j \le V(k).$
The size of the stack $\cA^\ml(k)$ changes by
\begin{align*}
  \#\cA^\ml(k) - \#\cA^\ml(k-1) =\underset{\textup{new half-edges added}}{\underbrace{\sum_{V(k-1)<j\le V(k)} (d_{(j)}^\ml -1 ) }}- \underset{\textup{back-edges found}}{\underbrace{c_{(k)}}} - \underset{\textup{removing }a}{\underbrace{\begin{cases}
      1 &: \cA(k-1)\neq \emptyset\\
      0&: \cA(k-1)= \emptyset
  \end{cases}}}.
\end{align*}
Thus set
\begin{equation}
    Y^{\ml}(k) = \sum_{j=1}^k (d_{(j)}^\ml-1)\qquad \textup{ and }\qquad \label{eqn:ZtildeDef}
\widetilde{Z}(k) = Y^{\ml}\circ V(k) - \sum_{j=1}^k c_{(j)} - k.
\end{equation}
The above analysis implies that the collection of $\mr$-vertices in the $j^{\textup{th}}$ connected component explored in the algorithm above (say  $\cC_{(j)}$) are
\begin{equation*}
    \{v\in \cC_{(j)}: v\textup{ is an }\mr\textup{ vertex}\} = \{w_{(k)}^\mr: \widetilde{\tau}(j-1)< k\le \widetilde{\tau}(j)\}
\end{equation*} where ${\widetilde{\tau}}(j) = \inf\{k: \widetilde{Z}(k) = -j\}.$
By the definition of $L$, we can re-write this as the following useful lemma.
\begin{lemma} \label{eqn:lemZnr} Let $s>0$ be a real number.
    Let ${Z}^{(s)} (t) = \widetilde{Z}(t) - s L(t)$. Then 
    \begin{equation*}
    \{x\in \cC_{(j)}: x\textup{ is an }\mr\textup{ vertex}\} = \{w_{(k)}^\mr: {\tau}(j-1)< k\le {\tau}(j)\}
    \end{equation*} where ${{\tau}}(j) = \inf\{k: {Z}^{(r)}(k) = -(1+r)j\}.$ Moreover, 
    \begin{equation*}
    \{v\in \cC_{(j)}: v\textup{ is an }\ml\textup{ vertex}\} = \{v_{(k)}^\ml: V({\tau}(j-1))< k\le V( {\tau}(j))\}.
    \end{equation*}
\end{lemma}

We now re-write this lemma in terms of point processes a la Aldous \cite{Aldous.97}. To do this, we introduce some notation. Let $
    \Xi\subset (0,\infty)^2\times \R$ be a locally finite point set of the form $\Xi =\{(t_\alpha,x_\alpha,y_\alpha):\alpha \in A\}$ for some index set $A$ with the property that
    \begin{equation*}
        t_\alpha = \sum_{\beta: t_\beta<t_\alpha} x_\beta.
    \end{equation*} Further, suppose that for all $\eps>0$ that
    \begin{equation*}
        \#\Xi\cap (0,\infty)\times [\eps,\infty)\times \R_+<\infty.
    \end{equation*} In this case we can index all the elements of $\Xi=\{(t_i,x_i,y_i);i\ge 1\}$ such that $x_i\ge x_j$ for all $i\le j$ and if $x_i = x_j$ then $t_i<t_j$. Under these conditions, we will define $\ORD(\Xi) = \{(x_i,y_i);i\ge 1\}$. 
\begin{lemma}\label{lem:pointProcessLemma} Let $s>0$ be fixed and let 
\begin{equation*}
    \Xi_n := \{(\tau(j), \tau(j)-\tau(j-l), V(\tau(j))-V(\tau(j-1))): j\ge 1\}.
\end{equation*}
Then
\begin{equation*}
    \ORD(\Xi) = \left( (\#\cC^\mr(i),\#\cC^\ml(i));i\ge 1\right)
\end{equation*}
where $\cC(1),\cC(2),\dotsm$ are the connected components of $\GBCM(\bd^\ml,\bd^\mr)$ listed in the order of decreasing number of $\mr$ vertices with ties broken by the order they appear in Exploration \ref{exploration:DFS}.
\end{lemma}

As we will see below, establishing a scaling limit for $\widetilde{Z}^{(r)}$ for a particular choice of $r$ is easier than establish a scaling limit for $\widetilde{Z}$. In the sequel, we will state scaling limits for the various processes involved in the definition of $Z$ above. We will always index the stochastic processes using a subscript $n$, but we will often omit the index from the corresponding increment. That is, we will write $Y^{\mr}_n$ when we have degrees $\bd^{\ml,(n)}$ and $\bd^{\mr,(n)}$ but we will write $d_{(j)}^{\mr}$ instead of $d_{(j)}^{\mr,(n)}$ unless it is needed for clarity.

\subsection{Weak convergence of $Y_n^\ml$ and $Y_n^\mr$}

We begin by stating several lemmas that will be useful in subsequent analysis. The statement of these results are quite similar to those found in \cite{DvdHvLS.17,Riordan.12} (under Assumption \ref{ass:asymptProp}) and \cite{DvdHvLS.20} (under Assumption \ref{ass:infintieThird}); however, these results dealt exclusively with the case
\begin{equation*}
    \frac{\E[D_n^\ml(D_n^\ml - 1)]}{\E[D_n^\ml]} \to 1\qquad\textup{and}\qquad \frac{\E[D_n^\mr(D_n^\mr-1)]}{\E[D_n^\mr]}\to 1.
\end{equation*} While we believe that one can alter the proofs of the weak-convergence results in \cite{Riordan.12,DvdHvLS.17,DvdHvLS.20} to obtain the statements below, we include a unified approach to both cases in Appendix \ref{sec:GenWeak} below.

\begin{lemma}\label{lem:YnmlAndYnmr}
Suppose $\bd^\ml$ and $\bd^\mr$ satisfy Assumption \ref{ass:asymptProp}. Then, jointly as $n\to\infty$, 
\begin{subequations}
\begin{align}
    \label{eqn:YnmlLimitModerate} &n^{-1/3}\left( Y_n^\ml(n^{2/3}t) - n^{2/3} \frac{\E[D_n^\ml(D_n^\ml-1)]}{\E[D_n^\ml]}t\right)_{t\ge 0}\weakarrow \left(\sqrt{\kappa^\ml} W^\ml(t)- \frac{\rho^\ml}{2} t^2\right)_{t\ge 0}\\
    \label{eqn:YnmrLimitModerate}& n^{-1/3}\left( Y_n^\mr(n^{2/3}t) - n^{2/3} \frac{\E[D_n^\mr(D_n^\mr-1)]}{\E[D_n^\mr]} t\right)_{t\ge 0}\weakarrow \left(\sqrt{\kappa^\mr} W^\mr(t)- \frac{\rho^\mr}{2\theta} t^2\right)_{t\ge 0}.
\end{align}
\end{subequations}
where $W^\ml,W^\mr$ are two independent standard Brownian motions and $\kappa^\ast, \rho^\ast$, $\ast\in \{\ml,\mr\}$ are defined in \eqref{eqn:sigmamlmrDef}--\eqref{eqn:rhopmlmrDef}. 
\end{lemma}

To analyze triangle counts later on, it will be convenient to mention the following lemma, which is an application of Lemma 8.2 in \cite{BSW.17}. We write $\operatorname{Id} = (t;t\ge 0)$ as the identity stochastic process. 
\begin{lemma}[Bhamidi, Sen, Wang {\cite[Lemma 8.2]{BSW.17}}]\label{lem:BSWlemma}
    Suppose $\bd^\ml,\bd^\mr$ satisfy Assumption \ref{ass:asymptProp}. Then for both $p = 1,2$ and 
    \begin{align*}
        \left(n^{-2/3}\sum_{j=1}^{n^{2/3}t} (d_{(j)}^\ml)^p \right)_{t\ge 0}\weakarrow \frac{\mu_{p+1}^\ml}{\mu_1^\ml}\operatorname{Id},&&&
        \left(n^{-2/3}\sum_{j=1}^{n^{2/3}t} (d_{(j)}^\mr)^p \right)_{t\ge 0}\weakarrow \frac{\mu_{p+1}^\mr}{\mu_1^\mr}\operatorname{Id}.
    \end{align*}
    
    In addition, if $D_n^\ml\to D_\infty^\ml$, $D_n^\mr\to D_\infty^\mr$ in $\mathscr{W}_4$. Then 
    \begin{align*}
        \left(n^{-2/3}\sum_{j=1}^{n^{2/3}t} \binom{d_{(j)}^\ml}{3}\right)_{t\ge 0} \weakarrow \frac{\E[D_\infty^\ml \binom{D_\infty^\ml}{3}]}{\E[D_\infty^\ml]}\operatorname{Id},&&&
        \left(n^{-2/3}\sum_{j=1}^{n^{2/3}t} \binom{d_{(j)}^\mr}{3}  \right)_{t\ge 0} \weakarrow \frac{\E[D_n^\mr\binom{D_\infty^\mr}{3}]}{\E[D_\infty^\mr]}\operatorname{Id}.
    \end{align*}
\end{lemma}
\begin{remark}
    One can obtain a fluctuation result by applying Theorem \ref{thm:SIZEBIASPARTIALSUMS} below if we replace the Wasserstein $\mathscr{W}_4$ convergence with $\mathscr{W}_9$ convergence. We leave the details of this to the reader. See the proofs of Lemmas \ref{lem:secondMOme} and \ref{lem:TriangleInfnite} where the details are laid out under Assumption \ref{ass:infintieThird}.
\end{remark}

Under Assumption \ref{ass:infintieThird} we have the following lemma that is key to analyzing the sizes of the connected components. 
\begin{lemma}\label{lem:YconvInfinitThird}
   Suppose $\bd^\ml$ and $\bd^\mr$ satisfy Assumption \ref{ass:infintieThird}. Then, jointly as $n\to\infty$, 
\begin{subequations}
\begin{align}
    \label{eqn:YnmlLimitInfin} & a_n^{-1}\left( Y_n^\ml(b_nt) - b_n \frac{\E[D_n^\ml(D_n^\ml-1)]}{\E[D_n^\ml]}t\right)_{t\ge 0}
    \weakarrow\left(\sum_{j=1}^\infty \beta_j^\ml \left(\bone_{[\eta^\ml_j\le t/\mu_1^\ml]} - \frac{\beta_j^\ml}{\mu_1^\ml} t\right)\right)_{t\ge 0}\\
    \label{eqn:YnmrLimitInfin}&  a_n^{-1}\left( Y_n^\mr(b_n t) - b_n \frac{\E[D_n^\ml(D_n^\ml-1)]}{\E[D_n^\ml]} t\right)_{t\ge 0}\weakarrow \left(\sum_{j=1}^\infty \theta^{1-s} \beta_j^\mr \left(\bone_{[\eta_j^\mr\le \theta^{-s} t/\mu_1^\mr]} - \frac{\theta^{-s} \beta^\mr_j}{\mu_1^\mr} t\right)
   \right)_{t\ge 0}
\end{align} where $s = \frac{\tau-2}{\tau-1}$.
where $\eta_j^\ml \sim \Exp(\beta_j^\ml)$ and $\eta_j^\mr\sim \Exp(\beta_j^\mr)$ are independent exponential random variables.
\end{subequations}
\end{lemma}

To establish tightness in $\ell^2$ of the component sizes and to analyze triangle counts later on, we will need to following two lemmas. Like the proofs of Lemma \ref{lem:YnmlAndYnmr}  and \ref{lem:YconvInfinitThird}, we delay the proofs of these two lemmas until Appendix \ref{sec:GenWeak}.
\begin{lemma}
\label{lem:secondMOme}
    Suppose that $\bd^\ml$ and $\bd^\mr$ satisfy Assumption \ref{ass:infintieThird}. Then, jointly with the convergence in \eqref{eqn:YnmlLimitInfin}--\eqref{eqn:YnmrLimitInfin}, 
    \begin{align*}
        &\left(\frac{1}{a_n^2} \sum_{j=1}^{b_nt} d_{(j)}^\ml(d_{(j)}^\ml-1)\right)_{t\ge 0} \weakarrow \left(\sum_{j=1}^\infty(\beta^\ml_j)^2 \bone_{[\eta_j\ml\le t/\mu_1^\ml]}\right)_{t\ge 0}\\
        &\left(\frac{1}{a_n^2} \sum_{j=1}^{b_nt} d_{(j)}^\mr(d_{(j)}^\mr-1)\right)_{t\ge 0} \weakarrow \left({\theta^{\frac{-2}{\tau-1}}}\sum_{j=1}^\infty(\beta^\mr_j)^2 \bone_{[\eta_j^\mr\le \theta^{-\frac{\tau-2}{\tau-1}} t/\mu_1^\mr]}\right)_{t\ge 0}
    \end{align*}    
\end{lemma}
\begin{lemma}\label{lem:TriangleInfnite}
    Suppose that $\bd^\ml$ and $\bd^\mr$ satisfy Assumption \ref{ass:infintieThird}. Then, jointly with the convergence in \eqref{eqn:YnmlLimitInfin}--\eqref{eqn:YnmrLimitInfin}, 
    \begin{align*}
        &\left(\frac{1}{a_n^3} \sum_{j=1}^{b_nt} \binom{d_{(j)}^\ml}{3}\right)_{t\ge 0} \weakarrow \left(\frac{1}{6}\sum_{j=1}^\infty(\beta^\ml_j)^3 \bone_{[\eta_j^\ml\le t/\mu_1^\ml]}\right)_{t\ge 0}\\
        &\left(\frac{1}{a_n^3} \sum_{j=1}^{b_nt} \binom{d_{(j)}^\mr}{3}\right)_{t\ge 0} \weakarrow \left(\frac{1}{6\theta^{\frac{3}{\tau-1}}}\sum_{j=1}^\infty(\beta^\mr_j)^3 \bone_{[\eta_j^\mr\le \theta^{-\frac{\tau-2}{\tau-1}} t/\mu_1^\mr]}\right)_{t\ge 0}
    \end{align*}
\end{lemma}

\subsection{Weak Convergence of ${Z}_n^{(r_n)}$}

 In this subsection, we establish the weak convergence of ${Z}^{(r)}_n$ for a particular choice of $r= r_n$ under Assumptions \ref{ass:asymptProp} and \ref{ass:infintieThird}. Let us introduce some notation before continuing. We write
\begin{subequations}
\begin{align}
   \label{eqn:mupmlmrDef} \mu_p^{\ml,(n)} &= \E[(D_n^\ml)^p] &\mu_p^{\mr,(n)} &= \E[(D_n^\mr)^p]\\
    \label{eqn:nupmlmrDef}\nu_n^\ml &= \frac{\E[D_n^\ml(D_n^\ml-1)]}{\E[D_n^\ml]} & \nu_n^\mr &= \frac{\E[D_n^\mr(D_n^\mr-1)]}{\E[D_n^\mr]}. 
\end{align}
\end{subequations} Also recall the constants defined in Theorem \ref{thm:MAIN1}.

Our main limiting result in this subsection involves the scaling limits of
\begin{equation}\label{eqn:Zdef}
    Z_n(t) = \widetilde{Z}_n(t) - \nu_n^\ml L(t)
\end{equation}
where $\widetilde{Z}_n$ is defined in \eqref{eqn:ZtildeDef}. We will prove the following:
\begin{proposition}\label{prop:ZfinThirdConverge}
    Let $\bd^\ml, \bd^\mr$ be degree sequences. 
    \begin{enumerate}
        \item Under Assumption \ref{ass:asymptProp} it holds that in the $J_1$ topology
    \begin{align*}
        &\left(n^{-1/3}Z_n(n^{2/3}t);t\ge 0\right)\weakarrow \left(Y^\ml_\infty\left(\nu_\infty^\mr  t\right) + \nu_\infty^\ml Y^\mr_\infty(t) + \lambda t \right)
    \end{align*}
 where $Y_\infty^\ml,Y_\infty^\mr$ are the scaling limits of \eqref{eqn:YnmlLimitModerate} and \eqref{eqn:YnmrLimitModerate}, respectively.
    \item Under Assumption \ref{ass:infintieThird} it holds that in the $J_1$ topology
    \begin{align*}
         &\left(a_n^{-1}Z_n(b_nt);t\ge 0\right)\weakarrow \left(Y^\ml_\infty\left(\nu_\infty^\mr  t\right) + \nu_\infty^\ml Y^\mr_\infty(t) + \lambda t \right)
    \end{align*}
    where $Y_\infty^\ml,Y_\infty^\mr$ are the scaling limits of \eqref{eqn:YnmlLimitInfin} and \eqref{eqn:YnmrLimitInfin}, respectively.
    \end{enumerate}
\end{proposition}

The proof of this proposition requires some work beforehand. To simplify some arguments, we will use a common set of notation. We leave the notation from Assumption \ref{ass:infintieThird} unchanged, but when working under Assumption \ref{ass:asymptProp} we write
\begin{equation*}
    a_n = c_n = n^{1/3}\qquad b_n = n^{2/3}
\end{equation*}
so that there is agreement between the two prelimits in Proposition \ref{prop:ZfinThirdConverge} and the form of the critical windows in Assumptions \ref{ass:asymptProp} and \ref{ass:infintieThird}. We also write for any sequence of functions $(f_n;n\ge 1)$ 
\begin{equation*}
    \overline{f}_n(t) = b_n^{-1}f_n(b_nt).
\end{equation*}

\subsubsection{Preliminary Tightness Results}

It is easy to see that if $\psi_n, \psi$ are c\`adl\`ag functions such that there is some function $f$ and some sequence $c_n\to\infty$ with
\begin{equation*}
    c_n \left(\psi_n - \psi \right) \longrightarrow f \qquad\textup{ in }J_1
\end{equation*} then $\psi_n \to  \psi$ locally uniformly. More generally, we have the following lemma, whose proof is elementary and omitted.
\begin{lemma}\label{lem:fluctImplyWLLN}
    Suppose $c_n\to\infty$, $\psi_n, \psi$ are c\`adl\`ag functions and $\alpha_n$ is a sequence of real numbers. Suppose that
    \begin{equation}\label{eqn:fluctuations}
        c_n \left(\psi_n -\alpha_n \psi\right) \longrightarrow f\qquad\textup{ in }J_1.
    \end{equation} If $\alpha_n\to \alpha\in \R_+$ then $\psi_n\to \alpha \psi$ locally uniformly.
\end{lemma}

Let us mention some immediate corollaries of Lemmas \ref{lem:YnmlAndYnmr} and \ref{lem:YconvInfinitThird}. 
\begin{corollary}\label{cor:Vntight} Suppose Assumption \ref{ass:asymptProp} or Assumption \ref{ass:infintieThird}. 
Then the collections 
\begin{equation*}
    \{\overline{V}_n; n\ge 1\}\qquad\textup{and}\qquad \{\overline{L}_n; n\ge 1\}
\end{equation*} are tight in $\D(\R_+,\R)$. Moreover, the subsequential weak limits are continuous. 
\end{corollary}
\begin{proof}
    Note that $L_n$ is non-decreasing and has increments bounded by $1$. The tightness and continuity stated is now standard. For $V_n$ we note that the process is almost surely non-decreasing. Moreover, its increments are bounded by those of $Y_n^\mr+L_n$. As $\overline{Y}_n^\mr(t)$ converges weakly by Lemma \ref{lem:fluctImplyWLLN} and Lemmas \ref{lem:YnmlAndYnmr}, \ref{lem:YconvInfinitThird} and $\overline{L}_n$ is tight, the process $V_n$ is tight. The continuity follows from the continuity of the weak limits of $\overline{Y}_n^\mr$ and weak subsequential limits of $\overline{L}_n$.
\end{proof}

Let \begin{equation} \label{eqn:Zupdef}
Z^{\textup{up}}_n(t) = Z_n(t) + 2 C_n(t)\qquad  \textup{where} \qquad C_n(t) = \sum_{i=1}^t c_{(i)}.
\end{equation}
Observe that $Z_n^{\textup{up}}\ge Z_n$.
\begin{lemma}\label{lem:ZtightfinThird}Suppose Assumption \ref{ass:asymptProp} or \ref{ass:infintieThird}.
    Then the sequence
    \begin{equation*}
        \left\{\left(a_n^{-1} Z_n^{\textup{up}}(b_nt);t\ge 0\right);n\ge 1\right\}\textup{ is tight}.
    \end{equation*}
\end{lemma}
\begin{proof} We will show that for any subsequence of $n$ there exists a further weakly convergent subsequence. To do this, we note that
    \begin{align*}
        Z_n^{\textup{up}} (t) &= Y_n^\ml \circ V_n(t) -  t -\nu_n^\ml L_n(t) = Y^\ml_n\circ V_n(t) - \nu_n^\ml V_n(t) + \nu_n^\ml (V_n(t) - L_n(t)) - t\\
        &= \left(Y^\ml_n\circ V_n(t)-\nu_n^\ml V_n(t)\right) + \nu_n^\ml \left(Y_n^\mr(t) - \nu_n^\mr t\right)+ (\nu_n^\ml\nu_n^\mr-1)t.
    \end{align*}
    
    Let us fix a subsequence, which we denote by $n$ for simplicity. By  Corollary \ref{cor:Vntight}, we can find a further subsequence, say $(n_k)$, and a continuous process $V_\infty$ such that
    \begin{equation*}
       \left(\overline{V}_{n_k}(t);t\ge 0\right) \weakarrow \left(V_\infty(t);t\ge 0\right).
    \end{equation*} 
    By Lemmas \ref{lem:YnmlAndYnmr}, \ref{lem:YconvInfinitThird} and the continuity of composition in \cite{Whitt.80} or \cite{Wu.08}, 
    \begin{equation*}
        \left(a_{n_k}^{-1}\left(Y^\ml_{n_k}\circ V_{n_k}(b_{n_k} t) - \nu_{n_k}^{\ml} V_{n_k}(b_{n_k}t) \right); t\ge 0\right) \weakarrow Y^\ml \circ V_\infty
    \end{equation*} in the $J_1$ topology. Similarly, in $J_1$
    \begin{equation*}
        \left(a_{n_k}^{-1}\left(Y_{n_k}^\mr(b_{n_k} t) -  \nu_{n_k}^\mr  b_{n_k}t\right); t\ge 0\right) \weakarrow Y^\mr.
    \end{equation*} Lastly, by the near-criticality Assumption \ref{ass:asymptProp}\textbf{(iv)} or \ref{ass:infintieThird}\textbf{(iv)}
    \begin{equation*}
        \nu_n^\ml \nu_n^\mr -1 = \nu_n -1  = (\lambda + o(1)) \frac{a_n}{b_n}.
    \end{equation*}Since $\nu_n^\ml \to \frac{\mu_2^\ml - \mu_1^\ml}{\mu_1^\ml}$ we see that along $n_k$ that    \begin{equation}\label{eqn:Zupweaksubsequentiallimits}
    \left(a_{n_k}^{-1}(Z_{n_k}^{\textup{up}}(b_{n_k}t));t\ge0\right) \weakarrow \left(Y^\ml\circ V_\infty(t) + \nu_\infty^\ml Y^\mr(t) + \lambda t;t\ge 0\right).
    \end{equation}
\end{proof}

\subsubsection{Number of Cycles Constructed}

Our next task is to show that by step $k = O(b_n)$ we have only seen $o(a_n)$ many cycles created. More precisely, we show:
\begin{lemma}\label{lem:cnfinthrid}
    Suppose Assumption \ref{ass:asymptProp}. Then
    \begin{equation*}
        \left(a_n^{-1} C_n(b_n t);t\ge 0\right) \weakarrow \bzer :=(0;t\ge 0).
    \end{equation*}
\end{lemma}

\begin{proof}

Let us look at the construction of the cycles $c_{(k+1)}$ at time $k+1$ in the exploration. Note that we have at most $r\le d_{(k+1)}^\mr$ many $\mr$-half-edges that can form cycles by looking at Step b1) or a2). Any cycle is formed by pairing an $\mr$-half-edge $f_{wj}$, for $j\in[r]$, with an $\ml$-half-edge in $\cA^\ml(k)$ or with a half-edge paired to a newly discovered $\ml$-vertex $v_{wi}$ for $i=1,2,\dotsm, j-1$. As $\{v_{wi};i\in[r]\} = \{v_{(j)}^\ml: V_n(k)< j\le V_n(k+1)\}$, we see that there are at most
\begin{equation*}
    \# \cA^\ml(k) + \sum_{i=V_n(k)+1}^{V_n(k+1)} d_{(i)}^\ml 
\end{equation*}
many $\ml$-half edges that can be paired to form such a cycle.
Moreover, each $\ml$-half edge not connected to $v_{(i)}^\ml$ for $i\le V_n(k+1)$ will not form a cycle and the number of such half-edges is
\begin{equation*}
    \sum_{i> V_{n}(k+1) } d_{(i)}^\ml = n\E[D_n^\ml] - \sum_{i=1}^{V_n(k+1)} d_{(i)}^\ml.
\end{equation*}Hence, the probability that the half-edge $f_{wj}$ forms a cycle is at most $p = p(k)$ defined by
\begin{equation}\label{eqn:Pkdef}
    p (k):=\frac{\#\cA^\ml(k) + \sum_{i=V_n(k)+1}^{V_n(k+1)} d_{(i)}^\ml}{n\E[D_n^\ml] - \sum_{i=1}^{V_n(k+1)} d_{(i)}^\ml}.
\end{equation}

Let $\eps>0, \delta> 0$, $T>0$ be fixed. We claim that for all $n$ sufficiently large
\begin{equation}\label{eqn:pboundFinThird}
    \PR\left(\sup_{k\le b_n T} p(k)> \delta a_n^{-1} \right) \le \eps.
\end{equation} To see why \eqref{eqn:pboundFinThird} holds, note that since $j\mapsto C_n(j)$ is non-decreasing
\begin{align*}
    \#\cA^\ml(k) &= Z_n(k)-\inf_{j\le k} Z_n(j)  = Z_n^{\textup{up}}(k) - 2C_n(k)- \inf_{j\le k} \left(Z_n^{\textup{up}}(j) - 2C_n(j)\right)\\
    &\le  Z_n^{\textup{up}} (k)-\inf_{j\le k} Z^{\textup{up}}_n(j).
\end{align*}
Therefore, by the tightness in Lemma \ref{lem:ZtightfinThird}, there exists an $M<\infty$ such that for all $n$ large
\begin{align}\label{eqn:pmax_1}
\PR&\left(\sup_{k\le b_n T} \# \cA^\ml(k) > M a_n\right) \le \PR\left(\sup_{t\le T} a_n^{-1}\big(Z^{\textup{up}}_n(b_nt)- \inf_{s\le t} Z_n^{\textup{up}}(b_ns) \big)> M \right) \le \eps/2.
\end{align}
Observe that by since 
\begin{equation}\label{eqn:pmax_2}
\sup_{k\le b_n T} \sum_{j= V_n(k)+1}^{V_n(k+1)} d_{(j)}^{\ml} = Y_n^\ml(V_n(k))-Y_n^\ml(V_n(k-1)) + V_n(k)-V_n({k-1})
\end{equation} and $R_n(t):= \overline{Y}_n^\ml\circ \overline{V}_n(t) + \overline{V}_n(t)$ is tight, with subsequential limits that are a.s. continuous (Corollary \ref{cor:Vntight}), we see that
\begin{equation*}
   \limsup_{n\to\infty}  \PR\left(\sup_{k\le b_n T} \sum_{j= V_n(k)+1}^{V_n(k+1)} d_{(j)}^{\ml} > \frac{2\delta}{\mu_1^\ml} b_n \right) = \limsup_{n\to\infty}  \PR(\sup_{t\le T} R_n(t)-R_n(t-)> \frac{2\delta}{\mu_1^\ml})  = 0.
\end{equation*} Combining this with \eqref{eqn:pmax_1}, we see that for $\delta>0$, $\eps>0$ fixed that for all large $n$
\begin{equation}\label{eqn:deltaboundhelps00}
    \PR\left(\sup_{k\le b_n T}\# \cA^\ml(k) + \sum_{j=V_n(k)+1}^{V_n(k)} d_{(j)}^\ml > \frac{2\delta}{\mu_1^\ml} b_n \right) \le \frac{\eps}{2}.
\end{equation}
Let us also observe that
\begin{equation*}
  b_n^{-1}  \sup_{k\le b_nT} \sum_{j= 1}^{V_n(k+1)} d_{(j)}^{\ml} \le  b_n^{-1}\sum_{j=1}^{V_n(2 b_n T)} d_{(j)}^\ml =\overline{Y}_n^\ml \circ \overline{V}_n(2T) + \overline{V}_n(2T).
\end{equation*} 
By combining Lemma \ref{lem:fluctImplyWLLN}, Lemmas \ref{lem:YnmlAndYnmr}, \ref{lem:YconvInfinitThird}, and Corollary \ref{cor:Vntight}, it is easy to see that right-hand side is a tight random variable for each choice of $T$. As $\E[D_n^\ml]\to \mu_1^\ml>0$ by Assumption \ref{ass:asymptProp}\textbf{(iii)} or Assumption \ref{ass:infintieThird}\textbf{(iii)}, we can find an $n$ sufficiently large such that 
\begin{equation}\label{eqn:pmax_3}
    \PR\left(\inf_{k\le b_nT } (n \E[D_n^\ml] - \sum_{j= 1}^{V_n(k+1)} d_{(j)}^{\ml}) < \frac{\mu_1^\ml}{2} n \right) \le \eps/2.
\end{equation}
Combining \eqref{eqn:pmax_1}--\eqref{eqn:pmax_3}, we see that
\begin{align*}
    \PR&(\sup_{k\le b_n T}p(k) > \delta a_n^{-1}) \\
    &\le \PR\left(\sup_{k\le b_nT} \frac{\#\cA^\ml(k) + \sum_{j=V_n(k)+1}^{V_n(k+1)} d_{(j)}^\ml}{\frac{\mu_1^\ml}{2}n} > \delta a_n^{-1}\right) +\PR\left(\inf_{k\le b_nT}\left (n \E[D_n^\ml] - \sum_{j= 1}^{V_n(k+1)} d_{(j)}^{\ml}\right) <\frac{\mu_1^\ml}{2} n \right)\\
    &\le \PR\left(\sup_{k\le b_n T} \#\cA^\ml(k)+\sum_{j=V_n(k)+1}^{V_n(k+1)} d_{(j)}^\ml> \frac{2\delta}{\mu_1^\ml} b_n \right)  + \eps/2 \le \eps.
\end{align*}

Let $\F_k = \sigma(d_{(j)}^\mr; j\le k)$ and observe as in the proof of Lemma 5.4 in \cite{DvdHvLS.17} that for $k\le b_n T$
\begin{equation*}
    \E[d_{(k+1)}^\mr | \F_k] = \frac{\sum_{j=1}^m (d_j^\mr)^2 - \sum_{i=1}^{k} d_{(i)}^\mr}{\sum_{j=1}^m d_{j}^\mr  - \sum_{i=1}^k d_{(i)}^\mr} \le \frac{m \E[(D_n^\mr)^2]}{m \E[D_n^\mr] - b_n T - Y_n^\mr(b_n T)}.
\end{equation*} 
Therefore, on $E_1(M,T):= \{\overline{Y}_n^\mr (T) \le TM\}$ the denominator above is $m \mu_1^{\mr,(n)}-o(1)$ and so
\begin{align}
   & \sum_{j=1}^{b_nT} \E[d_{(j)}^\mr\bone_{E_1(M,T)}]  =   \sum_{j=1}^{b_nT} \E\left[\E[d_{(j)}^\mr \bone_{E_1(M,T)}|\F_k]\right] 
    \le \frac{\mu_2^{\mr,(n)} Tb_n}{\mu_1^{\mr,(n)} - o(1)} = \frac{\mu_2^\mr}{\mu_1^\mr}(1+o(1))Tb_n.\label{eqn:dsummr}
\end{align}

Observe that on the event $E_2(\delta,M,T):= E_{1}(M,T) \cap \{\sup_{k\le b_nT} p(k) \le \delta a_n^{-1}\}$ the random variable $c_{(k+1)}$ is stochastically dominated by $\operatorname{Bin}(d_{(k+1)}^\mr, \delta a_n^{-1}).$ Hence
\begin{align*}
    \E[a_n^{-1}& C_n(b_nT) \mathbf{1}_{E_2(\delta,M,T)}] \le a_n^{-1}\sum_{j=1}^{b_nT} \E[\operatorname{Bin}(d_{(k+1)}\bone_{E_1(M,T)}, \delta a_n^{-1})]\\
    &\le \delta a_n^{-2} \sum_{j=1}^{b_nT} \E[d_{(j)}^\mr\bone_{E_1(M,T)}] \overset{\eqref{eqn:dsummr}}{\le} \frac{\delta b_n}{a_n^2} \frac{\mu_2^{\mr}}{\mu_1^\mr}(1+o(1))T \le (1+o(1))\frac{\mu^\mr_2}{\mu_1^\mr} \delta T
\end{align*}where the last inequality uses $b_n \le a_n^2$.
We conclude using Markov's inequality that for any $\eps,\delta, \eta, M, T>0$
\begin{align*}
    \PR& \left( a_n^{-1}C_n(b_n T) > \eps\right)\le \PR \left( a_n^{-1}C_n(b_n T) > \eps\textup{ and } E_2(\delta,M,T)\right) + \PR( E_2(\delta,M,T)^c)\\
    &\le \frac{\E[a_n^{-1} C_n(b_nT) \mathbf{1}_{E_2(\delta,M,T)}]}{\eps} +\PR( E_2(\delta,M,T)^c) \le (1+o(1))\frac{\mu_2^\mr}{\mu_1^\mr} \frac{\delta T}{\eps} + \PR( E_2(\delta,M,T)^c)
\end{align*} for all $n$ sufficiently large. Note that for $M$ sufficiently large, and any $\delta>0$
\begin{equation*}
    \limsup_{n\to\infty} \PR(E_2(\delta,M, T)^c) = 0
\end{equation*}
since $\overline{Y}_n^\mr$ is tight by Lemmas \ref{lem:YnmlAndYnmr} (or \ref{lem:YconvInfinitThird}) and Lemma \ref{lem:fluctImplyWLLN}. 
Taking the $\limsup_{n\to\infty}$ gives
\begin{equation*}
    \limsup_{n\to\infty}\PR \left(a_n^{-1} C_n(b_n T) \ge \eps\right) \le \frac{\mu_2^\mr T}{\mu_1^\mr \eps} \delta 
\end{equation*} As this holds for all $\delta,T>0$, we get
$\limsup_{n\to\infty}\PR \left(a_n^{-1} C_n(b_n T) \ge \eps\right) =0$
which, together with the fact that $C_n$ is non-decreasing and starts from $0$, is sufficient to establish the desired weak convergence.
\end{proof}

\subsubsection{Finishing up}

The next lemma is a slight generalization of Corollary \ref{cor:Vntight}.
\begin{lemma}\label{lem:VnFinThird}
   Suppose Assumption \ref{ass:asymptProp}. It holds that 
    \begin{equation*}
        \left(\overline{L}_n(t) ;t\ge 0\right) \weakarrow \bzer
 \qquad\textup{and}\qquad 
        \left(\overline{V}_n(t) ;t\ge 0\right) \weakarrow \nu_\infty^\mr\operatorname{Id}.
    \end{equation*}
\end{lemma}
\begin{proof}
    By combining Lemma \ref{lem:ZtightfinThird} and Lemma \ref{lem:cnfinthrid}, we see that
    \begin{equation*}
        \left\{\left(a_n^{-1} Z_n(b_n t);t\ge0\right);n\ge 1\right\} \textup{ is tight.}
    \end{equation*}
    As $a_n = o(b_n)$, 
    $
        \left(b_n^{-1} Z_n(b_n t);t\ge0\right) \weakarrow \bzer
    $ and $
        \left(b_n^{-1}\inf_{s \le b_n t} Z_n(s);t\ge 0\right) \weakarrow \bzer.$
    But $L_n(t) = -\frac{1}{1+\nu_n^\ml} \inf_{s\le t} Z_n(s).$ The re-scaled convergence of $L_n$ now follows. 

    For the convergence of $V_n$ we have
    \begin{equation*}
        \overline{V}_n(t) = \overline{Y}^\mr_n(t) + \overline{L}_n(t) - \overline{C}_n(t)
    \end{equation*} and Lemmas \ref{lem:fluctImplyWLLN} and Lemma \ref{lem:YnmlAndYnmr} imply $\overline{Y}^\mr_n\weakarrow \frac{\mu_2^\mr - \mu_1^\mr}{\mu_1^\mr} \operatorname{Id} = \nu_\infty^\mr \operatorname{Id}$. Also by Lemma \ref{lem:cnfinthrid}, $\overline{C}_n \weakarrow \bzer$. The result now follows from the convergence of $\overline{L}_n$ proved above.
\end{proof}

We can now prove Proposition \ref{prop:ZfinThirdConverge}.
\begin{proof}[Proof of Proposition \ref{prop:ZfinThirdConverge}]
    Note that
    \begin{equation*}
        a_n^{-1} Z_n(b_nt) = a_n^{-1} Z_n^{\textup{up}}(b_n t)  - 2 a_n^{-1} C_n(b_nt).
    \end{equation*}
    As observed in \eqref{eqn:Zupweaksubsequentiallimits}, the subsequential weak limits of $a_n^{-1}Z_n^{\textup{up}}(b_nt)$ are always of the form $Y^\ml\circ V_\infty(t) + \nu_\infty^\ml Y^\mr(t) + \lambda t$ where $V_\infty(t)$ is a subsequential weak limit of $\overline{V}_n(t)$. By Lemma \ref{lem:VnFinThird}, the weak limits $V_\infty(t) = \nu_\infty^\ml t$ no matter the subsequence. The result now follows from Lemma \ref{lem:cnfinthrid}.
\end{proof}

\subsection{Basic Scaling Relationships}

Let us rewrite the form of the limits appearing in Proposition \ref{prop:ZfinThirdConverge} in order to better match the scaling in Theorems \ref{thm:MAIN1} and \ref{thm:MAIN2}. The following two lemmas are elementary and included without proof.
\begin{lemma}\label{lem:jumplevypart}
    For any $\kappa,\rho\ge 0$, $\bbeta\in\ell^3_\downarrow$, $\lambda\in \R$ and any $a>0$
    \begin{equation*}
        \left(a \bW^{\kappa,\rho,\lambda,\bbeta}(at) ;t\ge 0\right)\overset{d}{=} \left(\bW^{a^3\beta,a^3\rho,a^2\lambda, a\bbeta}(t);t\ge 0\right)
    \end{equation*} where $\bW$ is defined using \eqref{eqn:bJump}--\eqref{eqn:bWeiner}.
\end{lemma}

\begin{lemma}\label{lem:summingtwothinned}
    Suppose that $\bW^{\kappa_j,\rho_j,\lambda_j,\bbeta_j}_j$, $j\in[2]$, are independent thinned L\'{e}vy processes defined \eqref{eqn:bJump}--\eqref{eqn:bWeiner}. Then
    \begin{equation*}\bW_1^{\kappa_1,\rho_1,\lambda_1,\bbeta_1}+\bW_2^{\kappa_2,\rho_2,\lambda_2,\bbeta_2}\overset{d}{=} \bW^{\kappa,\rho,\lambda,\bbeta}
    \end{equation*}
    where $\kappa= \kappa_1+\kappa_2,\rho = \rho_1+\rho_2, \lambda = \lambda_1+\lambda_2$ and $\bbeta = \bbeta_1\bowtie\bbeta_2$.
\end{lemma}

We have the following corollary of Lemma \ref{lem:jumplevypart} and the observation that
\begin{equation*}
 \mu_1^\ml \sim  \E[D_n^\ml] = n^{-1} \sum_{j} d_j^\ml = n^{-1} \sum_{j} d_j^\mr = \frac{m}{n}\E[D_n^\mr] \sim \theta \mu_1^\mr.
\end{equation*}
\begin{corollary}
    In Lemma \ref{lem:YconvInfinitThird}, the limits $Y_\infty^\ml$ and $Y_\infty^\mr$ are equal in law to
    \begin{align*}
        (Y_\infty^\ml(\nu_\infty^\mr t))_{t\ge 0} = \frac{\mu^\ml_1}{\nu_\infty^\mr} \bV^{\boldsymbol{\gamma}^\ml}=\mu_1^\ml \nu_\infty^\ml  \bV^{\boldsymbol{\gamma}^\ml}\qquad \left(\nu_\infty^\ml Y_\infty^\mr(t)\right)_{t\ge 0} = \nu^\ml_\infty \theta \mu_2^\mr \bV^{ \boldsymbol{\gamma}^\mr} = \nu^\ml_1 \mu_1^\ml  \bV^{ \boldsymbol{\gamma}^\mr}.
    \end{align*} where $\boldsymbol{\gamma}^\ml = \nu_\infty^\mr \bbeta^\ml/\mu_1^\ml$ and $\boldsymbol{\gamma}^\mr = \theta^{-\frac{\tau-2}{\tau-1}} \bbeta^\mr/\mu_1^\mr = \theta^{1/(\tau-1)} \bbeta^\mr/\mu_1^\ml$.
\end{corollary}

Using these two lemmas, and the fact that $\nu_\infty^\ml \nu_\infty^\mr = 1$, we obtain the following immediate corollary of Proposition \ref{prop:ZfinThirdConverge}.
\begin{corollary}\label{cor:rewrite}
Let $\bd^\ml, \bd^\mr$ be degree sequences. 
    \begin{enumerate}
        \item Under Assumption \ref{ass:asymptProp}
    \begin{align*}
        &\left(n^{-1/3}Z_n(n^{2/3}t);t\ge 0\right)\weakarrow \left(\nu_\infty^\ml\bW^{\kappa,\rho,\lambda\nu_\infty^\mr,\bzer} (t)\right) = \left(\nu_\infty^\ml\bW^{\kappa,\rho,0,\bzer} (t) + \lambda t;t\ge 0\right)
    \end{align*}
    where $\kappa = (\nu_\infty^\mr)^3\kappa^\ml + \kappa^\mr$ and $\rho =(\nu_\infty^\mr)^3\rho^\ml + \rho^\mr/\theta$ as in \eqref{eqn:kapparhodefforlimit}.
    \item Under Assumption \ref{ass:infintieThird}, 
    \begin{align*}
         &\left(a_n^{-1}Z_n(b_nt);t\ge 0\right)\weakarrow \left(\nu^\ml_\infty \mu_1^\ml \bW^{0,0, \lambda\nu_\infty^\mr/\mu_1^\ml, \bbeta}(t);t\ge 0 \right) =  \left(\nu^\ml_\infty \bW^{0,0, 0, \bbeta}(t) + \lambda t;t\ge 0 \right)  
    \end{align*}
   where $\bbeta = \left(\nu^\mr_\infty \bbeta^\ml/\mu_1^\ml\right) \bowtie \left(\theta^{1/(\tau-1)}\bbeta^\mr/\mu_1^\ml\right) .$
    \end{enumerate}
\end{corollary}
\begin{proof}
    The limits from Proposition \ref{prop:ZfinThirdConverge} are
    \begin{align*}
        &Y^\ml_\infty(\nu^\mr_\infty t) = \nu^\ml_\infty \nu_\infty^\mr \bW^{\kappa^\ml,\rho^\ml, 0,\bbeta^\ml}(\nu^\mr_\infty t)\qquad\textup{and}\qquad \nu_\infty^\ml Y_\infty^\mr(t) = \nu^\ml_\infty\bW^{\kappa^\mr,\rho^\mr/\theta,0,\bbeta^\mr}
    \end{align*}
    where $\kappa^\ast,\rho^\ast,\bbeta^\ast$ are clear based on what assumptions are used. Now apply Lemmas \ref{lem:jumplevypart} and \ref{lem:summingtwothinned}.
\end{proof}

\subsection{Appearance of Surplus Edges}
We now extend Lemma \ref{lem:cnfinthrid} so that we can more easily identify the limiting triangle counts for the associated random intersection graph. We start with the following lemma, for which we will define $Z_\infty$ as the weak limit of $a_n^{-1} Z_n(b_nt)$ appearing in Proposition \ref{prop:ZfinThirdConverge}. 
\begin{lemma}
    Suppose $\bd^\ml,\bd^\mr$ satisfy either Assumption \ref{ass:asymptProp} or \ref{ass:infintieThird}. Then for each $T>0$
    \begin{align*}
     & a_n^{-1}\sup_{t\le T}\#\cA^\ml(b_nt)\weakarrow \sup_{t\le T} (Z_\infty(t)-\inf_{s\le t}Z_\infty(s))\\
      &a_n^{-1} \sup_{t\le T} \sum_{j=V_n(b_nt)+1}^{V_n(b_nt +1 )} d_{(j)}^\ml \weakarrow \sup_{t\le T} (Z_\infty(t)-Z_\infty(t-)). 
    \end{align*}
\end{lemma}
\begin{proof}
    The first result is simply
    \begin{equation*}
        \sup_{t\le T} \#\cA^\ml(b_nt) = \sup_{t\le T} \left(Z_n(b_nt) - \inf_{s\le t}Z_n(b_ns)\right)
    \end{equation*}
    and the continuity of the Skorohod reflection map \cite[Theorem 13.4.1]{Whitt.02}.

    We now turn to the second. Let $\Delta f(t) = f(t)-f(t-)$ for each function $f$ and $t\ge 0$. Also set $Y_n^{\ml,*}(t) = t + Y_n^\ml(t)$ so that
    \begin{equation*}
         \sum_{j=V_n(b_nt)+1}^{V_n(b_nt +1 )} d_{(j)}^\ml = \Delta (Y_n^{\ml,\ast} \circ V_n)(b_nt+1).
    \end{equation*}
    Note that by Lemmas \ref{lem:YnmlAndYnmr}, \ref{lem:YconvInfinitThird} $Y_n^\mr$ has a re-centered scaling limit and by Lemmas \ref{lem:cnfinthrid} and \ref{lem:VnFinThird} this limit must be the same as the scaling limit for $V_n = Y_n^\mr + L_n - C_n$. Since
    \begin{equation*}
        a_n^{-1}\left(V_n(b_nt) - \nu_n^\ml b_n t\right)_{t\ge 0}\weakarrow Y_\infty^\mr(t)
    \end{equation*}
    we can use Proposition \ref{prop:ZfinThirdConverge} and Lemma \ref{lem:cnfinthrid} we see
    \begin{equation*}
\left(a_n^{-1}(Y_n^{\ml,\ast}\circ V_n(b_nt) - (\nu_n^\ml+1)b_n t)\right)\weakarrow Z_\infty.
    \end{equation*}
    By \cite[Proposition VI.2.4]{JS.13}, the function $f\mapsto \sup_{s\le T} |\Delta f(s)|$ from $\D\to \R_+$ is continuous at each $f\in \D$ such that $\Delta f(T)  = 0$. Re-centering by a continuous function has no effect on the size of the jumps, and since the limit $Z_\infty$ in Proposition \ref{prop:ZfinThirdConverge} is a.s. continuous at each fixed $T>0$ we see that 
    \begin{equation*}
        \sup_{t\le T} \left|\Delta(Y_n^{\ml,\ast}\circ V_n) (b_nt) \right| \weakarrow \sup_{t\le T} |\Delta Z_\infty(t)|.
    \end{equation*} 
\end{proof}

Looking back at the maximum probability $p(k)$ defined in \eqref{eqn:Pkdef}, we see that
\begin{equation*}
    \sup_{k\le b_nT} p(k) = \frac{ O_\PR(a_n)}{\Theta(n) - O_\PR(b_n)} = O_\PR(b_n^{-1}).
\end{equation*} Therefore, the above lemma has the following useful corollary, which is a big improvement to \eqref{eqn:pboundFinThird}.
\begin{corollary}\label{cor:Pbound}
    Let $p(k)$ be defined via \eqref{eqn:Pkdef} and let $\bd^\ml,\bd^\mr$ satisfy Assumption \ref{ass:asymptProp} or \ref{ass:infintieThird}. Then
    \begin{equation*}
      \{b_n\sup_{k\le b_nT} p(k);n\ge 1\}\textup{ is tight in }\R_+,
    \end{equation*}
\end{corollary}

We now include the following two elementary lemmas. The second follows easily from the first and the fact that $X\sim \Bin(n,p)$ is stochastically dominated by $X'\sim \operatorname{Poi}(2np)$.
\begin{lemma}
    Suppose $X_n, X_n'$ for $n\ge 1$ are non-negative random variables such that $X_n'$ stochastically dominates $X_n$. If $\{X_n';n\ge1\}$ is tight then $\{X_n;n\ge 1\}$ is tight.
\end{lemma}

\begin{lemma}
    Suppose that $N_n, p_n$ are a sequence of random variables such that $\{b_n^{-1} N_n\}$ and $\{b_n p_n\}$ are tight in $\R_+$ for some $b_n\to\infty$. Further, suppose that $N_n$ is integer valued. Conditionally given $N_n,p_n$, let $X_n \sim \Bin(N_n,p_n)$. Then $\{X_n;n\ge 1\}$ is tight.
\end{lemma}

\begin{lemma}\label{lem:surpluscountlim}\label{lem:countlim}
    Under either Assumption \ref{ass:asymptProp} or Assumption \ref{ass:infintieThird}, the sequence of random variables $\{C_n(b_nT);n\ge 1\}$ is tight in $\R_+.$ 
\end{lemma}
\begin{proof}
 Note that $C_n(b_nT)$ is stochastically dominated by
 \begin{equation*}
   X_n\sim   \Bin\left(N_n:=\sum_{j=1}^{b_n T} d_{(j)}^\mr,p_n:= \sup_{k \le b_nT} p(k)\right).
 \end{equation*} By Lemma \ref{lem:YnmlAndYnmr} or \ref{lem:YconvInfinitThird} and Lemma \ref{lem:fluctImplyWLLN}, $\{b_n^{-1}N_n;n\ge 1\}$ is tight. By Corollary \ref{cor:Pbound}, $\{b_n p_n;n\ge 1\}$ is tight. Now apply the previous two lemmas.
\end{proof}

\subsection{Note on Triangle Counts}

Let us examine briefly how a triangle is formed in $\GRIG_n(\bd^\ml,\bd^\mr)$ from the construction of $\GBCM_n(\bd^\ml,\bd^\mr)$ in Exploration \ref{exploration:DFS} above. First for each triangle $\{v_a,v_b,v_c\}$ in $\GRIG_n(\bd^\ml,\bd^\mr)$ there are vertices $w_{ab},w_{bc},w_{ac}\in \cV^\mr$, the $\mr$-vertex set of $\GBCM_n(\bd^\ml,\bd^\mr)$, such that 
\begin{equation*}
    v_{a}\sim w_{ab}, v_b\sim w_{ab}\qquad v_a\sim w_{ac}, v_c\sim w_{ac}\qquad v_{b}\sim w_{bc}, v_c\sim w_{bc}.
\end{equation*}
If $w_{ab} = w_{bc} = w_{ac}$ then we will call the triangle a \textbf{type I} triangle, and otherwise we will call it a \textbf{type II} triangle. We claim that the total triangle count is, essentially, determined by the type I triangles under either Assumption \ref{ass:asymptProp} or \ref{ass:infintieThird}. Let us be more precise.

Let $T_n(t)$ be the number of triangles constructed in the random intersection graph $\GRIG_n(\bd^\ml,\bd^\mr)$ after the exploration of $w_{(t)}^\mr$ in Exploration \ref{exploration:DFS}. Let $T_n^{\operatorname{I}}(t), T_n^{\operatorname{II}}(t)$ be the number of such triangles of type I and type II, respectively.
\begin{lemma}\label{lem:triangleApprox}
It holds that for all $T\ge 0$
    \begin{align*}
    \sup_{t\le T} \left|T_n^{\operatorname{I}}(t) - \sum_{j=1}^{t} \binom{d_{(j)}^\mr}{3} \right| \le \frac{13(d_{1}^\mr)^{2}}{6}  C_n(T)^3 &&\textup{and}&&&
     T_n^{\operatorname{II}}(T) \le C_n(T).
    \end{align*}
    In particular, \begin{equation*}\sup_{t\le T} \left|T_n(t) - \sum_{j=1}^{t} \binom{d_{(j)}^\mr}{3} \right| \le \frac{19(d_{1}^\mr)^{2}}{6}  C_n(T)^3.
\end{equation*}
\end{lemma}

\begin{proof}

Each time in the Exploration \ref{exploration:DFS} we select an $\mr$-vertex $w_{(j)}$ and this vertex will be connected to $d_{(j)}^\mr - c_{(j)}$ many \textit{distinct} $\ml$ vertices, where we recall that $c_{(j)}$ is the number of cycles created by the vertex $w_{(j)}$. Thus $\binom{d_{(j)}^\mr-c_{(j)}}{3}$ many type I triangles are created. Since
\begin{align*}
    \binom{d_{(j)}^\mr-c_{(j)}}{3} &= \binom{d_{(j)}^\mr}{3} + \frac{1}{6}\left((-3 (d_{(j)}^\mr)^{2}+6 d_{(j)}^\mr -2 )c_{(j)} + (3d_{(j)}^\mr -3 )c_{(j)}^2-c_{(j)}^{3}\right)
\end{align*} and $c_{(j)}^3\ge c_{(j)}^2\ge c_{(j)}$ for all $j$, we have for all
\begin{align*}
    &\left|\binom{d_{(j)}^\mr-c_{(j)}}{3} - \binom{d_{(j)}^\mr}{3} \right| \le\frac{1}{6}\left(6 (d_{(j)}^\mr)^2c_{(j)} + 6(d_{(j)}^\mr  )c_{(j)}^2+c_{(j)}^{3}\right)\le \frac{13}{6}(d_{1}^\mr)^2 c_{(j)}^3.
\end{align*}
Therefore for each
$t\ge 0$:
\begin{align*} 
   & \left|T_n^{\operatorname{I}}(t)- \sum_{j=1}^t   \binom{d_{(j)}^\mr}{3} \right|\le \sum_{j=1}^t  \left|\binom{d_{(j)}^\mr-c_{(j)}}{3} - \binom{d_{(j)}^\mr}{3} \right|\le \frac{13}{6} (d_1^\mr)^2 \sum_{j=1}^t c_{(j)}^3  \le \frac{13}{6} (d_1^\mr)^2  C_n(t)^3.
\end{align*}
The first claim follows. For $T_n^{\operatorname{II}}$ we simply observe that each type II triangle must have contained a surplus edge counted by $C_n(T)$.
\end{proof}

\section{Proofs of main result}

In this section we establish the convergence in $\ell^2_\downarrow$ of the connected components of the bipartite configuration model $\GBCM_n(\bd^\ml,\bd^\mr)$. This largely follows the approaches of \cite{Aldous.97, DvdHvLS.17,DvdHvLS.20}, although we do need a little additional work.

\subsection{Good functions}

In this section we establish and recall some results about the $J_1$ topology and convergence of excursions. For an interval $I\subset\R_+$ we let $\D_+(I)$ the collection of c\`adl\`ag functions $f$ such that $f(t)\ge f(t-)$ for all $t\in I$. Recall that $\EE(f)$ is the collection of excursion intervals of the function $f$. Following \cite{DvdHvLS.20}, define $\cY(f) = \{r: (l,r)\in \EE(f)\}$. 
\begin{definition}[Good function on {$[0,T]$}]\label{def:Good[0,T]} We say that a function $f\in \D_+([0,T])$ is \textit{good (on $[0,T]$)} if the following three properties hold
\begin{enumerate}[(i)]
    \item For all $\cY(f)$ does not have any isolated points;
    \item $[0,\sup \cY(f)]\setminus \bigcup_{e\in \EE(f)} e$ has Lebesgue measure $0$;
    \item \label{enum:good[0,t]:iii}$f$ does not attain a local minimum at any point $r\in \cY(f)$.
\end{enumerate}
\end{definition}

The following result is established in \cite[Remark 12]{DvdHvLS.20}.
\begin{lemma}\label{lem:good1}
    Suppose that $f\in \D_+([0,T])$ is good. Then $f$ is continuous at each $r\in \cY(f)$.
\end{lemma}

The following lemma is also easy to see.
\begin{lemma}
    Suppose that $f\in \D_+([0,T])$ is good and let $(l_j,r_j)\in \EE(f)$ for $j\in[2]$ with $l_1<l_2$. Then $r_1<l_2$. In other words, the excursions are not nested.
\end{lemma}
\begin{proof}
    Since $l_1<l_2$ and $(l_1,r_1)$ is an excursion interval, we have \begin{equation*} f (r_1) = \inf _{s\le r_1} f(s) = \inf_{s\le l_1} f(s) \ge \inf_{s\le l_2} f(s),
    \end{equation*}
    where the first equality is valid by Lemma~\ref{lem:good1}. If $r_1 \ge l_2$, then we would have equality above, implying that $r_1$ is a local minimum of $f$ contradicting Definition~\ref{def:Good[0,T]}\eqref{enum:good[0,t]:iii}. 
\end{proof}

\begin{definition}[Good on {$\R_+$}]\label{def:GoodonR}
A function $f\in \D_+ = \D_+(\R_+)$ is \textit{good} on $\R_+$ if:
\begin{enumerate}[(i)]
    \item $f$ is good on $[0,T]$ for each $T>0$;
    \item For all $\eps>0$, the set $\EE(f)$ has only finitely many excursions of length $r-l>\eps$;
    \item $f$ does not possess an infinite length excursion.
\end{enumerate}
\end{definition}

We now include the following deterministic lemma from \cite[Lemma 5.7]{Clancy.24}. 
\begin{lemma}\label{lem:goodPPwithG}
Suppose that $(f_n,g_n)\to (f,g)$ in the product $J_1$ topology and that $f$ is good on $\R_+$ and $g$ is non-decreasing function which is continuous at each $t\in \bigcup_{e = (l,r)\in \EE(f)}\{l,r\}$. Suppose that for each $n$ there is a sequence $(t_{n,i};i\ge 0)$ such that 
    \begin{enumerate}[(i)]
        \item \label{enum:5.5:i}$0=t_{n,0}<t_{n,1}<\dotsm$ and $\lim_{i\to\infty} t_{n,i} = +\infty$ for all $n$;
        \item \label{enum:5.5:ii} $f_n(t_{n,i}) = \inf_{s\le t_{n,i}} f_n(s)$ for all $i$ and all $n$;
        \item \label{enum:5.5:iii} $\max_{i: t_{n,i}\le s_0} f_{n}(t_{n,i})-f_{n}(t_{n,i+1}) \longrightarrow 0$ as $n\to\infty$ for each $s_0<\infty$.
    \end{enumerate} Then
\begin{align*}
    \Xi^{(n)}&:= \left\{(t_{n,i},t_{n,i}-t_{n,i-1}, g_n(t_{n,i})-g_n(t_{n,i-1}));i\ge 1\right\}\\
    &\longrightarrow\Xi^{(\infty)} :=\{(r,r-l,g(r)-g(l)):(l,r)\in \EE(f)\}
\end{align*}
in the vague topology for the associated counting measures on $\R_+\times (0,\infty)\times \R_+$.
\end{lemma}

\subsection{Some basic fluctuation theory}

Let $Z_\infty(t) = Y_\infty^\ml (\nu_\infty^\mr t) +\nu_\infty^\ml Y_\infty^\mr(t) +\lambda t$. These properties should not be surprising to the reader familiar with the fluctuation theory for thinned L\'{e}vy processes in \cite{AL.98}.

\begin{lemma}[Aldous and Limic {\cite[Proposition 14]{AL.98}}]\label{lem:Zpath}
Suppose either (1) both $\kappa,\rho>0$ or (2) $\bbeta\in\ell^{3}_\downarrow\setminus\ell^2_\downarrow$ with $\rho>0$ if $\kappa>0$. Let $\bW(t) = \bW^{\kappa,\rho,\lambda,\bbeta}(t)$. Then
\begin{enumerate}[(a)]
    \item $\bW(t)\to -\infty$ in probability as $t\to\infty$;
    \item $\PR(\bW(t)> \inf_{s\le t} \bW(s) ) = 1$ for all $t> 0$.
    \item $\sup\{r-l: (l,r)\in \EE(\bW); l\ge t\} \to 0$ in probability as $t\to\infty$;
    \item With probability $1$, $\{r: (l,r)\in \EE(\bW)\}$ has no isolated points;
    \item If $(l_1,r_1), (l_2,r_2)\in \EE(\bW)$ and $l_1<l_2$ then $Z(r_2)< Z(r_1).$
\end{enumerate}
\end{lemma} 

The following lemma will also be useful:
\begin{lemma}
    [Lemma 5.7(iii) in \cite{BDW.22}]\label{lem:BDWlemma} Under the assumptions of Lemma \ref{lem:Zpath}, 
    \begin{equation*}
        \PR\left(\bW(t) = \bW(t-)\,\,\forall t\in\bigcup_{(l,r)\in \EE(\bW) }\{l,r\}\right) = 1.
    \end{equation*}
\end{lemma}

The following lemma was observed in \cite{DvdHvLS.20} for the case where $\kappa = 0$; however, the same proof method works here. We omit a repetition of the proof.

\begin{lemma}\label{lem:Zgood}
    Under the assumptions in Lemma \ref{lem:Zpath}, $\bW$ is good on $\R_+$. In particular, $\bW$ is continuous at $r$ for each $r\in\cY_\infty(\bW)$. 
\end{lemma}

Finally, we recall the following result from \cite{DvdHvLS.20} which is listed as Fact 1. Technically, they prove this whenever there is no Brownian component; however, the general result follows from a similar argument found therein. 
\begin{lemma}\label{lem:disctinctExcursionLenghts}
    Let $\bW$ be as in Lemma \ref{lem:Zpath}. Then
    \begin{equation*}
        \PR\left(r_1-l_1  = r_2 - l_2 \textup{ for distinct}(l_1,r_1), (l_2,r_2)\in \EE(\bW)\right) = 0.
    \end{equation*}
\end{lemma}

\subsection{Susceptiblity}

In this section we do some analysis on the size of the large connected components and how early they need to be discovered in the exploration. Key to this is a generalization of some results of Janson \cite{Janson.10} to bipartite graphs.

To begin with we let $P_l$ denote the number of length $l$ paths consisting of distinct vertices in the graph $\GBCM_n(\bd^\ml,\bd^\mr)$. Note that all paths of length $2l$ are of the form
\begin{equation*}
(v_{i_0}^\ml, w_{j_1}^\mr, v_{i_1}^\ml,\dotms, w_{j_l}^\mr, v_{i_l}^\ml)\textup{ or }(w_{j_0}^\mr, v_{i_1}^\ml ,w_{j_1}^\mr,\dotms, v_{i_l}^\ml, w_{j_l}^\mr)
\end{equation*}
for distinct indices $i_0,\dotsm, i_l$ and distinct indices $j_0,j_1,\dotsm, j_l$. Here we will always use the index ``$i$'' for the index of a type $\ml$-vertex while we will use ``$j$'' for the type $\mr$-vertices. Let us call the collection of paths on the left as $P_{2l}^\ml$ and the collection of paths on the right as $P_{2l}^\mr$. We have the following extension of Lemma 5.1 of Janson \cite{Janson.10}.
\begin{lemma}\label{lem:JansonExpan1}
    In the graph $\GBCM_n(\bd^\ml,\bd^\mr)$ with $n$ $\ml$-vertices and $m$ $\mr$-vertices, it holds that for all $l\ge 1$
    \begin{equation*}
        \E\left[P_{2l}^\ml\right] \le n\frac{\E[D_n^\ml(D_n^\ml-1)]^{l-1}}{\E[D_n^\ml]^{l-2}}\frac{\E[D_n^\mr(D_n^\mr-1)]^{l}}{\E[D_n^\mr]^l} =n\frac{\E[D_n^\ml] \E[D_n^\mr(D_n^\mr-1)]}{\E[D_n^\mr]} (\nu_n)^{l-1}
    \end{equation*}
    and
    \begin{equation*}
        \E\left[P_{2l}^{\mr}\right] \le m\frac{\E[D_n^\mr(D_n^\mr-1)]^{l-1}}{\E[D_n^\mr]^{l-2}}\frac{\E[D_n^\ml(D_n^\ml-1)]^{l}}{\E[D_n^\ml]^l}= m\frac{\E[D_n^\mr] \E[D_n^\ml(D_n^\ml-1)]}{\E[D_n^\ml]} (\nu_n)^{l-1}
    \end{equation*}
\end{lemma}

\begin{proof} We only prove the first inequality. The second follows from swapping the vertex sets. The number of ways we can pair half-edges to create the path $(v_{i_0}^\ml, w_{j_1}^\mr, v_{i_1}^\ml,\dotms, w_{j_l}^\mr, v_{i_l}^\ml)$ is
\begin{equation*}
    d_{i_0}^\ml  d_{i_l}^\ml    \prod_{k=1}^{l-1} d_{i_k}  (d_{i_k}^\ml-1) \prod_{k=1}^l d_{j_k}^\mr(d_{j_k}^\mr-1)
\end{equation*}
and each such pairing of half-edges has probability
\begin{equation*}
   \prod_{k=0}^{2l-1} (\|d_n^\mr\|-k)^{-1} = \prod_{k=1}^l (\|\bd^\mr\|_1 - 2k+2)^{-1}(\|\bd^\ml\|_1 - 2k+1)^{-1}
\end{equation*}
as there are $2l$ many edges that need to be paired.

Let us write $\sum^\ast_{i_0,\dotms, i_l, j_1,\dotsm,j_l}$ for the sum over distinct indices $i$ and distinct indices $j$. 
\begin{align*}
    \E[P_{2l}^\ml] &= \frac{\displaystyle \operatornamewithlimits{{\sum}^\ast}_{\substack{i_0,\dotsm,i_l\\ j_1,\dotsm, j_l}} d_{i_0}^\ml  d_{i_l}^\ml \prod_{k=1}^{\ell-1} d_{i_k}  (d_{i_k}^\ml-1) \prod_{k=1}^l d_{j_k}^\mr(d_{j_k}^\mr-1)  }{\displaystyle \prod_{k=1}^l (\|\bd^\mr\|_1 - 2k+2)(\|\bd^\ml\|_1 - 2k+1)}\\
    &= \frac{\displaystyle \operatornamewithlimits{{\sum}^\ast}_{{i_0,\dotsm,i_l}} d_{i_0}^\ml d_{i_l}^\ml  \prod_{k=1}^{\ell-1} d_{i_k}  (d_{i_k}^\ml-1)  }{\displaystyle \prod_{k=1}^l(\|\bd^\ml\|_1 - 2k+1)}\frac{\displaystyle \operatornamewithlimits{{\sum}^\ast}_{ j_1,\dotsm, j_l} \prod_{k=1}^l d_{j_k}^\mr(d_{j_k}^\mr-1)  }{\displaystyle \prod_{k=1}^l (\|\bd^\mr\|_1 - 2k+2)} =: 
    \frac{A_1(l)}{B_1(l)}\frac{A_2(l)}{B_2(l)}\textup{ (say)}.
\end{align*}
These terms are those analyzed in the proof of Lemma 5.1 in \cite{Janson.10}.
More specifically, Jason shows for $R^\ml := \#\{i: d_i^\ml \ge 2\}\le \frac{1}{2} \|\bd^\ml\|_1 = {\frac{1}2}n \E[D_n^\ml]$
\begin{align*}
    A_1(l) &\le (\|\bd^\ml\|_1-2l+1)  (\|\bd^\ml\|_1-2l+3) \operatornamewithlimits{{\sum}^\ast}_{i_1,\dotsm, i_{l-1}} \prod_{k=1}^{l-1} d_{i_k}^\ml(d_{i_k}^\ml-1)\\
    &\le (\|\bd^\ml\|_1-2l+1)  (\|\bd^\ml\|_1-2l+3) \left(n \E[D_n^\ml(D_n^\ml-1)]\right)^{l-1} \prod_{k=1}^{l-2} \left(1-\frac{k}{R^\ml}\right)
\end{align*}
and $B_1(l) \ge (\|\bd^\ml\|_1 -2l+1) (\|\bd^\ml\|_1 - 2l+3) (\|\bd^\ml\|_1)^{l-2} \prod_{k=1}^{l-2}(1-k/R^\ml)$ and so
\begin{equation*}
    \frac{A_1(l)}{B_1(l)} \le \frac{(n \E[D_n^\ml(D_n^\ml-1)])^{l-1}}{(n\E[D_n^\ml])^{l-2}} = n \E[D_n^\ml(D_n^\ml-1)] \left(\frac{\E[D_n^\ml(D_n^\ml-1)]}{\E[D_n^\ml]}\right)^{l-2}.
\end{equation*}
Likewise, the argument of Janson \cite{Janson.10} implies
\begin{equation*}
    A_2(l) \le (m\E[D_n^\mr(D_n^\mr-1)])^{l} \prod_{k=1}^{l-1} \left(1-\frac{k}{R^\mr}\right)
\end{equation*} where $R^\mr := \{i: d_i^\mr\ge 2\}\le\frac{1}{2} \|\bd^\mr\|_1  =\frac{1}{2} m\E[D_n^\mr]$. Similarly,
and
\begin{equation*}
    B_2(l) = \prod_{k=0}^{k-1} (\|\bd\|_1) \left(1-\frac{2k}{\|\bd^\mr\|_1}\right) \ge (m\E[D_n^\mr])^{l} \prod_{k=1}^{l-1} (1-\frac{k}{R^\mr}).
\end{equation*} Thus,
\begin{equation*}
    \frac{A_2(l)}{B_2(l)} \le \left(\frac{\E[D_n^\mr(D_n^\mr-1)]}{\E[D_n^\mr]}\right)^{l}.
\end{equation*} The result follows.
\end{proof}

Given a vertex $x\in \GBCM_n(\bd^\ml,\bd^\mr)$, we write $\cC(x)$ as the connected component containing $x$ and $\cC^\mr(x)$ (resp. $\cC^\ml(x)$) the $\mr$-vertices (resp. $\ml$-vertices) contained in $\cC(x)$. The following lemma is an analog of Lemma 5.2 in \cite{Janson.10}. Recall \eqref{eqn:nuDef}.
\begin{lemma}\label{lem:JansonExpan2}
    Let $V_n^\ml$ be a uniformly chosen $\ml$ vertex and let $V_n^\mr$ denote a uniformly chosen $\mr$ vertex. If $\nu_n<1$ then
    \begin{align*}
        &\E\left[\#\cC^\mr(V_n^\mr)\right] \le  1+ \frac{\E[D_n^\mr]\E[D_n^\ml(D_n^\ml-1)]}{\E[D_n^\ml] (1-\nu_n)}=1+\frac{\nu_n^\ml \E[D_n^\mr]}{1-\nu_n},\\
        &\E\left[\#\cC^\ml(V_n^\ml)\right] \le  1+ \frac{\E[D_n^\ml]\E[D_n^\mr(D_n^\mr-1)]}{\E[D_n^\mr] (1-\nu_n)} =1 + \frac{\nu_n^\mr \E[D_n^\ml]}{1-\nu_n}.
    \end{align*}
\end{lemma}
\begin{proof}
    As in the proof of \cite[Lemma 4.6]{JR.12}, for any $\mr$-vertices $v$ and $w$ with $w\in\cC(v)$ there must be a path from $v$ to $w$ of even length. Hence all the paths from $v$ to $w$ are in $P_{2l}^\mr$. Since $\PR(V_n^\mr = v) = \frac{1}{m}$ for each $\mr$ vertex $v$, we see
    \begin{equation*}
        \E[\#\cC^\mr(V_n^\mr)] \le 1 + \frac{1}{m}\sum_{l=1}^\infty \E[P_{2l}^\mr] \le 1+\frac{\E[D_n^\mr]\E[D_n^\ml(D_n^\ml-1)]}{\E[D_n^\ml] } \sum_{l=1}^\infty \nu_n^{l-1}.
    \end{equation*}The second argument is completely analogous. 
\end{proof}

\subsection{Large Components are Discovered Early}

Given Lemmas \ref{lem:JansonExpan1} and \ref{lem:JansonExpan2}, we can show that the large connected components of $\GBCM_n(\bd^\ml,\bd^\mr)$ are, in some sense, discovered early in the exploration. More precisely, we will let $(\cC^{\ge T}_n{{(i)}};i\ge 1)$ be the connected components of $\GBCM_n(\bd^\ml,\bd^\mr)$ which are discovered after time $b_nT$, listed in the order in which they are discovered. We let $\cC^{\mr,\ge T}_n{(i)}$ be the collection of $\mr$-vertices in $\cC^{\ge T}_n{(i)}$ and similarly denote $\cC_n^{\ml,\ge T}(i)$. The main device is the following
\begin{lemma}\label{lem:latecompon}\label{lem:l2tight} Fix $\delta>0$. 
    Under either Assumption \ref{ass:asymptProp} or Assumption \ref{ass:infintieThird}, 
    \begin{equation*}
        \lim_{T\to \infty}\limsup_{n\to\infty} \PR\left(\sum_{i =1}^\infty \left(\#\cC_n^{\mr,\ge T}(i) \right)^2 + \left(\#\cC^{\ml,\ge T}_n(i)\right)^2> \delta^2 b_n^2 \right) = 0.
    \end{equation*}

\end{lemma}

The proof of this lemma follows the proof strategy of Lemmas 5.12 in \cite{DvdHvLS.17} and Lemma 13 in \cite{DvdHvLS.20}. Namely, we show that once a large number of $\mr$-vertices are explored, the unexplored portion of the graph becomes sub-critical. Let us now make this rigorous.

Recall that $(d_{(j)}^\ml;j\ge 1)$ and $(d_{(j)}^\mr;j\ge 1)$ are the degrees of the vertices $v^\ml_{(j)}, w^\mr_{(j)}$ and are listed in the order in which they appear in Exploration \ref{exploration:DFS}.  Let $\sV_T^\ml$ (resp. $\sV^\mr_T$) denote the collection of $\ml$-vertices (resp. $\mr$-vertices) discovered up through step $\lfloor{b_nT}\rfloor$ in Exploration \ref{exploration:DFS}. Note that
\begin{equation*}
    \sV^\ml_T = \{v_{(i)}^\ml: i\le V_n(b_nT)\}\qquad \textup{and}\qquad \sV^\mr_T = \{w_{(i)}^\mr: i\le b_nT\}.
\end{equation*}
Define
\begin{equation*}
    \nu_n(T) = \frac{\sum_{i\notin \sV^\ml_T} d_{i}^\ml (d_i^\ml-1)}{\sum_{i\notin \sV^\ml_T} d_i^\ml}\frac{\sum_{i\notin \sV^\mr_T} d_{i}^\mr (d_i^\mr-1)}{\sum_{i\notin \sV^\mr_T} d_i^\mr}
\end{equation*}
which we think of as the criticality parameter of the unexplored part of the graph $\GBCM_n(\bd^\ml,\bd^\mr)\setminus(\sV^\ml_T\cup \sV^\mr_T)$.

\begin{lemma}\label{lem:subcrit}
    Under either Assumption \ref{ass:asymptProp} and \ref{ass:infintieThird}, there is a non-decreasing and unbounded stochastic process $X$ such that 
    \begin{equation*}
        \left(c_n(\nu_n(t) - \nu_n)\right)_{t\ge 0}\weakarrow \left(- X(t);t\ge 0\right) 
    \end{equation*}
\end{lemma}
\begin{proof}
    Let us look at the $\mr$-term in the definition of $\nu_n(t)$. We have
    \begin{align*}
        &\frac{\sum_{i\notin \sV^\mr_t} d_{i}^\mr (d_i^\mr-1)}{\sum_{i\notin \sV^\mr_t} d_i^\mr} = \frac{n\E[D_n^\mr(D_n^\mr-1)] - \sum_{j=1}^{b_nt} d_{(j)}^\mr(d_{(j)}^\mr-1)}{n\E[D_n^\mr] - \sum_{j=1}^{b_nt} d_{(j)}^\mr} \\
        &=\left(\nu_n^\mr- \frac{1}{n\E[D_n^\ml]} \sum_{j=1}^{b_nt} d_{(j)}^\mr(d_{(j)}^\mr-1)\right)  \left(1 + \frac{1}{n\E[D_n^\ml]}\sum_{j=1}^{b_nt} d_{(j)} + O\left(n^{-1}\sum_{j=1}^{b_nt} d_{(j)}\right) \right).
    \end{align*}
    By Lemmas \ref{lem:BSWlemma} and \ref{lem:secondMOme} we note that for uniformly for $t\le T$
    \begin{align*}
      &n^{-1}  \sum_{j=1}^{b_nt} d_{(j)}^\mr (d_{(j)}^\mr-1)  = O_\PR(a_n^2/n) = O_\PR(c_n^{-1}) &\textup{and}&&&n^{-1}  \sum_{j=1}^{b_nt} d_{(j)}^\mr  = O_\PR(b_n/n) = o_\PR(c_n^{-1}).
    \end{align*}
    Thus
    \begin{align*}
        &\frac{\sum_{i\notin \sV^\mr_t} d_{i}^\mr (d_i^\mr-1)}{\sum_{i\notin \sV^\mr_t} d_i^\mr} = \nu_n^\mr - \frac{1}{c_n \E[D_n^\mr]} \frac{1}{a_n^2} \sum_{j=1}^{b_nt} d_{(j)}^\mr(d_{(j)}^\mr-1) + o_{\PR}(c_n^{-1}).
    \end{align*}
    Similarly, one can see that 
    \begin{equation*}
        \frac{\sum_{i\notin \sV^\ml_t} d_{i}^\ml (d_i^\ml-1)}{\sum_{i\notin \sV^\ml_t} d_i^\ml} = \nu_n^\ml - \frac{1}{c_n \E[D_n^\ml]} \frac{1}{a_n^2} \sum_{j=1}^{V_n(b_nt)} d_{(j)}^\ml(d_{(j)}^\ml-1) + o_{\PR}(c_n^{-1})
    \end{equation*}
    where the $o_{\PR}(c_n^{-1})$ term holds for all $t\le T$.

    Combining these last two displayed equations gives
    \begin{align*}
        -c_n(\nu_n(t)-\nu_n) 
        &=\frac{\nu_n^\mr}{\E[D_n^\ml]} \frac{1}{a_n^2}\sum_{j=1}^{V_n(b_nt)} d_{(j)}^\ml(d_{(j)}^\ml-1) \\
&\qquad\qquad + \frac{\nu_n^\ml}{\E[D_n^\mr]} \frac{1}{a_n^2}\sum_{j=1}^{b_nt} d_{(j)}^\mr(d_{(j)}^\mr-1)+o_\PR(1).
    \end{align*}
It is easy to see from stated weak convergence in Lemmas \ref{lem:BSWlemma}, \ref{lem:secondMOme}, and \ref{lem:VnFinThird}. Indeed
\begin{align*}
  \left(  \frac{\nu_n^\mr}{\E[D_n^\ml]} \frac{1}{a_n^2}\sum_{j=1}^{V_n(b_nt)} d_{(j)}^\ml(d_{(j)}^\ml-1)\right)_{t\ge 0}\weakarrow (\frac{\nu_\infty^\mr}{\mu_1^\ml}X^\ml( \nu_\infty^\mr t);t\ge 0) \\
\left(\frac{\nu_n^\ml}{\E[D_n^\mr]} \frac{1}{a_n^2}\sum_{j=1}^{b_nt} d_{(j)}^\mr(d_{(j)}^\mr-1)\right)_{t\ge 0}\weakarrow \left(\frac{\nu_\infty^\mr}{\mu_1^\mr}X^\mr(t);t\ge0\right)
\end{align*} where $X^\ml, X^\mr$ are  of the respective limits in Lemmas \ref{lem:BSWlemma} or \ref{lem:secondMOme}:
\begin{align*}
    &X^\ml(t) = \begin{cases}
   \displaystyle \frac{\mu_3^\ml-\mu_2^\ml}{\mu_2^\ml}t&:\textup{ under Assumption \ref{ass:asymptProp}}\\
       \displaystyle \sum_{j=1}^\infty (\beta^\ml_j)^2 \bone_{[\eta_j^\ml\le t/\mu_1^\ml]} &:\textup{ under Assumption \ref{ass:infintieThird}} 
    \end{cases}\\
    &X^\mr(t) = \begin{cases}
    \displaystyle\frac{\mu_3^\mr-\mu_2^\mr}{\mu_2^\mr}t&:\textup{ under Assumption \ref{ass:asymptProp}}\\    
       \displaystyle {\theta^{\frac{-2}{\tau-1}}}\sum_{j=1}^\infty(\beta^\mr_j)^2 \bone_{[\eta_j^\mr\le \theta^{-\frac{\tau-2}{\tau-1}} t/\mu_1^\mr]}&:\textup{ under Assumption \ref{ass:infintieThird}}
    \end{cases}
\end{align*}Now, a.s. the limits $X^\mr, X^\ml$ are continuous under Assumption \ref{ass:asymptProp}. Under Assumption \ref{ass:infintieThird} they are easily seen to satisfy
\begin{equation*}
    \PR\left( X^\ml(t)\neq X^\ml (t-) \textup{ and } X^\mr(s) \neq  X^\mr(s-)\textup{ for some } t = \nu_\infty^\mr s\right) = 0.
\end{equation*} In either case, they live on the subset of $\D\times\D$ such that $(f,g)\mapsto f+g$ is continuous (see \cite{Whitt.80}) and thus the desired weak convergence follows.

The fact that $X$ is unbounded follows from the fact that $\mu_3^\mr,\mu_1^\mr>0$ under Assumption \ref{ass:asymptProp}. Under Assumption \ref{ass:infintieThird}\textbf{(v)} we can see that
\begin{equation*}
   \lim_{t\to\infty} X^\mr(t) + X^\ml(t) = \|\bbeta^\mr\|_2^2 +  \|\bbeta^\ml\|_2^2 = +\infty.
\end{equation*} and so $X(t)$ is easily see to be unbounded.
\end{proof}

\begin{proof}[Proof of Lemma \ref{lem:latecompon}] We follow the method in Lemma 5.12 in \cite{DvdHvLS.17}. Let $Z$ be the weak limit of the exploration process $Z_n$ obtained in Proposition \ref{prop:ZfinThirdConverge}, under the appropriate assumption.

For each $T$, let $S_T = \inf\{t\ge b_n T: Z_n(b_n t) < \inf_{s\le b_nT} Z_n(s)\}$ be the first time that we start exploring a new connected component of $\GBCM_n(\bd^\ml,\bd^\mr)$ after time $ b_nT$ in Exploration \ref{exploration:DFS}. It is easy to see that the unexplored portion of the graph $\GBCM_n(\bd^\ml,\bd^\mr)$ after time $Tb_n$, which we will denote by $\overline{\G}_n (T)$ is still a bipartite configuration model with degree sequences
\begin{equation*}
\overline{\bd}^\ml = (d_j^\ml: v_j^\ml\notin \sV_{S_T}^\ml)\qquad\textup{and}\qquad \overline{\bd}^\mr = (d_j^\mr: w_j^\mr\notin \sV_{S_T}^\mr).
\end{equation*} Observe that $\overline{\G}_n(T)$ has criticality parameter $\nu_n(S_T)$.

Observe that by Lemma \ref{lem:Zpath}(a,c) that for each $\eps>0$ we can find a $T_{1} = T_{1}(\eps)$ sufficiently large such that
\begin{equation*}
    \PR\left(\inf_{s\le 2T} Z(s)< \inf_{s\le T}Z (s) \right) > 1-\eps\qquad\textup{ for all }T\ge T_1.
\end{equation*} In words, this means that with large probability that $S_T\in[T,2T]$ for $T$ sufficiently large.

Lemma \ref{lem:subcrit} implies that for any $C>0$, $\eps>0$ there exists a $T_2 = T_{2}(\eps,C) \ge T_1(\eps)$ large enough so that 
\begin{equation*}
   \liminf_{n\to\infty} \PR(\nu_n(t) \le \nu_n  -  C t c_n^{-1} \textup{ for all }t\in[T,2T]) \ge 1-\eps.
\end{equation*}
In particular, this implies that for any constant $C_0$ and $\eps>0$ there exists another time $T_3(\eps,C_0)\ge T_2(\eps, C_0+|\lambda|)$ such that
\begin{equation*}
    \liminf_{n\to\infty}  \PR(\nu_n(t) \le 1  -  C_0 t c_n^{-1} \textup{ for all }t\in[T,2T]) \ge 1-\eps\qquad\forall T\ge T_3.
\end{equation*} 
Given $t_1<t_2$ let $\operatorname{Disc}(t_1,t_2)$ denote the event that a new connected component of $\GBCM_n(\bd^\ml,\bd^\mr)$ is discovered between time $t_1b_n$ and $t_2b_n$ in the Exploration \ref{exploration:DFS}. Now note that for all $T\ge T_3$ it holds that
\begin{equation*}
  \liminf_{n\to\infty}  \PR\left(\{\nu_n(t)\le 1- C_0 tc_n^{-1} \,\,\forall t\in[T,2T]\}\cap \operatorname{Disc}(T,2T)\right) \ge 1-2\eps\qquad\textup{ for all }T\ge T_3.
\end{equation*} Let us call $E_n(T) = \{\nu_n(t)\le 1- C_0 tc_n^{-1} \,\,\forall t\in[T,2T]\}\cap \operatorname{Disc}(T,2T)$.

Let us now keep $\eps,\delta,C_0>0$ fixed. Given $T\ge T_3$ deterministic, let $\overline{V}_n^{\mr;T}$ denote uniformly chosen $\mr$-vertex in the residual configuration model $\overline{\G}_n$ defined above. 
For $T\ge T_3$ we have
\begin{align*}
\PR&\left(\sum_{i=1}^\infty \left(\#\cC_n^{\mr,\ge T}(i)\right)^2> \delta^2 b_n^2\right)\le2\eps +  \frac{1}{\delta^2 b_n^2}  \E\left[\sum_{i=1}^\infty \left(\#\cC_n^{\mr,\ge T}(i)\right)^2 \bone_{E_n(T)}\right]\\
&\le2 \eps + \frac{n}{\delta^2b_n^2} \E\left[\#\overline{\cC}_{n}^{\mr}(V_n^{\mr;\ge T})\bone_{E_n(T)}\right]\\
&\le2 \eps + \frac{n}{\delta^2b_n^2} \frac{O(1)}{C_0 T c_{n}^{-1} } = 2\eps + \frac{O(1)}{C_0 \delta^2 T},
\end{align*}
where going from the second line to the third we used Lemma \ref{lem:JansonExpan2} for the residual graph $\overline{\G}_n(T)$, which has criticality parameter $\nu_n(S_T)\le 1- C_0Tc_n^{-1}$ on $E_n(T)$. This implies that for all $\eps>0$ that 
\begin{equation*}
    \limsup_{T\to\infty} \limsup_{n\to\infty} \PR\left(\sum_{i=1}^\infty \left(\#\cC_n^{\mr,T}(i)\right)^2> \delta^2 b_n^2\right) \le 2\eps.
\end{equation*}
The stated result easily follows by selected an $\ml$ vertex instead of an $\mr$ vertex and obtaining an analogous bound for $\sum_{i} (\#\cC_n^{\ml,T}(i))^2$. We leave the details to the reader.
\end{proof}

\subsection{Proof of the Theorems \ref{thm:MAIN1} and \ref{thm:MAIN2}}

We can now turn to the proof of Theorems \ref{thm:MAIN1} and \ref{thm:MAIN2}. The key technical device is the following proposition from \cite{Clancy.24}.
\begin{proposition}\label{prop:WeakConvPointProcess}
Let $\Xi^{(n)}\subset \R_+\times(0,\infty)\times \R_+$, for $1\le n\le \infty$, are locally finite point sets as in Lemma~\ref{lem:goodPPwithG} for some random functions $(F_n,G_n)$. Suppose that
\begin{enumerate}[(i)]
\item\label{enum:5.11.i} $(F_n,G_n)\weakarrow (F,G)$ where $F$ is a.s. good and $G$ a.s. continuous at each excursion endpoint of $F$.
    \item \label{enum:5.11.ii}  For all $1\le n\le \infty$, we can write $ 
    \Xi^{(n)} = \{ (t_i^{(n)},x_i^{(n)},y_i^{(n)});i\ge1\}
$ where $\{t_i^{(n)};i\ge 1\}$ are all distinct,  $i\mapsto x_i^{(n)}$ is non-increasing with $i$, and if $x_i^{(n)} = x_j^{(n)}$, then $t_i^{(n)}< t_j^{(n)}$;
\item \label{enum:5.11.iii}For all $\eps>0$
\begin{equation}\label{eqn:tailSumBoundProp5.10}
    \lim_{T\to\infty} \limsup_{n\to\infty}\PR\left( \sum_{i: t_i^{(n)}> T} (x_i^{(n)})^2+(y_i^{(n)})^2 > \eps\right) = 0.
\end{equation}
\item \label{enum:5.11.iv}Almost surely the Stieljes measure $G(dt)$ does not charge the excursion end points of $F$; i.e. 
    \begin{equation*}
        \int_{0}^\infty  1_{[F(t) = \inf_{s\le t} F(s) ]} \, G(dt) = 0.
    \end{equation*}

\end{enumerate}
If $x_1^{(\infty)}>x_2^{(\infty)}>\dotsm > 0$ a.s., then 
\begin{equation}\label{eqn:l2pointprocess}
    \ORD(\Xi_n) = \left((x_i^{(n)},y_i^{(n)});i\ge 1\right) \weakarrow \left((x_i^{(\infty)}, y_i^{(\infty)});i\ge 1\right)\qquad\textup{in }\ell^{2,1}.
\end{equation}
\end{proposition}

We now turn to the proofs of Theorem \ref{thm:MAIN1} and \ref{thm:MAIN2}. For each $n\ge 1$, let 
\begin{equation*}
    \Xi_n = \left\{\Big( \overline{\tau}_n(j) , \overline{\tau}_n(j)- \overline{\tau}_n(j-1)  , \overline{V}_n(\overline{\tau}_n(j))-\overline{V}_n(\overline{\tau}_n(j-1))\Big);j\ge 1\right\}
\end{equation*}
where $\overline{\tau}_n(j) = \inf\{t: Z_n(b_nt) = -(1+\nu_n^\ml) j\}$. By Lemma \ref{lem:pointProcessLemma}, and a trivial scaling, we see
\begin{equation*}
    \ORD(\Xi_n) \overset{d}{=} \left(b_n^{-1}\big(\#\cC_n^\mr(j),\#\cC_n^\ml(j)\big);j\ge 1\right),
\end{equation*}
where we note the order of the components is different from those in Theorems \ref{thm:MAIN1} and \ref{thm:MAIN2}. Observe that by Proposition \ref{prop:ZfinThirdConverge} and Lemma \ref{lem:VnFinThird}
\begin{equation*}
    \left(a_n^{-1}Z_n(b_nt), \overline{V}_n(t);t\ge 0\right)\weakarrow (Z_\infty,\nu^\mr_\infty \operatorname{Id})
\end{equation*} where, using Corollary \ref{cor:rewrite} 
\begin{equation*}
Z_\infty(t) = \begin{cases}
\nu^l_\infty\bW^{\kappa,\rho,\lambda^\mr_\infty,\bzer}(t)&:\textup{under Assumption \ref{ass:asymptProp}}\\
\nu^\ml_\infty\mu_1^\ml\bW^{0,0,\lambda\nu_\infty^\mr/\mu_1^\ml,\bbeta}(t)&:\textup{under Assumption \ref{ass:infintieThird}}
\end{cases} .
\end{equation*} 

The main theorems now follow from an application of Proposition \ref{prop:WeakConvPointProcess} provided that we can verify  the assumptions. Clearly, Assumption \ref{prop:WeakConvPointProcess}\eqref{enum:5.11.i} follows from Lemma \ref{lem:Zgood} and the fact that the limiting process $V_\infty$ is a.s. continuous. Assumption \ref{prop:WeakConvPointProcess}\eqref{enum:5.11.ii} is simply an ordering lemma. Next, Assumption \ref{prop:WeakConvPointProcess}\eqref{enum:5.11.iii} is precisely the conclusion of Lemma \ref{lem:l2tight}.  Moreover, \ref{prop:WeakConvPointProcess}\eqref{enum:5.11.iv} holds as $Z_\infty$ is good and the limiting Stieljes measure is absolutely continuous with respect to the Lebesgue measure. Lastly, the limits are a.s. distinct by Lemma \ref{lem:disctinctExcursionLenghts}.

\subsection{Proofs of Propositions \ref{prop:finThirdMomentTriang} and \ref{prop:infTriangl}}

The proofs of the triangle counts are quite standard, so we briefly sketch the argument under Assumption \ref{ass:infintieThird}. First, by Lemmas \ref{lem:TriangleInfnite}, \ref{lem:countlim}, \ref{lem:triangleApprox} and Proposition \ref{prop:ZfinThirdConverge} we have, jointly in $\D(\R_+,\R)$
\begin{align*}
    \left(a_n^{-1} Z_n(b_nt)\right)_{t\ge 0}\weakarrow X\qquad\textup{and}\qquad \left(a_n^{-3} T_n(b_nt)\right)_{t\ge 0} \weakarrow T
\end{align*}
where $X$ and $T$ are as in Proposition \ref{prop:infTriangl}. 

Let 
\begin{equation*}\Xi_n = \left\{\Big( \overline{\tau}_n(j) , \overline{\tau}_n(j)- \overline{\tau}_n(j-1)  , a_n^{-3} T_{n}(b_n\overline{\tau}_n(j))- a_n^{-3} T_n(b_n\overline{\tau}_n(j-1))\Big);j\ge 1\right\}
\end{equation*} where, again, $\overline{\tau}_n(j) = \inf\{t: Z_n(b_nt) = -j\}$. Using Proposition \ref{prop:WeakConvPointProcess},
\begin{equation*}
    \Xi_n\weakarrow \Xi:=\{(r,r-l,T(r)-T(l)); (l,r)\in \EE(X)\}
\end{equation*} since $T$ is a.s. continuous at each excursion endpoint of $X$ by Lemma \ref{lem:BDWlemma}. Here the convergence is in the sense of the vague topology of the associated counting measure. 

By Lemma \ref{lem:l2tight} and Lemma \ref{lem:disctinctExcursionLenghts}, one can easily see that
$
    \ORD(\Xi_n) \weakarrow \ORD(\Xi)
$ with respect to the product topology, as desired.

\appendix
\section{Fluctuations of size-biased partial sums}\label{sec:GenWeak}

Key to the prior analysis was understanding the fluctuations of size-biased partial sums. In this appendix we prove the results. Suppose that $\bw$ is a weight sequence with $\bw = (w_1,w_2,\dotsm, w_N,0,\dotsm)$ where $w_j> 0$ for $j \in[N]$. In addition, suppose that $\bu = (u_1,u_2,\dotsm, u_N, 0,0,\dotsm)$ is another sequence on $[N]$, where $u_j \ge 0$ for all $j\in[N]$. Write $\sigma_p = \sigma_p(\bw)$ and $\sigma_{p,q} = \sigma_{p,q}(\bw,\bu)$ as
\begin{equation*}
    \sigma_p = \sum_{j=1}^N w_j^p\qquad \textup{and} \qquad \sigma_{p,q} = \sum_{j=1}^N w_j^p u_j^q,\qquad \forall p,q\ge 0.
\end{equation*} Note that $\sigma_0(\bw) = N$.

Suppose that $\bw = (\bw^{(n)};n\ge 1)$, $\bu = (\bu^{(n)};n\ge 1)$ are sequences with $N_n = \sigma_0(\bw)\to \infty$. Suppose that there exists $\bbeta = (\beta_j;j\ge 1)\in \ell^3$ and $\widetilde{\bbeta} = (\tbeta_j;j\ge 1)\in \ell^3$, as well as constants $\varsigma_{a,b}\ge 0$ for $a,b\in\{0,1,2,3\}$ with $a+b\le 3$, $\alpha\ge 0$, $\lambda\in\R$, and two sequences $\eps = \eps_n\to 0, \gamma_n\to \gamma\in(0,\infty)$ such that
    \begin{subequations}
        \begin{align}\label{eqn:Weigh1}
      &\frac{(w_j,u_j)}{\sigma_2}\to (\beta_j,\tbeta_j)\,\,\,\forall j\ge 1, & & \frac{\sigma_{a,b}}{\sigma_2^3}\to \varsigma_{a,b} + \sum_{j=1}^\infty \beta_j^a \tbeta_j^{b}\,\,\,\, \forall a+b = 3,\\
     \label{eqn:Weigh2} & \sigma_2 \to 0 && \frac{\sigma_1\eps}{\sigma_2} = \gamma_n +\lambda \sigma_2 + o(\sigma_2) \\
     \label{eqn:Weigh3} &\frac{\sigma_{a,b}}{\sigma_2} \to \varsigma_{a,b},\,\,\,\forall a+b = 2&& \frac{\eps}{\sigma_2^2}\to \alpha.
        \end{align}
    \end{subequations} Note that this implies that $\sigma_1 \to + \infty$ which is typically the case with the extremal eternal versions of the multiplicative coalescence \cite{AL.98}. One reason why we include $\gamma_n$ in \eqref{eqn:Weigh2} is that we do not have require any rate of convergence of $\E[D_n^\ml(D_n^\ml-1)]/\E[D_n^\ml]$ in Assumption \ref{ass:asymptProp}.

Given such a weight sequence $\bw = (\bw^{(n)};n\ge 1)$ with $N = N_n = \sigma_0(\bw^{(n)})$ define $\pi\in \fS_{N}$ as a size-biased permutation on $N$ many letters with law
\begin{equation}\label{eqn:piSizeBiasDef}
    \PR(\pi_j = i_j\textup{ for all }j\in[N]) = \prod_{j=1}^N \frac{w_{i_j}}{\sum_{l\ge j} w_{i_l} }.
\end{equation} Let $(u_{(j)};j\ge 1)$ denote the random re-arrangement $u_{(i)} = w_{\pi_i}$ for all $i\in[N]$. 
\begin{theorem}    \label{thm:SIZEBIASPARTIALSUMS} Suppose that $(\bw,\bu) = ((\bw^{(n)},\bu^{(n)});n\ge 1)$ are weight sequences which satisfy \eqref{eqn:Weigh1}--\eqref{eqn:Weigh3}. Then
\begin{align*}
    &\left(\frac{1}{\sigma_2} \left( \sum_{j=1}^{\eps^{-1}_n t} u_{(j)} - \frac{\sigma_{1,1}}{\sigma_2\gamma_n} t\right);t \ge 0\right)\\
    &\weakarrow \bigg(\sum_{j=1}^\infty \tbeta_j \left(\bone_{[\eta_j\le  t/\gamma]} - \frac{\beta_j}{\gamma} t\right)+ \sqrt{\kappa} W(t) - \frac{\rho}{2}t^2  -\varsigma_{1,1}\mu t;t\ge 0\bigg)
\end{align*}
where $\rho = \frac{\varsigma_{2,1}\gamma - \varsigma_{1,1}\alpha}{\gamma^3}$, $\kappa = \frac{\varsigma_{1,2}\gamma-\varsigma_{1,1}^2\alpha}{\gamma^{2}}$, $W$ is a standard Brownian motion and $\eta_j$ are independent $\Exp(\beta_j)$ random variables that are also independent of $W$.
\end{theorem}

The proof of this theorem is delayed until Section \ref{ssec:proofOfPartial}. Instead of proving the result directly, we use a standard trick of relating size-biased random variables to the order statistics of independent exponential random variables and then some classical results from diffusive limits of queues.

\subsection{General Weak Convergence Result}
The purpose of this section is to prove an approximation result which involves some additional randomness. Similar ideas can be found in \cite{BSW.17,Aldous.97,AL.98,BvdHvL.10} as well as the Poissonization and de-Poissonization trick in \cite{Joseph.14}. 
\begin{proposition}\label{prop:generalWeakConv}
    Suppose $(\bw,\bu) = ((\bw^{(n)},\bu^{(n)});n\ge 1)$ satisfies \eqref{eqn:Weigh1}--\eqref{eqn:Weigh3}. Let $\xi_j\sim \Exp(w_j)$ be independent random variables. Set  
    \begin{equation*}
        X_n(t) = \sum_{j=1}^N u_j \bone_{[\xi_j\le \sigma_2^{-1}t]} \qquad\textup{and}\qquad F_n(t) = \sum_{j=1}^{N} \eps\bone_{[\xi_j\le \sigma_2^{-1} t]}
    \end{equation*}
    where $\xi_j \sim \Exp(w_j)$ are independent exponential random variables. 
    Then in the $J_1$ topology of $\D(\R,\R^2)$,
    \begin{align}\label{eqn:XFlim}
      & \left( \left(\sigma_2^{-1}(X_n(t)-\frac{\sigma_{1,1}}{\sigma_2}t), \sigma_2^{-1} (F_n(t)- \gamma_n t) \right);t\ge 0\right)\\
      \nonumber &\weakarrow \left( \left(\sum_{j=1}^\infty \tbeta_j (\bone_{[\eta_j\le t]} - \beta_jt) + B_1(t) - \frac{\varsigma_{2,1}}{2}t^2  , B_2(t)-\frac{\alpha}{2} t^2 + \lambda t\right);t\ge 0\right)
    \end{align}
    where $B = (B_1,B_2)$ is an $\R^2$-valued (non-standard) Brownian motion with quadratic covariation $\langle B_1 \rangle (t) = \varsigma_{1,2} t, \langle B_1,B_2\rangle (t) =\alpha \varsigma_{1,1} t, \langle B_2\rangle (t) = \alpha \gamma t$.
\end{proposition}
\begin{proof} We follow the method of proof of Limic \cite{Limic.19} (with the majority of the proof in the supplement \cite{Limic.19_Supplement}). 
We write $Z_n(t) = (Z_n^1(t),Z_n^2(t))^T$ as
\begin{align*}
    Z_n^{(1)}(t) &= \sigma_2^{-1} (X_n(t)- \frac{\sigma_{1,1}}{\sigma_{2}}t) = \sum_{j=1}^N \sigma_2^{-1} u_j (\bone_{[\xi_j\le \sigma_2^{-1}t]} - \sigma_2^{-1} w_jt) \\
    Z^{(2)}_n(t) &= \sigma_2^{-1} (F_n(t) - \gamma_n t) = \sum_{j=1}^N \eps\sigma_2^{-1}(\bone_{[\xi_j\le \sigma_2^{-1}t]} - \sigma_2^{-1}w_jt) + (\eps \sigma_2^{-2} \sigma_1 -\sigma_2^{-1} \gamma_n )t\\
    &=\sum_{j=1}^N \eps\sigma_2^{-1}(\bone_{[\xi_j\le \sigma_2^{-1}t]} - \sigma_2^{-1}w_jt) + \lambda  t + o(1)
\end{align*} where the $o(1)$ term is locally uniform in $t$. Observe that we can re-write $Z_n$ as
\begin{align}\label{eqn:Zapprox}
    Z_n(t) = \sum_{j=1}^N (M_{j,n}(t) - A_{j,n}(t)) + \begin{bmatrix}
        0\\\lambda 
    \end{bmatrix}t + o(1)
\end{align} where $M_j = M_{j,n}$ and $A_j = A_{j,n}$ are
\begin{align*}
    M_j(t) &= \begin{bmatrix}
        \sigma_2^{-1} u_j\\
        \sigma_2^{-1} \eps
    \end{bmatrix} \left(\bone_{[\xi_j\le \sigma_2^{-1}t]} - \sigma_2^{-1} w_j t\right)\bone_{[\xi_j\le \sigma_2^{-1}t]}&
    A_j(t) &= \begin{bmatrix}
        \sigma_2^{-2} u_jw_j\\
        \sigma_2^{-2}\eps  w_j
    \end{bmatrix} (t-\xi_j/q)_+. 
\end{align*}Note that $(M_j,A_j)_{j\in[N]}$ are independent over the index $j$ and that $M_j$ are martingales with ($\R^{2\times2}$-valued) quadratic covariation
\begin{equation*}
    \langle\langle M_j\rangle\rangle (t)= \begin{bmatrix}
        u_j^2 &   \eps u_j\\
      \eps u_j&   \eps^2
    \end{bmatrix} \sigma_2^{-3} w_j (t\wedge (\xi_j\sigma_2)).
\end{equation*} We us $\langle\langle-\rangle\rangle$ to mean the matrix-valued quadratic covaration and use $\langle-,-\rangle$ to mean the real-valued quadratic covariation as in \cite{Whitt.07}. 
It is convenient to note the following expectation bounds
\begin{subequations}
    \begin{align}
        \label{eqn:ANodiff}&\E[(t-\xi_j/q)_+] \le \E[t \bone_{\xi_j\le \sigma_2^{-1}t}] \le \sigma_2^{-1} w_jt^2\\
        \label{eqn:twedgeExp} &\E[t\wedge (\xi_j/q)] = t + O(1) \sigma_2^{-1} w_j t^2
    \end{align} where the $O(1)$ term is independent of $t\ge 0$, $j\ge 1$, $n\ge 1$.
\end{subequations}

Following \cite{AL.98}, one can easily check that we can find a sequence $m = m(n)\to\infty$ slowly enough so that for all $a,b\in\{0,1,2,3\}$ with $a+b\le 3$ that
\begin{subequations}
\begin{align}
&   \left| \sum_{j=1}^{m} \left(\sigma_2^{-(a+b)} w_j^a u_j^b - \beta_j^a\tbeta_j^b\right)\right|\to 0&
& \left|\sum_{j=1}^m \left(\sigma_2^{-1} w_j - \beta_j\right)^3\right|\to 0   \label{eqn:mtoinfincond1}\\
&\sigma_1^{-1} \sum_{j=1}^m w_j \to 0\label{eqn:mtoinfincond2}&
& \sigma_2^{-1} \sum_{j=1}^m w_j^2 \to 0 
\end{align}
\end{subequations}

Now, write $Z_n(t) = R_n(t) + Y_n(t)$ where
\begin{equation*}
    R_n(t) = \sum_{j=1}^m(M_j(t) - A_j(t))\qquad \textup{and}\qquad Y_n(t)  = \sum_{j>m} (M_j(t)-A_j(t)).
\end{equation*}
We claim
\begin{align}\label{eqn:RNconv}
    R_n(t) \to \sum_{j=1}^\infty \begin{bmatrix}
        \tbeta_j\\
        0
    \end{bmatrix} (\bone_{[\eta_j \le t]} - \beta_jt).
\end{align}
We begin by showing the convergence of the second coordinate. For this, we just need to check for each $\delta>0$ that
\begin{align*}
 \limsup_{n\to\infty} \PR\left(\sup_{s\le t}\left|\sum_{j=1}^m (M_j^{(2)}(s) - A_j^{(2)}(s)) \right| >\delta\right) = 0,
\end{align*} where we use the superscript ``$(2)$'' for the second coordinate.
By the super-martingale inequality \cite[Lemma 2.54.5]{RW.94} we have for all large $n$
\begin{align*}
    \PR&\left(\sup_{s\le t}\left|\sum_{j=1}^m (M_j^{(2)}(s) - A_j^{(2)}(s)) \right| >\delta\right) \le \frac{9}{\delta} \left(\E\left[\sum_{j=1}^m \langle M^{(2)}_j\rangle(t)\right] + \E\left[\sum_{j=1}^m A_j^{(2)}(t)\right]\right)^{1/2}\\
    &\le \frac{9}{\delta} \left(\sum_{j=1}^m   \frac{\eps^2}{\sigma_2^3} w_j  + \sum_{j=1}^m   \frac{\eps}{\sigma_2^3} w_j^2 \right)^{1/2} =o(1).
\end{align*} For the second inequality we used \eqref{eqn:ANodiff}--\eqref{eqn:twedgeExp}. For the limit we used both the limits in \eqref{eqn:mtoinfincond2} and both of the right-most equations in \eqref{eqn:Weigh2}, \eqref{eqn:Weigh3}.

Let us now show that the first coordinate in \eqref{eqn:RNconv} has the desired limit. This is sufficient as the second coordinate is deterministic. For this note that a.s. the jumps times $\eta_j$ are distinct and so for any fixed $k$
\begin{equation*}
    \left(\sum_{j=1}^k\frac{u_j}{\sigma_2} \left(\bone_{[\xi_j\le \sigma_2^{-1}t]} - \sigma_2^{-1} w_jt \right);t\ge 0\right)\weakarrow \left(\sum_{j=1}^k \tbeta_j\left(\bone_{[\eta_j\le t]} - \beta_jt \right);t\ge 0\right).
\end{equation*}
By Theorem 3.2 in \cite{Billingsley.99}, equation \eqref{eqn:RNconv} follows if we can show for all $\delta>0$ and $t>0$
\begin{equation}\label{eqn:TightRNAPPROX}
    \lim_{k\to\infty} \limsup_{n\to\infty} \PR\left(\sup_{s\le t} \left|\sum_{j=k+1}^m \frac{u_j}{\sigma_2} \left(\bone_{[\xi_j\le \sigma_2^{-1}s]} - \sigma_2^{-1} w_js \right)\right|\ge \delta\right) = 0.
\end{equation}
Again, we use the super-martingale inequality to conclude
\begin{align*}
    \limsup_{n\to\infty}\PR&\left(\sup_{s\le t} \left|\sum_{j=k+1}^m \frac{u_j}{\sigma_2} \left(\bone_{[\xi_j\le \sigma_2^{-1}s]} - \sigma_2^{-1} w_js \right)\right|\ge \delta\right) \\
    &\le\limsup_{n\to\infty} \frac{9}{\delta} \left(\E\left[\sum_{j=k+1}^m \langle M^{(1)} \rangle (t)\right] + \E\left[\sum_{j=k+1}^mA_j^{(1)}(t)\right]\right)^{1/2}\\
    &= \limsup_{n\to\infty} \frac{9}{\delta}\left(\sum_{j=k+1}^m (\sigma_2^{-3} u_j^2 w_j) t+ \sum_{j=k+1}^m (\sigma_2^{-3} u_jw_j^2)t^2\right)^{1/2}\\
    & = \frac{9}{\delta} \left(\sum_{j=k+1}^m (\tbeta_j^2 \beta_jt +  \tbeta_j \beta_j^2 t^2)\right)^{1/2}.
\end{align*} The last equality follows from \eqref{eqn:mtoinfincond1}. As $\widetilde{\bbeta},\bbeta\in \ell^3$ we can see that $(\beta_j^2\tbeta_j+\beta_j^1\tbeta_j^2;j\ge 1)$ is summable, the last term is $o(1)$ as $k\to\infty$. We conclude \eqref{eqn:TightRNAPPROX} and, hence, \eqref{eqn:RNconv}.

Next, using \eqref{eqn:twedgeExp},
\begin{align*}
    \E&\left[\langle \langle Y_n \rangle\rangle (t)\right] = \sum_{j>m} \sigma_2^{-3}\begin{bmatrix}
         w_ju_j^2 &   \eps u_jw_j \\
         \eps u_jw_j &   \eps^2 w_j 
    \end{bmatrix} \left(t + O(1) \sigma_2^{-1} w_j t^2\right) \\
    &=  (1+o(1) t)\sum_{j>m} \sigma_2^{-3}\begin{bmatrix}
        w_ju_j^2 &   \eps w_ju_j \\
         \eps w_ju_j &   \eps^2 w_j 
    \end{bmatrix}t
\end{align*} where the $O(1)$ term is uniform in $j>m$ and $n$. In the second line, the $o(1)$ term follows from the uniform bound $\sigma_2^{-1} w_j \le \sigma_2^{-1}(u_m+ w_m) = o(1)$ for all $j>m$. Combining \eqref{eqn:Weigh1}--\eqref{eqn:Weigh3} with \eqref{eqn:mtoinfincond1}--\eqref{eqn:mtoinfincond2} it is easy to see that
\begin{align*}
    &\sigma_2^{-3}\sum_{j>m} w_j u_j^2  \to \varsigma_{1,2} &    & \sigma_2^{-3} \eps \sum_{j>m}   w_ju_j = \frac{\eps}{\sigma_2^2} \frac{1}{\sigma_2}\sum_{j>m} w_ju_j \to \alpha\varsigma_{1,1}\\
    &\sigma_2^{-3}  \eps^2\sum_{j>m}   w_j =  \frac{\eps}{\sigma_2^2} \frac{\eps \sigma_1}{\sigma_2}  \frac{1}{\sigma_1}\sum_{j>m} w_j \to \alpha\gamma.
\end{align*} Similarly, using $t\wedge a \le t$ for all $t\ge 0$ and $a\ge 0$ that for all large $n$
\begin{align*}
    \Var&(\langle Y^{(1)}\rangle(t)) +
    \Var(\langle Y^{(1)},Y^{(2)}\rangle(t)) + \Var(\langle Y^{(2)}\rangle(t)) \\
&\le\left( \sum_{j>m} (\sigma_2^{-6} w_j^2u_j^4 +\eps^2\sigma_2^{-6} w_j^2u_j^2 + \eps^4 \sigma_2^{-6} w_ju_j ) \right)t^2\\
&\le \left(\frac{(u_m+w_m)^3}{\sigma_2^{3}} \frac{\sigma_{2,1}}{\sigma_2^3} + \frac{\eps^2}{\sigma_2^2}\frac{u_m+ w_m}{\sigma_2} \frac{\sigma_3}{\sigma_2^3} + \frac{\eps^4}{\sigma_2^6}\sigma_{1,1}\right)t^2 = o(1).
\end{align*} Combining these last three displayed equations with Chebyshev's inequality, we see
\begin{equation}\label{eqn:MFCLT1}
    \langle\langle Y\rangle\rangle (t) \weakarrow \begin{bmatrix}
        \varsigma_{1,2} & \alpha\varsigma_{1,1}\\
        \alpha \varsigma_{1,1} & \alpha\gamma
    \end{bmatrix} t.
\end{equation}
As $\sum_{j>m}M_j$ is an $\R^2$-valued martingale with continuous quadratic covariation $\langle \langle Y_n\rangle\rangle$ and a.s. distinct jumps of size $o(1)$, we conclude from \eqref{eqn:MFCLT1} and the martingale functional central limit theorem \cite{Whitt.07} that 
\begin{equation}\label{eqn:Mconv}
    \left(\sum_{j>m} M_j(t);t\ge 0\right) \weakarrow \left(B(t);t\ge 0\right)
\end{equation}
where $B$ is a non-standard $\R^2$-valued Brownian motion with $\langle\langle B\rangle\rangle(t) = \begin{bmatrix}
        \varsigma_{1,2} & \alpha\varsigma_{1,1}\\
        \alpha \varsigma_{1,1} & \alpha\gamma
    \end{bmatrix}t.$

We now turn to showing that 
\begin{equation}\label{eqn:Aconv}
\left(    \sum_{j>m} A_j(t) ;t\ge 0 \right) \weakarrow \left( \frac{1}{2}\begin{bmatrix}
    {\varsigma_{2,1}}\\
    {\alpha}
\end{bmatrix}t^2; t\ge 0\right)
\end{equation} This proceeds analogously to the martingale part. First, observe that a.s.
$\frac{d}{dt} (t-\xi_j/q)_+ = \bone_{[\xi_j\le \sigma_2^{-1}t]}$ exists for all rational $t$ and $j\ge 1$, $n\ge 1$. Thus, for $t\in \Q\cap\R_+$,
\begin{align*}
    \E[A_j'(t)] &= \sigma_2^{-2}\begin{bmatrix}
        w_j u_j\\
        \eps w_j
    \end{bmatrix} \PR(\xi_j\le \sigma_2^{-1}t) 
    = \sigma_2^{-3} \begin{bmatrix}
    w_j^2 u_j\\
       \eps   w_j^2
    \end{bmatrix} \left(t + O(1) \sigma_2^{-1} w_j t^2\right) \\&= \sigma_2^{-3}  \begin{bmatrix}
        w_j^2u_j\\
        \eps w_j^2
    \end{bmatrix} (t + o(1))
\end{align*} where the $O(1)$ term is uniform over $j>m$ and $t$ but the last $o(1)$ term is locally uniform in $t$. As we already observed, 
\begin{equation*}
   \sigma_2^{-3} \sum_{j>m}  w_j^2u_j \to \varsigma_{2,1}\qquad\textup{and}\qquad \sigma_2^{-3}\sum_{j>m} \eps  w_j^2 \sim \sigma_2^{-2} \eps\to \alpha.
\end{equation*} Therefore, for all $t\in \Q\cap \R_+$ we have
\begin{equation*}
    \E\left[\sum_{j>m} A_j'(t)\right]  \longrightarrow \begin{bmatrix}
        \varsigma_{2,1}\\
        \alpha
    \end{bmatrix} t.
\end{equation*}
Also writing $\Var^\ast\left(\begin{bmatrix}
    X_1\\X_2
\end{bmatrix}\right) = \begin{bmatrix}
    \Var(X_1)\\ \Var(X_2)
\end{bmatrix}$for the vector of variances, we see that coordinate-wise
\begin{equation*}
    \Var^\ast\left(\sum_{j>m} A_j'(t) \right) \le \sum_{j>m}\sigma_2^{-4}\begin{bmatrix}
        w_j^2 u_j^2 \\
        \eps^2  w_j^2
    \end{bmatrix} \sigma_2^{-1} w_j t \le  \sigma_{2}^{-5}\begin{bmatrix}
        (w_m+u_m)^2 \sigma_{2,1}\\
        \eps^2 \sigma_3
    \end{bmatrix}t = o(1).
\end{equation*} Applying Chebyshev, we see
$ \sum_{j>m} A_j'(t) \weakarrow \begin{bmatrix}
        \varsigma_{2,1}\\\alpha
    \end{bmatrix}t$ for each fixed $t$. Since $\sum_{j>m} A_j'(t)$ has non-decreasing coordinates and the marginal limits are deterministic, Theorem VI.2.15 in \cite{JS.13} implies that in $\D(\R_+,\R^2)$
\begin{equation*}
    \left(\sum_{j>m} A_j'(t);t\ge 0\right) \weakarrow \left( \begin{bmatrix}
        \varsigma_{2,1}\\\alpha 
    \end{bmatrix}t;t\ge 0\right).
\end{equation*} Hence, \eqref{eqn:Aconv} follows from continuity of integration.
The result follows from combining \eqref{eqn:RNconv}, \eqref{eqn:Mconv} and \eqref{eqn:Aconv} and the approximation in \eqref{eqn:Zapprox}.
\end{proof}

\subsection{Consequences of Proposition \ref{prop:generalWeakConv}}

As a consequence of the Lemma \ref{lem:fluctImplyWLLN} we obtain the following corollary of Proposition \ref{prop:generalWeakConv}. We will write $\operatorname{Id} = (t;t\ge 0)$ to simplify this and later results. 
\begin{corollary}\label{cor:WLLNfromFLucForXGen}
   Suppose that $\bw$ satisfies \eqref{eqn:Weigh1}--\eqref{eqn:Weigh3}. Then \begin{equation}\label{eqn:limitXn/alphanbn}
            \left(X_n(t),  F_n(t) ;t\ge 0\right)\weakarrow \left(\operatorname{Id},\gamma  \operatorname{Id}\right).
    \end{equation} 
\end{corollary}

We now turn to a result originally due to Vervaat \cite{Vervaat.72}; however, we will use a version found in Whitt's book \cite{Whitt.02}. Corollary 13.7.3 therein together with the equivalence \cite{Skorohod.56} of $M_1$ convergence and uniform convergence for continuous functions implies the following lemma.
\begin{lemma}\label{lem:fluctuationsOfInverse}
    Let $\psi_n$ be a c\`adl\`ag function and $c_n\to\infty$. If there exists are constants $\mu_n\to \mu>0$ and a continuous function $f$ such that 
    $c_n(\psi_n-\mu_n \operatorname{Id}) \longrightarrow f$ in $J_1$     
    then for $\psi_n^{-1}(t) = \inf\{s: \psi_n(s)>t\}$ it holds
    \begin{equation*}
        c_n \left(\psi_n^{-1} - \frac{1}{\mu_n} \operatorname{Id}\right) \longrightarrow \left(- \frac{1}{\mu} f\left({t}/{\mu}\right);t\ge 0\right)\qquad\textup{ in }J_1.
    \end{equation*}  
\end{lemma} We note that weaker assumptions can be placed on the limit function $f$ and still maintain the same conclusion. See Section 7 of \cite{Whitt.80} and Section 13.7 in \cite{Whitt.02}. 

Let us now observe the fluctuations for 
\begin{equation*}
    \widetilde{J}_n(t) = \min\{u: F_n(u) \ge t\}.
\end{equation*} Note that $\widetilde{J}_n$ is the first hitting time of $F_n$, while $\widetilde{J}_n(t+\sigma_2^{2}) = \inf\{t: F_n(t)>t\}$ is the first passage time of $F_n$. Thus Proposition \ref{prop:generalWeakConv} and Lemma \ref{lem:fluctuationsOfInverse} imply the following lemma.
\begin{lemma}\label{lem:Nnsinverse}
    Suppose $(\bw,\bu) = (\bw^{(n)},\bu^{(n)})_{n\ge 1}$ satisfies \eqref{eqn:Weigh1}--\eqref{eqn:Weigh3}. Then, jointly with the convergence in \eqref{eqn:XFlim},
    \begin{align*}
       & \left(\sigma_2^{-1}\left( \widetilde{J}_n(t) - \frac{1}{\gamma_n} t\right)\right)_{t\ge 0}\weakarrow -\frac{1}{\gamma}\left( B_2(t/\gamma) - \frac{\alpha}{2\gamma^2} t^2 + \frac{\lambda}{\gamma}  t\right)_{t\ge 0}
    \end{align*}
\end{lemma}

\subsection{Proof of Theorem \ref{thm:SIZEBIASPARTIALSUMS}}\label{ssec:proofOfPartial}

\begin{proof}[Proof of Theorem \ref{thm:SIZEBIASPARTIALSUMS}]  Let
\begin{equation*}
    Y_n(k) = \sum_{j=1}^k u_{(j)}\qquad\textup{where}\qquad u_{(j)} = u_{\pi_j}.
\end{equation*} and let $J_n(k) = \inf\{t: \sum_{j=1}^N \bone_{[\xi_j\le \sigma_2^{-1}t]} \ge k\}$. Note that by elementary properties of exponential random variables
\begin{equation*}
    (Y_n(k);k\ge 0)  \overset{d}{=}( X_n\circ J_n(k);k\ge 0).
\end{equation*} Therefore, defining $b_n = \eps_n^{-1}$, 
\begin{equation}\label{eqn:Yndecomp}
    (Y_n(b_n t) -  \frac{\sigma_{1,1}}{\sigma_2\gamma_n}t)_{t\ge 0}\overset{d}{=}\left(X_n\circ J_n(b_n t) - \frac{\sigma_{1,1}}{\sigma_2} J_n(b_n t )+ \frac{\sigma_{1,1}}{\sigma_2}(J_n(b_n t) -  \frac{1}{\gamma_n} t) \right)_{t\ge 0}.
\end{equation}
It is easy to check that $J_n(b_n t) = \widetilde{J}_n (t)$. 
We can apply Lemma \ref{lem:Nnsinverse} to conclude
\begin{align}
   \label{eqn:Jconv22}&\left(\sigma_2^{-1}  \left(J_n(b_n t) -  t\right)\right)_{t\ge 0} \weakarrow\left(-\frac{1}{\gamma}\left( B_2(t/\gamma) - \frac{\alpha}{2\gamma^2} t^2 + \frac{\lambda}{\gamma}  t\right)\right)_{t\ge 0}.
\end{align} 

From Lemma \ref{lem:fluctImplyWLLN}, we see that $\left(J_n(b_nt)\right)_{t\ge 0} \to  \gamma^{-1}\operatorname{Id}. $ applying Proposition \ref{prop:generalWeakConv} we have
\begin{align*}
   & \sigma_2^{-1} \left(X_n\circ J_n(b_nt ) - J_n(b_nt)\right)_{t\ge 0}\weakarrow \bigg(\sum_{j=1}^\infty \tbeta_j \left(\bone_{[\eta_j\le  t/\gamma]} - \frac{\beta_j}{\gamma} t\right) + B_1( t/\gamma) - \frac{\varsigma_{2,1}}{2\gamma^2}t^2 \bigg)_{t\ge 0}
\end{align*}
jointly with the convergence in \eqref{eqn:Jconv22}.
Using \eqref{eqn:Yndecomp}, and the fact that the limits in \eqref{eqn:Jconv22} are continuous, we have
\begin{align*}
    \left(\sigma_2^{-1}(Y_n(b_n t) - \frac{\sigma_{1,1}}{\sigma_2\gamma_n}t)\right)_{t\ge 0}
    &\weakarrow \bigg(\sum_{j=1}^\infty \tbeta_j \left(\bone_{[\eta_j\le  t/\gamma]} - \frac{\beta_j}{\gamma} t\right) + B_1( t/\gamma) - \frac{\varsigma_{2,1}}{2\gamma^2}t^2  \\
    &\qquad\qquad -\frac{\varsigma_{1,1}}{\gamma}\left( B_2(t/\gamma) - \frac{\alpha}{2\gamma^2} t^2 + \frac{\lambda}{\gamma}  t\right)\bigg)_{t\ge 0}\\
    &\overset{d}{=} \bigg(\sum_{j=1}^\infty \tbeta_j \left(\bone_{[\eta_j\le  t/\gamma]} - \frac{\beta_j}{\gamma} t\right) + \sqrt{\kappa} W(t) - \frac{\rho}{2}t^2  -\mu t\bigg)_{t\ge 0}
\end{align*}
as desired.
\end{proof}

We will need an additional lemma in some subsequent analysis. Recall that $\xi_j$ are the random times constructed in Proposition \ref{prop:generalWeakConv}. Also recall the definition of the size-biased permutation from \eqref{eqn:piSizeBiasDef}.
\begin{lemma}\label{lem:JointConvForJumps}
   Suppose that $(\bw,\bu)$ satisfies \eqref{eqn:Weigh1}--\eqref{eqn:Weigh3}. Fix a $K$ such that $\beta_K>0$. Then for all $t_1,\dotsm, t_K \ge 0$
   \begin{equation*}
       \lim_{n\to\infty} \PR( j\in \{\pi_1,\dotsm, \pi_{\lfloor b_n t_j\rfloor }\} \textup{ for all }j\le K)  = \PR\left(\eta_j \le t_j/\gamma\textup{ for all }j\le K\right).
   \end{equation*}
\end{lemma}
\begin{proof}
    First note that from the previous coupling that
    \begin{equation*}
        j\in\{\pi_1,\dotsm, \pi_{\lfloor b_n t_j\rfloor}\} \qquad\textup{ if and only if }\qquad \xi_j \le \sigma_2^{-1}{J}_n(b_nt_j).
    \end{equation*}
    As $J_n(b_n\bullet) \weakarrow \gamma^{-1} \operatorname{Id}$ and $\sigma_2 \xi_j \weakarrow \eta_j$ the result easily follows.    
\end{proof}

\subsection{Applications of Theorem \ref{thm:SIZEBIASPARTIALSUMS}}

We now turn to several key applications of Theorem \ref{thm:SIZEBIASPARTIALSUMS} involving degree sequences. Given two degree sequences $\bd^\ml = \bd^{\ml,(n)}$ and $\bd^\mr = \bd^{\mr,(n)}$, we write
\begin{equation*}
    Y^\ml(k) = \sum_{j=1}^k (d_{(j)}^\ml - 1)\qquad \textup{and}\qquad Y^\mr(k) = \sum_{j=1}^k (d_{(j)}^\mr - 1)
\end{equation*}
where $(d_{(j)}^\ml;j\in[n])$, $(d^{\mr}_{(j)};j\in[m])$ are size-biased permutations of the respective degree sequences $\bd^\ml, \bd^\mr$. We will assume that the two walks are independent.

\subsubsection{Under Assumption \ref{ass:asymptProp}}\label{ssec:lightproof}

\begin{proof}[Proof of Lemma \ref{lem:YnmlAndYnmr}] The joint convergence follows from the marginal convergence as $Y_n^\ml$ and $Y_n^\mr$ are independent of each other. Let us first establish the limit for $Y^\ml_n$. Write $b_n = n^{2/3}$, $a_n = n^{1/3}$, and $\eps = \eps_n = n^{-2/3}$. Set $w_j= u_j = \eps d_j^\ml$. Then observe for $a+b = p\in\{1,2,3\}$ we have
\begin{align*}
    \sigma_{p} = \sigma_{a,b} = \eps^{p} n \E[D_n^\ml] \sim n^{1-\frac23 p} \mu_{p}^\ml.
\end{align*} 
Now one can check that \eqref{eqn:Weigh1}--\eqref{eqn:Weigh3} with $\bbeta = \widetilde{\bbeta} = \bzer$, and 
\begin{align*}
    \varsigma_{a,b} &= \frac{\mu_{3}^\ml}{(\mu_2^\ml)^3}\qquad \forall a+b = 3, &\varsigma_{a,b}&=1\qquad\forall a+b = 2\\
    \gamma_n&= \frac{\E[D_n^\ml]}{\E[(D_n^\ml)^2]}, & \gamma &= \frac{\mu_1^\ml}{\mu_2^\ml} & \alpha&= \frac{1}{(\mu_2^\ml)^2}.
\end{align*} Using the notation Theorem \ref{thm:SIZEBIASPARTIALSUMS} we see that
\begin{align*}
    \kappa &= \frac{\varsigma_{1,2}\gamma-\alpha \varsigma_{1,1}}{\gamma^{2}} = \frac{\mu_{1}^\ml \mu_3^\ml - (\mu_2^\ml)^2}{(\mu_1^\ml \mu_2^\ml)^2}, &
    \rho&= \frac{\varsigma_{2,1}\gamma - \varsigma_{1,1}\alpha}{\gamma^3} = \frac{\mu_{1}^\ml \mu_3^\ml - (\mu_2^\ml)^2}{(\mu_1^\ml)^3 \mu_2^\ml} 
\end{align*}
Thus, using Theorem \ref{thm:SIZEBIASPARTIALSUMS}, 
\begin{align*}
&\left(n^{-1/3}(Y_n^\ml(n^{2/3}t) - n^{2/3} \frac{\E[D_n^\ml(D_n^\ml-1)]}{\E[D_n^\ml]} t\right)_{t\ge 0}
\\
   &=\E[(D_n^\ml)^2]\left(\frac{n^{1/3}}{\E[(D_n^\ml)^2]}\left( \sum_{j=1}^{n^{2/3} t}   n^{-2/3}d_{(j)} - \frac{\E[(D_n^\ml)^2]}{\E[D_n^\ml]}t\right)\right)_{t\ge 0}\\
   &\qquad \weakarrow \mu_2^\ml \left(\frac{\sqrt{\mu_1^\ml\mu_3^\ml-(\mu_2^\ml)^2}}{\mu_1^\ml\mu_2^\ml} W(t) - \frac{1}{2} \frac{\mu_{1}^\ml \mu_3^\ml - (\mu_2^\ml)^2}{(\mu_1^\ml)^3 \mu_2^\ml} t^2\right)_{t\ge 0}\\
   &\qquad\qquad =\left(\frac{\sqrt{\mu_1^\ml\mu_3^\ml-(\mu_2^\ml)^2}}{\mu_1^\ml} W(t) - \frac{1}{2} \frac{\mu_{1}^\ml \mu_3^\ml - (\mu_2^\ml)^2}{(\mu_1^\ml)^3 } t^2\right)_{t\ge 0}.
\end{align*}

By similar logic, we see that
\begin{align*}
    &m^{-1/3}\left(Y_n^\mr(m^{2/3}t) - m^{2/3}\frac{\E[D_n^\mr(D_n^\mr-1)]}{\E[D_n^\mr]} t\right)_{t\ge 0} \\
    &\quad\weakarrow  \left(\frac{\sqrt{\mu_1^\mr\mu_3^\mr-(\mu_2^\mr)^2}}{\mu_1^\mr} W(t) - \frac{1}{2} \frac{\mu_{1}^\mr \mu_3^\mr - (\mu_2^\mr)^2}{(\mu_1^\mr)^3 } t^2\right),_{t\ge 0}
\end{align*} for a standard Brownian motion $W$.
As $\theta^{-2/3}m^{2/3} \sim n^{2/3}$ and $\theta^{1/3}m^{-1/3}\sim  n^{-1/3}$, we see
\begin{align*}
 &n^{-1/3}\left(Y_n^\mr(n^{2/3}t) - n^{2/3}\frac{\E[D_n^\mr(D_n^\mr-1)]}{\E[D_n^\mr]} t\right)_{t\ge 0}\\
    &=\frac{n^{-1/3}}{m^{-1/3}} m^{-1/3}\left(Y_n^\mr\left(\frac{n^{2/3}}{m^{2/3}} m^{2/3}t\right) - m^{2/3}\frac{\E[D_n^\mr(D_n^\mr-1)]}{\E[D_n^\mr]} \frac{n^{2/3}}{m^{2/3}} t\right)_{t\ge 0}\\
    &\quad\weakarrow\theta^{1/3}  \left(\frac{\sqrt{\mu_1^\mr\mu_3^\mr-(\mu_2^\mr)^2}}{\mu_1^\mr} W(\theta^{-2/3}t) - \frac{1}{2} \frac{\mu_{1}^\mr \mu_3^\mr - (\mu_2^\mr)^2}{(\mu_1^\mr)^3 } {\theta^{-4/3}} t^2\right)_{t\ge 0}.
\end{align*}The result now follows from Brownian scaling.
\end{proof}

\subsubsection{Under Assumption \ref{ass:infintieThird}}

\begin{proof}[Proof of Lemma \ref{lem:YconvInfinitThird}]
    We prove the limit for $Y_n^\ml$. The limit for $Y_n^\mr$ follows exactly the same reasoning and the following (easy to check) asymptotics:
    \begin{equation*}
        \frac{a_m}{a_n}\to \theta^{1/(\tau-1)}\qquad\textup{and}\qquad  \frac{b_n}{b_m} \to \theta^{-(\tau-2)/(\tau-1)}. 
    \end{equation*}

    Let us set $w_j = u_j = \frac{1}{\mu_2^\ml b_n}d_j^\ml$ and set $\eps = b_n^{-1}$. Note that $a_nb_n = n$, $b_n/a_n = c_n$ and so $n/b_n^2 = c_n^{-1}$. Also note that $c_n^2/b_n = o(1)$. It is easy to see,     \begin{equation*}
        \sigma_1  \sim a_n \frac{\mu_1^\ml}{\mu_2^\ml}, \qquad \sigma_2 \sim c_n^{-1} \frac{1}{\mu_2^\ml}, \qquad \sigma_{3} = \frac{1}{b_n^3 (\mu_2^\ml)^3}\sum_{j=1}^n (d_j^\ml)^3.
    \end{equation*}
    Thus,
    \begin{equation*}
        \frac{w_j}{\sigma_2} \sim \frac{d_j^\ml}{a_n} \to \beta_j^\ml\qquad\textup{and}\qquad   \frac{\sigma_3}{\sigma_2^3}\sim \frac{c_n^3}{b_n^3} \sum_{j=1}^n (d_j^\ml)^3 \to \sum_{j=1}^\infty (\beta^\ml_j)^3.
    \end{equation*} Thus \eqref{eqn:Weigh1}--\eqref{eqn:Weigh3} hold with
    \begin{align*}
        \varsigma_{a,b} &= 0,\quad \forall a+b = 3;& \varsigma_{a,b}&= 1,\quad \forall a+b = 2; & \alpha&=0\\
        \bbeta& = \widetilde{\bbeta} = \bbeta^\ml&  \gamma_n &= \frac{\sigma_1\eps}{\sigma_2} = \frac{\E[D_n^\ml]}{\E[(D_n^\ml)^2]} & \gamma &=\mu_1^\ml.
     \end{align*}
     Applying Theorem \ref{thm:SIZEBIASPARTIALSUMS} we see
     \begin{align*}
        &a_n^{-1} \left(Y_n^\ml(b_nt) - b_n \frac{\E[D_n^\ml(D_n^\ml-1)]}{\E[D_n^\ml]} t\right)_{t\ge 0}= \mu_2^\ml c_n \left(\sum_{j=1}^{b_nt} \frac{d_{(j)}^\ml}{b_n\mu_2^\ml}  - \frac{\E[(D_n^\ml)^2]}{\E[D_n^\ml]} t\right)_{t\ge 0}\\&\weakarrow \left(\sum_{j=1}^\infty \beta_j^\ml \left(\bone_{[\eta^\ml_j\le t/\mu_1^\ml]} - {\beta_j^\ml} t/{\mu_1^\ml}\right)\right)_{t\ge 0}.
     \end{align*}
\end{proof}

We now turn to some useful lemmas for related to the triangle counts. For this we need a preliminary lemma, its proof is an elementary consequence of Assumption \ref{ass:infintieThird}\textbf{(iii)}.
\begin{claim}\label{Claim:HeavyMoments}
    Let $\bd^\ml$ and $\bd^\mr$ be as in Assumption \ref{ass:infintieThird}. Then for $p\ge 3$
    \begin{equation*}
        \lim_{n\to\infty} \frac{n\E[(D_n^{\ml})^p]}{a_n^p} = \|\bbeta^\ml\|_p^p \qquad\textup{and}\qquad \lim_{n\to\infty} \frac{m\E[(D_n^{\ml})^p]}{a_m^p}  = \|\bbeta^\mr\|_p^p.
    \end{equation*}
    In particular, for all $p\ge 0, q\ge 1$
    \begin{equation*}
    \E\left[(D_n^\ml)^p\binom{D_n^\ml}{3}^q\right] \sim \frac{1}{6^q}\E[(D_n^\ml)^{p+3q}]\qquad\textup{and}\qquad \E\left[(D_n^\mr)^p\binom{D_n^\mr}{3}^q\right] \sim \frac{1}{6^q}\E[(D_n^\mr)^{p+3q}].
\end{equation*}
\end{claim}

Using the above claim, we can establish Lemmas \ref{lem:TriangleInfnite}.
\begin{proof}[Proof of Lemma \ref{lem:TriangleInfnite}]
Obtaining the joint convergence with \eqref{eqn:YnmlLimitInfin}--\eqref{eqn:YnmrLimitInfin} from the marginal convergence follows easily using Lemma \ref{lem:JointConvForJumps}. We leave these details to the reader and only prove the marginal convergence for the $\ml$-degrees. The marginal convergence for the $\mr$-vertices is and follows from the observation that if $\frac{1}{a_m^p} f_n(b_m t)\to f(t)$ then $\frac{1}{a_n^p} f_n(b_nt) \to\theta^{-p/(\tau-1)} f(\theta^{-\frac{\tau-2}{\tau-1}} t)$.

    Let us apply Theorem \ref{thm:SIZEBIASPARTIALSUMS} with $w_j = \frac{1}{\mu_2^\ml b_n} d_j^\ml$, $u_j = \frac{1}{\mu_2^\ml na_n} \binom{d_j^\ml}{3}$, $\eps = b_n^{-1}$. Note that ${na_n} = {b_n a_n^2} = {c_n a_n^3}$. We can easily see that
    \begin{align*}
        \sigma_1 \sim a_n \frac{\mu_1^\ml}{\mu_2^\ml} \qquad \sigma_2 \sim \frac{1}{c_n\mu_2^\ml}\qquad\textup{and} \qquad   \frac{(w_j,u_j)}{\sigma_2} \to \left(\beta_j, \frac{1}{6} \beta_j^3\right)=: (\beta_j,\tbeta_j).
    \end{align*}
   
    We can also derive using Claim \ref{Claim:HeavyMoments}
    \begin{equation*}
        (\mu_2^\ml)^3\sigma_{1,2} = \frac{n}{(na_n)^2b_n}\E\left[D_n^\ml \binom{D_n^\ml}{3}^2\right] \sim \frac{\|\bbeta\|^7_7 a_n^7}{36n^2a_n^2 b_n}  = \frac{\|\bbeta\|_7^7}{36} \frac{a_n^6}{n^3} = \frac{\|\bbeta\|_7^7}{36} c_n^{-3}.
    \end{equation*}
    Therefore,
    \begin{align*}
        \frac{\sigma_{1,2}}{\sigma_{2}^3} \to \frac{1}{36}  \|\bbeta\|_7^7 = \sum_{j} \beta_j \widetilde{\beta}_j^2.
    \end{align*} One can similarly check the remaining identities in \eqref{eqn:Weigh1} hold for $\varsigma_{a,b} = 0$ for $a+b =3$.
    Note that
    \begin{align*}
        \sigma_{1,1} = \frac{n}{na_n b_n (\mu_2^\ml)^2} \E\left[D_n^\ml \binom{D_n^\ml}{3}\right] \sim \frac{a_n^4}{6n^2} \|\bbeta\|_4^4 = c_n^{-2} \|\bbeta\|_4^4  = o(\sigma_2).
    \end{align*}Therefore, left-hand side of \eqref{eqn:Weigh3} $\varsigma_{1,1} = \varsigma_{0,2} = 0$. The remain assumptions in Theorem \ref{thm:SIZEBIASPARTIALSUMS} were checked in the proof of Lemma \ref{lem:YconvInfinitThird}. 
    
    Therefore, by Theorem \ref{thm:SIZEBIASPARTIALSUMS}, 
    \begin{align*}
  &\frac{1}{\sigma_2} \left(   \sum_{j=1}^{b_nt} \frac{1}{\mu_{2}^\ml na_n}\binom{d_{(j)}^\ml}{3} 
 - \frac{\sigma_{1,1}}{\sigma_1 \eps}t\right)_{t\ge 0}\weakarrow \left(\frac{1}{6}\sum_{j=1}^\infty\beta_j^3 (\bone_{[\eta_j\le t/\mu_1^\ml]} - \beta_j t/\mu_1^\ml) \right)_{t\ge 0}.
    \end{align*} Note that
    \begin{equation*}
        \frac{\sigma_{1,1}}{\sigma_2 \sigma_1 \eps} \sim \frac{c_n^{-2}(\mu_2^\ml)^{-2} \|\bbeta\|^4_4/6}{c_n^{-1}(\mu_2^\ml)^{-1} \cdot a_n \frac{\mu_1^\ml}{\mu_2^\ml} \cdot b_n^{-1}}  = \frac{1}{6 \mu_1}\|\bbeta\|_4^4.
    \end{equation*}
   As $\bbeta\in \ell^3\subset\ell^4$, we see
   \begin{align*}
  &\frac{1}{\sigma_2} \left(   \sum_{j=1}^{b_nt} \frac{1}{\mu_{2}^\ml na_n}\binom{d_{(j)}^\ml}{3} 
\right)_{t\ge 0}\weakarrow \left(\frac{1}{6}\sum_{j=1}^\infty\beta_j^3 \bone_{[\eta_j\le t/\mu_1^\ml]} \right)_{t\ge 0}.
   \end{align*}
\end{proof}

The proof of Lemma \ref{lem:secondMOme} is similar, and one can check this by choosing $w_j = \frac{1}{\mu_2^\ml b_n} d_j^\ml$ and $u_j = \frac{1}{\mu_n^\ml n} d_j^\ml(d_j^\ml-1)$. We leave the details to the reader.

\bibliography{ref, references}

\providecommand{\bysame}{\leavevmode\hbox to3em{\hrulefill}\thinspace}
\providecommand{\MR}{\relax\ifhmode\unskip\space\fi MR }
\providecommand{\MRhref}[2]{%
  \href{http://www.ams.org/mathscinet-getitem?mr=#1}{#2}
}
\providecommand{\href}[2]{#2}
\begin{thebibliography}{10}

\bibitem{Aldous.97}
David Aldous, \emph{Brownian excursions, critical random graphs and the multiplicative coalescent}, Ann. Probab. \textbf{25} (1997), no.~2, 812--854. \MR{1434128}

\bibitem{AL.98}
David Aldous and Vlada Limic, \emph{The entrance boundary of the multiplicative coalescent}, Electron. J. Probab. \textbf{3} (1998), no. 3, 59. \MR{1491528}

\bibitem{BST.14}
Frank~G. Ball, David~J. Sirl, and Pieter Trapman, \emph{Epidemics on random intersection graphs}, Ann. Appl. Probab. \textbf{24} (2014), no.~3, 1081--1128. \MR{3199981}

\bibitem{Behrisch.07}
Michael Behrisch, \emph{Component evolution in random intersection graphs}, The Electronic Journal of Combinatorics \textbf{14} (2007), no.~1, R17.

\bibitem{BSW.17}
Shankar Bhamidi, Sanchayan Sen, and Xuan Wang, \emph{Continuum limit of critical inhomogeneous random graphs}, Probab. Theory Related Fields \textbf{169} (2017), no.~1-2, 565--641. \MR{3704776}

\bibitem{BvdHvL.10}
Shankar Bhamidi, Remco van~der Hofstad, and Johan S.~H. van Leeuwaarden, \emph{Scaling limits for critical inhomogeneous random graphs with finite third moments}, Electron. J. Probab. \textbf{15} (2010), no. 54, 1682--1703. \MR{2735378}

\bibitem{Billingsley.99}
Patrick Billingsley, \emph{Convergence of probability measures}, second ed., Wiley Series in Probability and Statistics: Probability and Statistics, John Wiley \& Sons, Inc., New York, 1999, A Wiley-Interscience Publication. \MR{1700749}

\bibitem{Bloznelis.10}
Mindaugas Bloznelis, \emph{The largest component in an inhomogeneous random intersection graph with clustering}, Electron. J. Combin. \textbf{17} (2010), no.~1, Research Papr 110, 17. \MR{2679564}

\bibitem{Bloznelis.13}
\bysame, \emph{Degree and clustering coefficient in sparse random intersection graphs}, Ann. Appl. Probab. \textbf{23} (2013), no.~3, 1254--1289. \MR{3076684}

\bibitem{Bloznelis.17}
\bysame, \emph{Degree-degree distribution in a power law random intersection graph with clustering}, Internet Math. (2017), 25. \MR{3683432}

\bibitem{BJR.07}
B\'{e}la Bollob\'{a}s, Svante Janson, and Oliver Riordan, \emph{The phase transition in inhomogeneous random graphs}, Random Structures Algorithms \textbf{31} (2007), no.~1, 3--122. \MR{2337396}

\bibitem{BR.12a}
B{\'e}la Bollob{\'a}s and Oliver Riordan, \emph{Asymptotic normality of the size of the giant component in a random hypergraph}, Random Structures \& Algorithms \textbf{41} (2012), no.~4, 441--450.

\bibitem{BDW.22}
Nicolas Broutin, Thomas Duquesne, and Minmin Wang, \emph{Limits of multiplicative inhomogeneous random graphs and l{\'e}vy trees: The continuum graphs}, The Annals of Applied Probability \textbf{32} (2022), no.~4, 2448--2503.

\bibitem{Chodrow.20}
Philip~S Chodrow, \emph{Configuration models of random hypergraphs}, Journal of Complex Networks \textbf{8} (2020), no.~3, cnaa018.

\bibitem{Clancy.24}
{David Clancy, Jr.}, \emph{Inhomogeneous percolation on the hierarchical configuration model with a heavy-tailed degree distribution}, 2024.

\bibitem{DvdHvLS.20}
Souvik Dhara, Remco van~der Hofstad, Johan S.~H. van Leeuwaarden, and Sanchayan Sen, \emph{Heavy-tailed configuration models at criticality}, Ann. Inst. Henri Poincar\'{e} Probab. Stat. \textbf{56} (2020), no.~3, 1515--1558. \MR{4116701}

\bibitem{DvdHvLS.17}
Souvik Dhara, Remco van~der Hofstad, Johan~SH Van~Leeuwaarden, and Sanchayan Sen, \emph{Critical window for the configuration model: finite third moment degrees}, Electronic Journal of Probability \textbf{22} (2017), 1--33.

\bibitem{Federico.19}
Lorenzo Federico, \emph{Critical scaling limits of the random intersection graph}, arXiv preprint arXiv:1910.13227 (2019).

\bibitem{JS.13}
Jean Jacod and Albert Shiryaev, \emph{Limit theorems for stochastic processes}, vol. 288, Springer Science \& Business Media, 2013.

\bibitem{Janson.09}
Svante Janson, \emph{On percolation in random graphs with given vertex degrees}, Electron. J. Probab. \textbf{14} (2009), no. 5, 87--118. \MR{2471661}

\bibitem{Janson.10}
\bysame, \emph{Susceptibility of random graphs with given vertex degrees}, J. Comb. \textbf{1} (2010), no.~3-4, 357--387. \MR{2799217}

\bibitem{JR.12}
Svante Janson and Oliver Riordan, \emph{Susceptibility in inhomogeneous random graphs}, Electron. J. Combin. \textbf{19} (2012), no.~1, Paper 31, 59. \MR{2880662}

\bibitem{Joseph.14}
Adrien Joseph, \emph{The component sizes of a critical random graph with given degree sequence}, Ann. Appl. Probab. \textbf{24} (2014), no.~6, 2560--2594. \MR{3262511}

\bibitem{Limic.19}
Vlada Limic, \emph{The eternal multiplicative coalescent encoding via excursions of {L}\'{e}vy-type processes}, Bernoulli \textbf{25} (2019), no.~4A, 2479--2507. \MR{4003555}

\bibitem{Limic.19_Supplement}
\bysame, \emph{Supplement to ``{The eternal multiplicative coalescent encoding via excursions of {L}\'{e}vy-type processes}''}, Bernoulli \textbf{25} (2019), no.~4A.

\bibitem{Riordan.12}
Oliver Riordan, \emph{The phase transition in the configuration model}, Combin. Probab. Comput. \textbf{21} (2012), no.~1-2, 265--299. \MR{2900063}

\bibitem{RW.94}
L.~C.~G. Rogers and David Williams, \emph{Diffusions, {M}arkov processes, and martingales. {V}ol. 1}, second ed., Wiley Series in Probability and Mathematical Statistics: Probability and Mathematical Statistics, John Wiley \& Sons, Ltd., Chichester, 1994, Foundations. \MR{1331599}

\bibitem{Skorohod.56}
A.~V. Skorohod, \emph{Limit theorems for stochastic processes}, Teor. Veroyatnost. i Primenen. \textbf{1} (1956), 289--319. \MR{84897}

\bibitem{vanderHofstad.24}
Remco van~der Hofstad, \emph{Random graphs and complex networks. {V}ol. 2}, Cambridge Series in Statistical and Probabilistic Mathematics, Cambridge University Press, 2024.

\bibitem{vdHKV.21}
Remco Van Der~Hofstad, J{\'u}lia Komj{\'a}thy, and Vikt{\'o}ria Vadon, \emph{Random intersection graphs with communities}, Advances in Applied Probability \textbf{53} (2021), no.~4, 1061--1089.

\bibitem{vdHKV.22}
Remco van~der Hofstad, J{\'u}lia Komj{\'a}thy, and Vikt{\'o}ria Vadon, \emph{Phase transition in random intersection graphs with communities}, Random Structures \& Algorithms \textbf{60} (2022), no.~3, 406--461.

\bibitem{Vervaat.72}
Wim Vervaat, \emph{Functional central limit theorems for processes with positive drift and their inverses}, Z. Wahrscheinlichkeitstheorie und Verw. Gebiete \textbf{23} (1972), 245--253. \MR{321164}

\bibitem{Villani.09}
C\'{e}dric Villani, \emph{Optimal transport}, Grundlehren der mathematischen Wissenschaften [Fundamental Principles of Mathematical Sciences], vol. 338, Springer-Verlag, Berlin, 2009, Old and new. \MR{2459454}

\bibitem{Wang.23}
Minmin Wang, \emph{Large random intersection graphs inside the critical window and triangle counts}, arXiv preprint arXiv:2309.13694 (2023).

\bibitem{Whitt.80}
Ward Whitt, \emph{Some useful functions for functional limit theorems}, Mathematics of operations research \textbf{5} (1980), no.~1, 67--85.

\bibitem{Whitt.02}
\bysame, \emph{Stochastic-process limits}, Springer Series in Operations Research, Springer-Verlag, New York, 2002, An introduction to stochastic-process limits and their application to queues. \MR{1876437}

\bibitem{Whitt.07}
\bysame, \emph{Proofs of the martingale {FCLT}}, Probab. Surv. \textbf{4} (2007), 268--302. \MR{2368952}

\bibitem{Wu.08}
Biao Wu, \emph{On the weak convergence of subordinated systems}, Statistics \& probability letters \textbf{78} (2008), no.~18, 3203--3211.

\end{thebibliography}
\bibliographystyle{amsplain}

\end{document}